\documentclass[12pt,reqno]{amsart}
\usepackage{amsmath,amsthm,amsfonts,amssymb,eucal, hyperref}
\usepackage{graphicx}

\newcommand{\e}{\epsilon}

\renewcommand{\l}{\lambda}

\newcommand{\z}{\zeta}

\newcommand{\T}{\Theta}

\newcommand{\oq}{\ {\raise 7pt\hbox{${\scriptstyle\circ}$}}
\kern -7pt{
\hbox{$Q$}}}

\newcommand{\R}{ \mathbb R}
\newcommand{\N}{ \mathbb N}

\newcommand{\Rd}{ \mathbb R^d}

\newcommand {\GD}{\mathfrak D}

\newcommand {\GH}{\mathfrak H}

\newcommand {\GW}{\mathfrak W}
\newcommand {\GV}{\mathfrak V}
\newcommand {\GU}{\mathfrak U}

\newcommand {\ba}{\mathbf a}

\newcommand {\bb}{\mathbf b}

\newcommand {\BB}{\mathbf B}
\newcommand {\BD}{\mathbf D}
\newcommand {\BG}{\mathbf G}

\newcommand {\BF}{\mathbf F}
\newcommand {\BM}{\mathbf M}

\newcommand {\bx}{\mathbf x}

\newcommand {\be}{\mathbf e}

\newcommand {\bk}{\mathbf k}

\newcommand {\bm}{\mathbf m}

\newcommand {\bt}{\mathbf t}

\newcommand {\bn}{\mathbf n}
\newcommand {\bnu}{\boldsymbol\nu}

\newcommand {\bal}{\boldsymbol\alpha}

\newcommand {\bmu}{\boldsymbol\mu}
\newcommand {\bth}{\boldsymbol\theta}
\newcommand {\boldeta}{\boldsymbol\eta}

\newcommand {\bxi}{\boldsymbol\xi}

\newcommand {\BUps}{\boldsymbol\Upsilon}
\newcommand {\BPsi}{\boldsymbol\Psi}

\newcommand{\lu}{\langle}
\newcommand{\ru}{\rangle}

\newcommand{\CV}{\mathcal V}

\newcommand{\CB}{\mathcal B}

\newcommand{\CX}{\mathcal X}
\newcommand{\CH}{\mathcal H}

\newcommand{\CP}{\mathcal P}

\newcommand{\CA}{\mathcal A}

\newcommand{\CC}{\mathcal C}

\newcommand{\CD}{\mathcal D}

\newcommand{\1}
{{\,\vrule depth3pt height9pt}{\vrule depth3pt height9pt}
{\vrule depth3pt height9pt}{\vrule depth3pt height9pt}\,}

\DeclareMathOperator \volume {{vol}}

\newtheorem{thm}{Theorem}[section]
\newtheorem{cor}[thm]{Corollary}
\newtheorem{lem}[thm]{Lemma}

\theoremstyle{definition}

\theoremstyle{remark}
\newtheorem{rem}[thm]{Remark}

\numberwithin{equation}{section}

\newcommand{\bee}{\begin{equation}}
\newcommand{\ene}{\end{equation}}
\newcommand{\bees}{\begin{equation*}}
\newcommand{\enes}{\end{equation*}}
\newcommand{\bes}{\begin{split}}
\newcommand{\ens}{\end{split}}

\newcommand{\bet}{\begin{thm}}
\newcommand{\ent}{\end{thm}}
\newcommand{\bel}{\begin{lem}}
\newcommand{\enl}{\end{lem}}
\newcommand{\bec}{\begin{cor}}
\newcommand{\enc}{\end{cor}}
\newcommand{\bep}{\begin{proof}}
\newcommand{\enp}{\end{proof}}
\newcommand{\ber}{\begin{rem}}
\newcommand{\enr}{\end{rem}}
\newcommand{\ep}{\varepsilon}
\newcommand{\la}{\lambda}
\newcommand{\de}{\delta}
\newcommand{\al}{\alpha}
\newcommand{\Z}{\mathbb Z}
\newcommand{\Ga}{\Gamma}

\newcommand{\De}{\Delta}

\makeatletter
\def\square{\RIfM@\bgroup\else$\bgroup\aftergroup$\fi
  \vcenter{\hrule\hbox{\vrule\@height.6em\kern.6em\vrule}\hrule}\egroup}
\makeatother

\begin{document}

\title[Bethe-Sommerfeld conjecture
(\the\day.\the\month.\the\year)]
{Bethe-Sommerfeld conjecture}
\author[L.Parnovski (\the\day.\the\month.\the\year)]{Leonid Parnovski}
\address{Department of Mathematics\\ University College London\\
Gower Street\\ London\\ WC1E 6BT UK}
\email{Leonid@math.ucl.ac.uk}
\begin{abstract}
We consider Schr\"odinger operator $-\Delta+V$ in $\R^d$ ($d\ge 2$) with smooth
periodic potential $V$ and prove that there are only finitely many gaps in its spectrum.
\end{abstract}

\maketitle

\

\centerline{Dedicated to the memory of B.M.Levitan}

\

\section{Introduction}
This paper is devoted to proving the Bethe-Sommerfeld conjecture which
states that number of gaps in the spectrum of a Schr\"odinger operator
\bee\label{Schroedinger}
-\De+V(\bx), \qquad \bx\in\R^d
\ene
with a 
periodic potential $V$
is finite whenever $d\ge 2$. We prove the conjecture for smooth potentials
in all dimensions greater than one and for arbitrary lattices of periods.   
The conjecture so far was
proved by V.Popov and M.Skriganov \cite{PS} (see also \cite{Skr0}) in dimension $2$,
by M.Skriganov \cite{Skr1}, \cite{Skr2} in dimension $3$,
and by B.Helffer and A.Mohamed  \cite{HelMoh} in dimension $4$; M.Skriganov \cite{Skr1}
has also shown the conjecture to hold in arbitrary dimension under the assumption
that the lattice of periods is rational. In the case $d=3$ the conjecture was proved 
in \cite{Kar} for non-smooth or even singular potentials (admitting Coulomb and even stronger
singularities). An interesting
approach to proving the conjecture was presented by O.A.Veliev in \cite{Vel}.

There is a number of problems closely related to the Bethe-Sommerfeld conjecture on which
extensive work has been done; the relevant publications include,
but are by no means restricted to, \cite{DahTru}, \cite{Kar} (and references therein), \cite{PS1},
\cite{PS2}. Methods used to tackle these problems range from number theory
(\cite{Skr1}, \cite{Skr2}, \cite{PS1}, \cite{PS2}) to microlocal analysis
in \cite{HelMoh} and perturbation theory in \cite{Kar}, \cite{Vel0} and \cite{Vel}. The approach used in the present paper consists, mostly,
of perturbation theoretical arguments with a bit of geometry and geometrical
combinatorics thrown in at the end.

There are certain parallels between the approach of our paper and the approach used in \cite{Vel}.
In particular, there are several important intermediate results in our paper and in \cite{Vel}
which look rather similar to each other. Examples of such similarities are: precise asymptotic
formulae for eigenvalues in the non-resonance regions and some, although not very precise, formulae in the
resonance regions; proving that the eigenvalue is simple when we move the dual parameter $\bxi$ along
a certain short interval, and, finally, the use of geometrical combinatorics.
However, here the similarities
end, because the detailed methods used on each step are completely different. For example, paper \cite{Vel}
makes a heavy use of the asymptotic formulae for the eigenfunctions, whereas in our paper they are not needed.
On the other hand, we prove that each eigenvalue close to $\la$ is described by exactly one asymptotic formula
(i.e. the mapping $f$ constructed in our paper is a bijection in a certain sense), and this plays an essential
role in our proof, but in \cite{Vel}
this property is not required at all. In \cite{Vel} a very important role is played by the isoenergetic surface,
whereas we don't need it. This list can be continued, but it is probably better to stop here and state
once again: the methods of \cite{Vel} and our paper are different, despite the similarity of some intermediate
results. It is also worthwhile to mention that asymptotic expressions for eigenfunctions as well as
asymptotic formulae for isoenergetic surfaces were obtained by Yu.Karpeshina (see for example \cite{Kar}).



In many of the papers mentioned above, proving the conjecture in special cases comes together with obtaining lower bounds for
either of the functions describing the band structure of the spectrum:
the multiplicity of overlapping $m(\la)$ and the overlapping function
$\zeta(\la)$ (we will give a definition of these functions in the next
section). For example, in dimensions $d=2,3,4$ it has been proved in
\cite{Skr1}, \cite{Skr2}, \cite{DahTru}, and \cite{PS2} that for large $\la$ we have
\bees
m(\la)\gg \la^{\frac{d-1}{4}}
\enes
and
\bee\label{estimate:zeta1}
\zeta(\la)\gg  \la^{\frac{3-d}{4}};
\ene
however, these estimates do not seem
likely to hold in high dimensions. The estimates of the present paper are rather
weaker, but they hold in all dimensions. Unfortunately, our approach does not allow to
say anything
stronger than $m(\la)\ge 1$ for large $\la$ (this inequality is equivalent to the finiteness
of the number of spectral gaps). However, it is possible to give a nontrivial lower bound for the
overlapping function: we will show that in all dimensions for sufficiently large $\la$
\bee\label{estimate:zeta2}
\zeta(\la)\gg  \la^{\frac{1-d}{2}}.
\ene

The rest of the introduction is devoted to the informal discussion of the proof.
Since the proof of the main theorem \ref{maintheorem2} is rather complicated and
technically involved, the major ideas are outlined here.

After an affine change of coordinates, we can re-write our operator
\eqref{Schroedinger} as
\bee\label{Schroedinger1}
H=H_0+V(\bx), \qquad H_0=\BD \BG\BD,
\ene
with the potential $V$ being smooth and periodic with the lattice of periods of $V$ equal
$(2\pi\Z)^d$
($\BD=i\nabla$ and $\BG=\BF^2$ is $d\times d$ positive matrix, where the matrix $\BF$ is
also assumed to be positive).
Without loss of generality, we assume that the average of the potential $V$ over the
cell $[0,2\pi]^d$
is zero (otherwise we simply subtract this average from the potential, which will
just shift the spectrum of the problem). Let us fix a sufficiently large
value of spectral parameter $\lambda=\rho^2$; we will prove that $\la$ is inside
the spectrum of $H$.

The first step of the proof, as usual, consists in performing the Floquet-Bloch
decomposition to our operator \eqref{Schroedinger1}:
\bee\label{directintegral}
H=\int_{\oplus}H(\bk)d\bk,
\ene
where $H(\bk)=H_0+V(\bx)$
is the family of `twisted' operators with the same symbol as $H$ acting in
$L^2(T^d)$ where $T^d:=\{\bx\in\R^d,\,|x_j|\le\pi,1\le j\le d\}$.
The domain $\GD(\bk)$ of $H(\bk)$ consists of functions $f\in H^2(T^d)$
satisfying the boundary conditions $f\bigm|_{x_j=\pi}=e^{i2\pi k_j}f\bigm|_{x_j=-\pi}$,
$\frac{\partial f}{\partial x_j}\bigm|_{x_j=\pi}=e^{i2\pi k_j}\frac{\partial f}{\partial x_j}\bigm|_{x_j=-\pi}$.
These auxiliary operators are labelled by
the quasi-momentum $\bk\in\R^d/\Z^d$;
see \cite{RS} for more details about this decomposition.
The next step is to assume that the potential $V$
is a finite trigonometric polynomial whose Fourier coefficients
$\hat V(\bm)$ vanish when $|\bm|>R$.
The justification of the fact
that it is enough to prove the conjecture in this case is not too difficult once we keep
careful control of the dependence of all the estimates on $R$.
The main part of the argument consists of finding an asymptotic formula for all
sufficiently large eigenvalues
of all operators $H(\bk)$, with an arbitrarily small power of the energy in
the remainder estimate. In order to be able to write such a formula, however, we have
to abandon the traditional way of labelling eigenvalues of each $H(\bk)$ in the
non-decreasing order. Instead, we will label eigenvalues by means of the integer vectors
$\bn\in\Z^d$. Consider, for example, the unperturbed operator $H_0(\bk)$. Its
eigenfunctions and eigenvalues are
\bees
\{e^{i(\bn+\bk)\bx}\}_{\bn\in\Z^d}
\enes
and
\bee\label{unpeig}
\{|\BF(\bn+\bk)|^2\}_{\bn\in\Z^d}
\ene
correspondingly. However, despite our precise knowledge
of eigenvalues, it is extremely difficult to write them in increasing order or, indeed,
even to derive the one-term asymptotic formula for the $j$-th eigenvalue with the
precise remainder estimate. It is rather convenient to introduce one parameter
which takes care of both the quasi-momentum $\bk$ and the integer vector $\bn$ which
labels eigenvalues in \eqref{unpeig}. We denote $\bxi:=\bn+\bk$ (notation indicates
that $\bxi$ can be thought of as being a dual variable)
so that $\bn=[\bxi]$ and $\bk=\{\bxi\}$ (integer and
fractional parts, respectively). Then we can reformulate formula \eqref{unpeig} for the
unperturbed eigenvalues as follows: there is a mapping $f:\R^d\to\R$, given by the
formula $f(\bxi)=|\BF(\bn+\bk)|^2$ such that for each $\bk$ the restriction of $f$
to $\{\bxi\in\R^d:\, \{\bxi\}=\bk\}$ is a bijection onto the set of all eigenvalues of $H(\bk)$ (counting
multiplicities). We want to give an analogue of this (trivial) statement in the
general case. Let us  
define the spherical layer
\bees
\CA:=\{\bxi\in{\mathbb R}^d,\, \bigm||\BF\bxi|^2-\la\bigm|\le 40v \}
\enes
($v$ is the $L_\infty$-norm of $V$).
Let $N\in\N$ be a fixed number. We will construct two mappings
$f,g:\CA\to{\mathbb R}$
which satisfy the following properties:

(I) for each $\bk$ the first mapping $f$ is an injection from
the set of all $\bxi$ with fractional part equal to $\bk$ into the
spectrum of $H(\bk)$ (counting multiplicities). Each eigenvalue of
$H(\bk)$ inside $J:=[\la-20v,\la+20v]$ has a pre-image $\bxi\in\CA$ with $\{\bxi\}=\bk$.
The perturbation
inequality $|f(\bxi)-|\BF\bxi|^2|\le 2v$ holds for all $\bxi\in\CA$.

(II) for $\bxi\in\CA$ satisfying $|\BF\bxi|^2\in J$ we have: $|f(\bxi)-g(\bxi)|<\rho^{-N}$;

(III) one can divide the domain of $g$ in two parts: $\CA=\CB\cup\CD$
(both $\CB$ and
$\CD$ are intersections of $\CA$ with some cones centered at the origin -- at least modulo
very small sets) such that $g(\bxi)$ is given by an explicit formula when $\bxi\in\CB$,
we have some control over $g(\bxi)$ when $\bxi\in\CD$, and the ratio of volumes of
$\CB$ and $\CD$ goes to infinity when $\rho\to\infty$.

The set $\CB$, called the non-resonance set, contains, among others,
all points $\bxi\in\CA$ which
satisfy the inequality
\bee\label{resonanceset}
|\lu \bxi,\BG\bth\ru|\ge \rho^{1/3}|\BF\bth|
\ene
for all non-zero integer vectors $\bth$ with $|\bth|\ll RN$.
The precise formula for $g$ will imply, in particular, that when $\bxi\in\CB$ we have
$g(\bxi)=|\BF\bxi|^2+G(\bxi)$ with all partial derivatives of $G$ being $O(\rho^{-\epsilon})$
for some $\epsilon>0$.
When $\bxi$ belongs to the resonance set $\CD$, we can give good estimates
only of the partial derivative
of $g$ along one direction; this direction has a small angle with the direction of $\bxi$.
The behaviour of $g$ along all other directions is much worse. Indeed,
by considering potentials $V$ which allow to perform the separation of variables,
one can see that the function $g$ can not, in general, be made even continuous in the resonance set.
However, we still have some (although rather weak) control over the behaviour
of $g$ along all directions inside the resonance set; see lemma \ref{newpartialofs} for the precise
formulation of these properties.

One should mention that asymptotic formulae of non-resonance eigenvalues (i.e. the function
$g(\bxi)$ for $\bxi\in\CB$ in our notation) and some resonance eigenvalues
were obtained before in certain cases, using completely different methods,
by O. A. Veliev, \cite{Vel0} and \cite{Vel} and Y. E. Karpeshina (see \cite{Kar0}, \cite{Kar} and references therein).
However,
as has been already mentioned, there are certain distinctions between the settings of \cite{Vel} and \cite{Kar}
and the settings of our paper. Because of this, and in order to make our paper self-contained, it seems sensible
to include an independent proof of the asymptotic formula for eigenvalues.

Before describing how to construct these mappings, we explain first how to prove the
Bethe-Sommerfeld conjecture using them. Put $\de=\rho^{-N}$. For each
$\boldeta\in\R^d$ of unit length we denote $I_{\boldeta}$ the interval consisting of points
$\bxi=t\boldeta$, $t>0$ satisfying $g(\bxi)\in[\la-\de,\la+\de]$; we will consider only
vectors $\boldeta$ for which $I_{\boldeta}\in\CB$.
Suppose we have found an interval $I_{\boldeta}$
on which the mapping $f$ is continuous. Then property (II) above together with the
intermediate value theorem would imply that there is a point $\bxi(\la)\in I_{\boldeta}$
satisfying $f(\bxi(\la))=\la$, which would mean that $\la$ is in the spectrum of $H$.
Thus, if we assume that $\la$ belongs to the spectral gap, this would imply that the mapping
$f$ is not continuous on each of the intervals $I_{\boldeta}$. A simple argument shows
that in this case for each point $\bxi\in\CB$ with $|g(\bxi)-\la|\le \de$ there exists another
point $\bxi_1\ne\bxi$ with $\bxi-\bxi_1\in\Z^d$ and $|g(\bxi_1)-\la|\ll \de$.
The existence of such a point $\bxi_1$ (which we call a conjugate point to $\bxi$)
is a crucial part of the proof;
it seems that similar arguments based on the existence of conjugate points
could be helpful in analogous problems. Afterwards, a
geometrical combinatorics argument shows that for sufficiently small $\de$
($\de\asymp\rho^{1-d}$ would do) some (moreover, most) of the points
$\bxi\in\CB\cap g^{-1}\bigl([\la-\de,\la+\de]\bigr)$ have no conjugate points;
the important part in the proof is played by the fact that the
surface $g^{-1}(\la)\cap\CB$ has positive curvature in each direction.

Now let us discuss how to construct mappings $f$ and $g$ with properties described above.
This is done in several steps. First, we prove lemma \ref{perturbation2}
which states that under certain conditions it is possible
instead of studying eigenvalues of the operator $H=H_0+V$, to study eigenvalues of the operator
\bee\label{project:0}
\sum_jP^jHP^j,
\ene
where $P^j$ are spectral projections of $H_0$; the error of this approximation
is small. This result can be applied to the operators $H(\bk)$ from the direct
integral \eqref{directintegral}. We want therefore to study the spectrum of the (direct) sum
\eqref{project:0}
where $P^j$ are projections `localized' in some domains of the $\bxi$-space.
The geometrical structure of these projections will depend on whether the localization happens inside
or outside the resonance regions.
The case of a projection
$P^j$ `localized' around a point $\bxi\in\CB$ is relatively simple: the rank of such projections does
not depend on $\rho$ or the `localization point' $\bxi$.
Thus, in this case we will need to compute the eigenvalue of the finite matrix
$P^jH(\bk)P^j$. This can be done by computing the characteristic polynomial of this matrix
and then using the iterative process based on the Banach fixed point theorem to find
the root of this characteristic polynomial. It is much more difficult to construct projections
$P^j$ corresponding to the points $\bxi$ located inside the resonance set $\CD$. The form of
projections will depend on, loosely speaking, how many linearly independent integer
vectors $\bth$ for which
\eqref{resonanceset} is not satisfied are there. The construction of such projections is the most
technically difficult part of the paper. Once these projections are constructed, it turns
out that the eigenvalues of $P^jHP^j$ with large $\rho$
can be easily expressed in terms of the eigenvalues of
the operator pencil $rA+B$ where $A$ and $B$ are fixed and $r\sim\rho$ is a large parameter.
The rest is a relatively simple perturbation theory.

The approach used in this paper can be applied to various related problems. For example,
it seems possible to obtain several new terms of the asymptotics of the integrated density
of states using these methods. It might even be possible to obtain the complete asymptotic
formula; however, this would require much more careful analysis of the mapping $g$ in the
resonance set. As an immediate `free' corollary of our results we obtain the theorem
\ref{clusters} which seems too be new. Loosely speaking, it states that there are no
`simultaneous clusters' of eigenvalues of all operators $H(\bk)$.

The approach of this paper works almost without changes for the polyharmonic operators
$(-\De)^l+V$ with a smooth periodic potential $V$.
Another possible field of applications of the results of this paper is studying the structure
of the (complex) Bloch and Fermi varieties.

The rest of the paper is constructed as follows: section 2 has all necessary preliminaries;
also, in this section for the convenience of the reader we,
taking into account the size of the paper, give references to the definitions of all major
objects in the paper. Section 3 proves the abstract result allowing to reduce
computation of the
spectrum of $H=H_0+V$ to the computation of the spectrum of $\sum_jP^jHP^j$,
$P^j$ being the spectral projections of $H_0$. Section 4 proves various
estimates of angles between lattice points which are needed to keep track on the dependence
of all results on $R$ -- the size of the support of the potential. In section 5 we apply
the abstract lemma from section 3 to our case and perform the reduction of $H$ to the sum
of simpler operators. In section 6 we compute the eigenvalues of these simpler operators
corresponding to the non-resonance set; we also give the formula for $g(\bxi)$ when
$\bxi\in\CB$. Section 7 is devoted to the study of the properties of these simpler operators
and the mapping $g$ restricted to the
resonance set $\CD$. Finally, in section 8 we prove the Bethe-Sommerfeld conjecture.

When this manuscript was ready, I have learned that another article of Veliev \cite{Vel1} was published recently.

{\bf Acknowledgement} First and foremost, I am deeply grateful to Alex Sobolev.
I was introduced to periodic problems by working jointly with him,
and our numerous conversations and discussions
resulted in much better understanding of this subject by me (and, I do hope,
by him as well). He has read the preliminary version of this manuscript and made
essential comments. Thanks also go to Keith Ball who made several important suggestions
which have substantially simplified proofs of the statements from section 4.
I am also immensely grateful to Gerassimos Barbatis, Yulia Karpeshina, Michael Levitin, and Roman Shterenberg for reading the preliminary version of this manuscript and making very useful comments and also for helping me
to prepare the final version of this text.

\section{Preliminaries}

We study the Schr\"odinger operator
\bee\label{Schroedinger11}
H=H_0+V(\bx), \qquad H_0=\BD \BG\BD,
\ene
with the potential $V$ being infinitely smooth and periodic with the lattice of periods equal
$(2\pi\Z)^d$. Here, $\BD=i\nabla$, and $\BG=\BF^2$ is $d\times d$ positive matrix; $\BF$ is
also taken to be positive.

Throughout the paper we use the following notation. If $A$ is a
bounded below self-adjoint operator with compact resolvent, then
we denote by $\{\mu_j(A)\}$ ($j=1,2,\dots$) the set of eigenvalues
of $A$ written in non-decreasing order, counting multiplicities.

As we have already mentioned, the spectrum of $H$ is the union
over $\bk\in\R^d/\Z^d$ of the spectra of the operators $H(\bk)$, the domain of each
$H(\bk)$ is $\GD(\bk)$ and
$H(\bk):=\BD \BG\BD+V(\bx)$.
By $\GH:=L^2(T^d)$ we denote the Hilbert space in which all the operators
$H(\bk)$ act.
We also denote by $H_0(\bk)$ the operator $\BD \BG\BD$ with the domain $\GD(\bk)$.
Let $\la_j(\bk)=\mu_j(H(\bk))$ be the $j$th eigenvalue of $H(\bk)$.
Then it is well-known (see, for example, \cite{RS})
that each function $\la_j(\cdot)$ is continuous
and piecewise smooth. Denote by $\ell_j$ the image of $\la_j(\cdot)$. Then
$\ell_j$ is called the $j$th {\it spectral band}. We also define, for each
$\la\in\R$, the following functions:
\bees
m(\la)=\#\{j:\,\la\in \ell_j\}
\enes
is the {\it multiplicity of overlapping} ($\#$ denotes the number of elements in
a set). The {\it overlapping function} $\zeta(\la)$ is
defined as the maximal number $t$ such that the
symmetric interval
$[\l - t, \l + t]$ is entirely
contained in one of the bands $\ell_j$:
\begin{equation*}
\z(\l) = \max_j\max \{ t\in\R :\, [ \l - t, \l + t]\subset \ell_j\}.
\enes
Finally,
\bee\label{densityofstates}
N(\la)=\int_{[0,1]^d}\# \{j:\,\la_j(\bk)<\la\}d\bk
\ene
is the {\it integrated density of states} of the operator \eqref{Schroedinger1}.
For technical reasons sometimes
it will be convenient to assume that the dimension $d$ is at least
$3$ (in the $2$-dimensional case the conjecture has been proved already, so this assumption
does not restrict generality).
The main result of the paper is the following:
\bet\label{maintheorem2} Let $d\ge 3$.
Then all sufficiently large points $\la=\rho^2$ are
inside the spectrum of $H$. Moreover, there exists a positive constant
$Z$ such that
for large enough $\rho$ the whole interval $[\rho^2-Z\rho^{1-d},\rho^2+Z\rho^{1-d}]$ lies
inside some spectral band.
\ent

Without loss of generality we always assume that $\int_{[0,2\pi]^d}V(\bx)d\bx=0$.
Abusing the notation slightly, we will denote by $V$ both the potential itself and the
operator of multiplication by $V$.

By 
$B(R)$ we denote a ball 
of radius $R$ centered at the origin. By
$C$ or $c$ we denote positive constants, depending only on $d$,
$\BG$, and norms of the potential in various Sobolev spaces $H^s$.
In section 5 we will introduce parameters $p$, $q_j$ and $M$;
constants are allowed to depend on the values of these parameters as well.
The exact value of constants can be different each time they
occur in the text,
possibly even each time they occur in the same formula. 
On the other hand, the constants which are labelled (like $C_1$, $c_3$, etc)
have their values being fixed throughout the text.
Whenever we use $O$, $o$, $\gg$, $\ll$, or $\asymp$ notation, the
constants involved will also depend on $d$, $\BG$, $M$, and norms of
the potential; the same is also the case when we use the
expression `sufficiently large'. Given two positive functions $f$ and $g$,
we say that $f\gg g$, or $g\ll
f$, or $g=O(f)$ if the ratio $\frac{g}{f}$ is bounded. We say
$f\asymp g$ if $f\gg g$ and $f\ll g$.
By $\la=\rho^2$ we denote a point on the spectral axis.
We will always assume that $\la$ is sufficiently large.
We also denote by $v$ the $L_{\infty}$-norm of the potential $V$, and
$J:=[\la-20v,\la+20v]$. 
Finally,
\bee\label{Arho}
\CA:=\{\bxi\in{\mathbb R}^d,\, \,\bigm| |\BF\bxi|^2-\la\bigm|\le 40 v \}
\ene
and
\bee\label{Arho1intro}
\CA_1:=\{\bxi\in{\mathbb R}^d,\, \,\bigm| |\BF\bxi|^2-\la\bigm|\le 20 v \}.
\ene
Notice that the definition of $\CA$ obviously implies that if $\bxi\in\CA$, then
$\bigm| |\BF\bxi|-\rho\bigm|\ll\rho^{-1}$.

Given several vectors $\boldeta_1,\dots,\boldeta_n\in\R^d$, we denote by
$R(\boldeta_1,\dots,\boldeta_n)$ the linear subspace spanned by
these vectors, and by $Z(\boldeta_1,\dots,\boldeta_n)$ the lattice
generated by them (i.e. the set of all linear combinations of
$\boldeta_j$ with integer coefficients; we will use this
notation only when these vectors are linearly independent).
We denote by $\BM(\boldeta_1,\dots,\boldeta_n)$ the $d\times n$
matrix whose $j$th column equals $\boldeta_j$.
Given
any lattice $\Ga$, we denote by $|\Ga|$ the volume of the
cell of $\Ga$, so that if $\Ga=Z(\boldeta_1,\dots,\boldeta_d)$,
then $|\Ga|$ is the absolute value of the determinant of $\BM(\boldeta_1,\dots,\boldeta_d)$.
We also denote, for any linear space $\GV\subset\R^d$, 
$B(\GV;R):=\GV\cap B(R)$.
For any non-zero vector $\bxi\in\R^d$ we denote
$n(\bxi):=\frac{\bxi}{|\BF\bxi|}$. Any vector $\bxi\in\R^d$ can be
uniquely decomposed as $\bxi=\bn+\bk$ with $\bn\in\Z^d$ and
$\bk\in [0,1)^d$. We call $\bn=[\bxi]$ the integer part of $\bxi$ and
$\bk=\{\bxi\}$ the fractional part of $\bxi$.

Whenever $P$ is a projection and $A$ is an arbitrary operator
acting in a Hilbert space $\CH$, the expression $PAP$ means,
slightly abusing the notation, the operator $PAP:P\CH\to P\CH$.


Throughout the paper we use the following convention:
vectors are denoted by bold lowercase letters; matrices by bold uppercase letters;
sets (subsets of $\R^d$) by calligraphic uppercase letters; linear subspaces by gothic
uppercase letters. By $\volume(\CC)$ we denote the Lebesgue measure of the set $\CC$.
If $\CC_j\subset \R^d$, $j=1,2$ are two subsets of $\R^d$, their sum is defined in the
usual way:
\bees
\CC_1+\CC_2=\{\bxi\in\R^d:\,\bxi=\bxi_1+\bxi_2,\,\bxi_j\in\CC_j\}.
\enes

Finally, for the benefit of the reader we will list here either the
definitions of the major objects introduced later in the paper or references
to the formulas in which they are defined.

$f,g:\CA\to\R$ are mappings satisfying properties listed in theorem \ref{maintheorem1}
(if the Fourier transform of $V$ has compact support) and in corollary \ref{maincorollary1} for general
potentials. The sets $\T_j$ and $\T'_j$ are defined in \eqref{Tj}.
The projections
$\CP^{(\bk)}(\CC)$ are defined immediately before Lemma \ref{potential}. $\CV(n)$, $\bxi_\GV$,
$\bxi^{\perp}_\GV$, and $\T(\GV)$ are defined at the beginning of subsection 7.1.
The sets $\Xi(\GV)$ and $\Xi_j(\GV)$ ($j=0,\dots,3$)
are defined in formulas \eqref{defXi0}-\eqref{defXi}; the sets
$\BUps_j(\bxi)$, $\BUps(\bxi)$, $\BUps(\bxi_1;\bxi_2)$, and $\BUps(\bxi;U)$ are defined by
formulas \eqref{newBUpsj}, \eqref{newBUps}, \eqref{bupsxi12}, and \eqref{bupsxiu} correspondingly.
The numbers $p$ and $q_n$ are defined in \eqref{pq}, $K=\rho^p$ and
$L_n=\rho^{q_n}$.
The projection $P(\bxi)$ and the operator $H'(\bxi)$ are defined in \eqref{Pbxi} and \eqref{Hbxi} correspondingly.
The sets $\CB$ and $\CD$ are defined in \eqref{Brho} and \eqref{CD}.
$r(\bxi)$ and $\bxi'_\GV$ are defined by formula
\eqref{rbxi}. Operators $A$ and $B$ are defined by \eqref{operatorA} and \eqref{operatorB}.
Finally, the sets $\CA(\de)$, $\CB(\de)$, and $\CD(\de)$ are defined before Lemma
\ref{volumeCA5}.

\section{Reduction to invariant subspaces: general result}

The key tool in finding a good approximation of the eigenvalues of
$H(\bk)$ will be the following two lemmas.

\begin{lem}\label{perturbation1}
Let $H_0$, $V$ and $A$ be self-adjoint operators such that $H_0$
is bounded below and has compact resolvent, and $V$ and $A$ are
bounded. Put $H=H_0+V$ and $\hat H=H_0+V+A$ and denote by
$\mu_l=\mu_l(H)$ and $\hat\mu_l=\mu_l(\hat H)$ the sets of
eigenvalues of these operators. Let $\{P_j\}$ ($j=0,\dots,n$) be a
collection of orthogonal projections commuting with $H_0$ such
that $\sum P_j=I$, $P_jVP_k=0$ for $|j-k|>1$, and $A=P_nA$. Let
$l$ be a fixed number. Denote by $a_j$ the distance from $\mu_l$
to the spectrum of $P_jH_0P_j$
Assume that for $j\ge 1$ we have $a_j>4a$, where $a:=||V||+||A||$. Then
$|\hat\mu_l-\mu_l|\le 2^{2n}a^{2n+1}\prod_{j=1}^n(a_j-2a)^{-2}$.
\end{lem}
\begin{proof}
Let $H_t=H+tA$, $0\le t\le 1$ and let $\mu(t)=\mu_l(H(t))$ be the
corresponding family of eigenvalues. We also choose the family of corresponding
normalized eigenfunctions $\phi(t)=\phi_l(t)$. We will skip
writing the index $l$ in the rest of the proof. Elementary
perturbation theory (see, e.g., \cite{Kat}) implies that $\mu(t)$
is piecewise differentiable and
\bee\label{pert1}
\frac{d\mu(t)}{dt}=(A\phi(t),\phi(t)).
\end{equation}
Let $\Phi_j=\Phi_j(t):=P_j\phi(t)$, and let $V_{kj}:=P_kVP_j$ (so
that $V_{jk}=0$ if $|j-k|>1$). Then the eigenvalue equation for
$\phi(t)$ can be written in the following way:
\bee\label{pert2}
\bes H_0\Phi_0+V_{0\,0}\Phi_{0}+V_{0\,1}\Phi_{1}
&=\mu(t)\Phi_0\\
H_0\Phi_j+V_{j\,j-1}\Phi_{j-1}+V_{j\,j}\Phi_{j}+V_{j\,j+1}\Phi_{j+1}
&=\mu(t)\Phi_j,\,\,1\le j<n\\
H_0\Phi_n+V_{n\,n-1}\Phi_{n-1}+V_{n\,n}\Phi_{n}+t A\Phi_{n}
&=\mu(t)\Phi_n.
\end{split}
\end{equation}
Indeed, let us apply $P_k$ to both sides of equation
$(H_0+V+tA)\phi=\mu\phi$. We will obtain
\bee\label{add1}
H_0\Phi_k+P_kV\phi+tP_kA\phi=\mu\Phi_k.
\ene Now we use the
following identities:
\begin{equation}\label{add2}
P_kV\phi=P_kV\sum_jP_j\phi=\sum_{j,|j-k|\le 1}P_kVP_j\phi=
\sum_{j,|j-k|\le 1} V_{kj}\Phi_j
\end{equation}
and $A=P_nA=(P_nA)^*=AP_n=P_nAP_n$, so
\bee\label{add3}
P_kA\phi=\de_{k,n}A\Phi_n.
\ene
Identities
\eqref{add1}--\eqref{add3} imply \eqref{pert2}.

Let us now prove, using the backwards induction, that for all $k$,
$1\le k\le n$ and all $t\in [0,1]$ we have \bee\label{induction}
||\Phi_k(t)||\le \frac{2a}{a_k-2a}||\Phi_{k-1}(t)||.
\end{equation}
Indeed, from the last equation in \eqref{pert2} we see that
\bee\label{Phin} \Phi_n(t)=-[P_n(H_0+V_{n\,n}+t
A-\mu(t))P_n]^{-1}V_{n\,n-1}\Phi_{n-1}(t).
\end{equation}
Since $|\mu-\mu(t)|\le a$, the distance from $\mu(t)$ to the
spectrum of $P_nH_0P_n$ is at least $a_n-a$. Since $||V_{n\,n}+t A||\le a$,
this implies
\begin{equation*}
||[P_n(H_0+V_{n\,n}+t A-\mu(t))P_n]^{-1}||\le \frac{1}{a_n-2a}.
\end{equation*}
Thus, \eqref{Phin} implies  $||\Phi_n(t)||\le
\frac{a}{a_n-2a}||\Phi_{n-1}(t)||$, and \eqref{induction} holds
for $k=n$. Assume now that we have proved \eqref{induction} for
all $k$ between $j+1$ and $n$, $1\le j<n$. Then, analogously to
\eqref{Phin}, we have:
\bee
\Phi_j(t)=-[P_j(H_0+V_{j\,j}-\mu(t))P_j]^{-1}
[V_{j\,j-1}\Phi_{j-1}(t)+V_{j\,j+1}\Phi_{j+1}(t)],
\end{equation}
so
\bee
\bes ||\Phi_j(t)||&\le
\frac{a}{a_j-2a}(||\Phi_{j-1}(t)||+||\Phi_{j+1}(t)||)\\
&\le
\frac{a}{a_j-2a}||\Phi_{j-1}(t)||+\frac{2a^2}{(a_j-2a)(a_{j+1}-2a)}||\Phi_j(t)||,
\end{split}
\end{equation}
where we have used the validity of \eqref{induction} for $k=j+1$.
This shows that \eqref{induction} holds for $k=j$, since
\bees
\frac{2a^2}{(a_j-2a)(a_{j+1}-2a)}<\frac{1}{2}.
\enes

Using \eqref{induction} and the fact that $||\Phi_1||\le 1$, we
see that
\bees
||\Phi_n||\le
\frac{2^{n}a^{n}}{\prod_{j=1}^n(a_j-2a)}.
\enes
Since the RHS of
\eqref{pert1} equals
\begin{equation*}
(P_nAP_n\phi(t),\phi(t))= (A\Phi_n(t),\Phi_n(t)),
\end{equation*}
this finishes the proof
\end{proof}

Now we formulate the immediate corollary of lemma
\ref{perturbation1} which we will be using throughout.

\bel\label{perturbation2} Let $H_0$ and $V$ be self-adjoint
operators such that $H_0$ is bounded below and has compact
resolvent and $V$ is bounded. Let $\{P^m\}$ ($m=0,\dots,T$) be a
collection of orthogonal projections commuting with $H_0$ such
that if $m\ne n$ then $P^mP^n=P^mVP^n=0$. Denote $Q:=I-\sum P^m$.
Suppose that each $P^m$ is a further sum of orthogonal projections
commuting with $H_0$: $P^m=\sum_{j=0}^{j_m} P^m_j$ such that
$P_j^mVP_l^m=0$ for $|j-l|>1$ and $P_j^mVQ=0$ if $j<j_m$. Let
$v:=||V||$ and let us fix an interval $J=[\la_1,\la_2]$ on the
spectral axis which satisfies the following properties: spectra of
the operators $QH_0Q$ and $P_j^kH_0P_j^k$, $j\ge 1$ lie outside $J$;
moreover, the distance from the spectrum of $QH_0Q$ to $J$ is
greater than $6v$ and the distance from the spectrum of
$P_j^kH_0P_j^k$ ($j\ge 1$) to $J$, which we denote by $a_j^k$, is
greater than $16v$. Denote by $\mu_p\le\dots\le\mu_q$ all
eigenvalues of $H=H_0+V$ which are inside $J$. Then the
corresponding eigenvalues $\tilde\mu_p,\dots,\tilde\mu_q$ of the
operator
\bees
\tilde H:=\sum_m P^mHP^m+QH_0Q
\enes
are eigenvalues of
$\sum_m P^mHP^m$, and they satisfy
\bees
|\tilde\mu_r-\mu_r|\le
\max_m \left[(6v)^{2j_m+1}\prod_{j=1}^{j_m}(a_j^m-6v)^{-2}\right];
\enes
all other
eigenvalues of $\tilde H$ are outside the interval
$[\la_1+2v,\la_2-2v]$.

\enl

\bep
Assumptions of the lemma imply that
\bees
H=\tilde H+
(Q+\sum_m P_{j_m}^m)V(Q+\sum_m P_{j_m}^m)-(\sum_m
P_{j_m}^m)V(\sum_m P_{j_m}^m).
\end{equation*}
Therefore, $\tilde H-2v(Q+\sum_m P_{j_m}^m)\le H\le \tilde
H+2v(Q+\sum_m P_{j_m}^m)$, and the elementary perturbation theory
implies that for all $l$
\bee\label{1.7}
\mu_l(\tilde H-2v(Q+\sum_m P_{j_m}^m))\le \mu_l(H)
\le \mu_l(\tilde H+2v(Q+\sum_m P_{j_m}^m)).
\ene
The operators $\tilde H\pm 2v(Q+\sum_m P_{j_m}^m)$ split into
the sum of invariant operators $QH_0Q\pm 2vQ$ and $P^mHP^m\pm 2v
P_{j_m}^m$ ($m=0,\dots,n$). The spectrum of operators $QH_0Q\pm
2vQ$ is outside $[\la_1-4v,\la_2+4v]$ due to the assumptions of the
lemma. Therefore, since the shift of an eigenvalue is at most the
norm of the perturbation, for $p\le l\le q$, $\mu_l(\tilde H\pm
2v(Q+\sum_m P_{j_m}^m))$ is an eigenvalue of one of the operators
$P^mHP^m\pm 2v P_{j_m}^m$. If we now apply lemma \ref{perturbation1}
to each of the operators $P^mHP^m\pm 2v
P_{j_m}^m$ with $A:=\pm 2v P_{j_m}^m$ and $a=3v$, we will obtain
\bee\label{neweq:0}
\bes
|\mu_k(P^mHP^m\pm 2v
P_{j_m}^m)&-\mu_k(P^mHP^m)|\\
&\le 6^{2j_m+1}v^{2j_m+1}\max_m\prod_{j=1}^{j_m}(a_j^m-6v)^{-2}\\
&\le \max_m \left[(6v)^{2j_m+1}\prod_{j=1}^{j_m}(a_j^m-6v)^{-2}\right]=:\tau,
\end{split}
\ene
provided $\mu_k(P^mHP^m)\in [\la_1-4v,\la_2+4v]$. Let us now define the bijection $F$ mapping
the set of all eigenvalues of $\tilde H$ to the set of all eigenvalues of
$\{\mu_l(\tilde H+2v(Q+\sum_m P_{j_m}^m))\}$ (counting multiplicities) in the following way.
Suppose, $\mu$ is an eigenvalue of $\tilde H$. Then either $\mu=\mu_k(QH_0Q)$, or
$\mu=\mu_k(P^mHP^m)$ for some $k,m$. We define $F(\mu):=\mu_k(QH_0Q+2vQ)$ in the former case, and
$F(\mu):=\mu_k(P^mHP^m+2vP^m_{j_m})$ in the latter case. Then the mapping $F$ satisfies the following properties:
\bee\label{neweq:1}
|F(\mu)-\mu|\le 2v;
\ene
moreover, if $\mu\in [\la_1-4v, \la_2+4v]$, then
\bee\label{neweq:2}
|F(\mu)-\mu|\le \tau
\ene
(this follows from \eqref{neweq:0}). A little thought shows that this implies
\bee\label{Ger2}
|\mu_l(\tilde H+ 2v(Q+\sum_m
P_{j_m}^m))-\mu_l(\tilde H)|\le
\tau
\ene
for $p\le l\le q$. Indeed, suppose that \eqref{Ger2} is not satisfied for some $l$, say
\bee\label{Ger3}
\mu_l(\tilde H+ 2v(Q+\sum_m
P_{j_m}^m))-\mu_l(\tilde H)>\tau;
\ene
in particular, this implies $\mu_l(\tilde H+ 2v(Q+\sum_mP_{j_m}^m))>\la_1-2v$.
Then the pigeonhole principle shows that $F$ maps at least one of the eigenvalues $\mu_k(\tilde H)$,
$k\le l$ to $\mu_t(\tilde H+2v(Q+\sum_m P_{j_m}^m))$ with
$t\ge l$. If $\mu_k(\tilde H)<\la_1-4v$, this contradicts \eqref{neweq:1},
and if $\mu_k(\tilde H)\ge\la_1-4v$, this contradicts \eqref{neweq:2}.
These contradictions prove \eqref{Ger2}. Similarly, we prove that
\bee\label{Ger4}
|\mu_l(\tilde H- 2v(Q+\sum_m
P_{j_m}^m))-\mu_l(\tilde H)|\le
\tau.
\ene
Estimates \eqref{Ger2} and \eqref{Ger4}
together with \eqref{1.7} prove the lemma.
\enp

\bec\label{perturbationnew}
If all conditions of lemma \ref{perturbation2} are satisfied, there exists an injection
$G$ defined on a set of eigenvalues of the operator $\sum_m P^mHP^m$ (all eigenvalues are
counted according to their multiplicities) and mapping them to a subset of the
set of eigenvalues of
$H$ (again considered counting multiplicities) such that:

(i) all eigenvalues of $H$ inside $J$ have a pre-image,

(ii) If $\mu_j\in [\la_1+2v,\la_2-2v]$ is an eigenvalue of $\sum_m P^mHP^m$, then
\bees
|G(\mu_j)-\mu_j|\le
\max_m \left[(6v)^{2j_m+1}\prod_{j=1}^{j_m}(a_j^m-6v)^{-2}\right],
\enes

and

(iii) $G(\mu_j(\sum_m P^mHP^m))=\mu_{j+l}(H)$, where $l$ is the number of eigenvalues of
$QH_0Q$ which are smaller than $\lambda_1$.
\enc

\bep
Statements (i) and (ii) follow immediately from lemma \ref{perturbation2}, and to prove
(iii) we just notice that if $\mu_j(\sum_m P^mHP^m)\in J$, then
\bees
\mu_j(\sum_m P^mHP^m)=\mu_{j+l}(QH_0Q+\sum_m P^mHP^m).
\enes
\enp

\section{Lattice points}
In this section, we prove various auxiliary estimates of angles
between integer vectors.
\bel\label{complement0} Let
$\boldeta_1,\dots,\boldeta_n\in\Z^d$ be linearly independent. Let
$\Ga=Z(\boldeta_1,\dots,\boldeta_n)$ and suppose that
$\bnu_1,\dots,\bnu_{n-1}\in\Ga\cap B(R)$. Then there exists a
vector $\bth\in\Ga$, $\bth\ne 0$ orthogonal to all $\bnu_j$'s,
such that
\bee\label{eq:complement0}
|\bth|\le 2^{n}|\Ga|\prod_{j=1}^{n-1}|\bnu_j|
\ene
and, therefore, $|\bth|\le 2^{n}|\Ga|R^{n-1}$.
\enl
\begin{proof} 
For $r>1$ let $\CA_r\subset R(\boldeta_1,\dots,\boldeta_n)$ be the
set
\begin{equation*}
\CA_r = \{\bxi\in R(\boldeta_1,\dots,\boldeta_n): |\lu \bxi,
\bnu_j\ru|< 1, j = 1, 2, \dots, n-1,\ \& \ |\bxi|< r \}.
\end{equation*}
This set is obviously convex and symmetric about the origin.
Moreover,
\begin{equation*}
\volume(\CA_r)> r \prod_{j=1}^{n-1}|\bnu_j|^{-1}
\end{equation*}
By Minkowski's convex body theorem (see, e.g., \cite{Cas},
\S III.2.2 Theorem II), under the condition
$\volume(\CA_r)> |\Ga| 2^n$ the set $\CA_r$ contains at least
two non-zero points $\pm\bth\in\Ga$. The above condition is satisfied if $r
\prod_{j=1}^{n-1}|\bnu_j|^{-1}\ge 2^d|\Ga|$, that is if $r\ge
2^{d}|\Ga|\prod_{j=1}^{n-1}|\bnu_j|$. Since $\bnu_j$'s and $\bth$
are integer vectors, the condition $|\lu\bth, \bnu_j\ru| < 1$ is
equivalent to $\lu\bth, \bnu_j\ru = 0$. This implies the required
result.
\end{proof}

\bel\label{angle} Let $\bth_1,\dots,\bth_n,\bmu\in\Z^d\cap B(R)$
be linearly independent. Then the angle between $\bmu$ and
$R(\bth_1,\dots,\bth_n)$ is $\gg R^{-n-1}$.
\enl
\bep
Suppose this
angle is smaller than $R^{-n-1}$. Then the lattice
$\Ga=Z(\bth_1,\dots,\bth_n,\bmu)$ has $|\Ga|\le 1$. Lemma
\ref{complement0} then implies that there exists a vector
$\bth\in\Ga$, $\bth\perp\bth_j$, $|\bth|\ll R^n$. Then, since $\bth$
and $\bmu$ are non-orthogonal integer vectors, we have:
$|\lu\bmu,\bth\ru|\ge 1$, and $\sin$ of the angle between $\bmu$
and $R(\bth_1,\dots,\bth_n)$, which equals $\cos$ of the angle
between $\bmu$ and $\bth$, is bounded below by
$|\bth|^{-1}|\bmu|^{-1}\gg R^{-n-1}$.
\enp

\bec\label{angle:cor}
Let $\bth_1,\dots,\bth_n,\bmu\in\Z^d\cap
B(R)$ be linearly independent. Then the angle between $\BF\bmu$
and $R(\BF\bth_1,\dots,\BF\bth_n)$ is $\gg R^{-n-1}$.
\enc
\bep
This is equivalent to saying that for each $\bxi\in
R(\bth_1,\dots,\bth_n)$ the distance between $\BF(n(\bmu))$ and
$\BF\bxi$ is larger than $c R^{-n-1}$.
But the distance between $\BF(n(\bmu))$ and $\BF\bxi$
is not greater than the largest eigenvalue of $\BF$ times the
distance between $n(\bmu)$ and $\bxi$. Now the statement follows
from lemma \ref{angle}.
\enp

It is possible to generalize lemma \ref{angle} a bit: if we talk about distance from a vector to
a linear sub-space instead of the angle between a vector and a subspace,
we can drop the assumption that $|\mu|\le R$:
\bel\label{new:distance} Let $\bth_1,\dots,\bth_n\in\Z^d\cap B(R)$ and $\bmu\in\Z^d$
be linearly independent. Then the distance between $\bmu$ and
$R(\bth_1,\dots,\bth_n)$ is $\gg R^{-n}$.
\enl
\bep
The distance between $\bmu$ and $R(\bth_1,\dots,\bth_n)$ equals
$\frac{|Z(\mu,\bth_1,\dots,\bth_n)|}{|Z(\bth_1,\dots,\bth_n)|}$. The square of the denominator of this fraction is
the determinant of the $n\times n$ matrix $A$ with $A_{jk}:=\lu\bth_j,\bth_k\ru=O(R^2)$, so the denominator
is $O(R^n)$. Similarly, the square of the numerator is the determinant of $(n+1)\times(n+1)$ non-singular matrix
with integer entries. Therefore, the absolute value of the numerator is at least $1$. This proves our statement.
\enp

The following result is a generalization of lemma \ref{complement0}
and the proof is similar:
\bel\label{complement} Let $\Ga$ be as above and let
$\bnu_1,\dots,\bnu_m\in\Ga\cap B(R)$ ($m<n$)
Then there exist
linearly independent vectors $\bth_1,\dots,\bth_{n-m}\in\Ga$ such
that each $\bth_l$ is orthogonal to each $\bnu_j$ and
\bee\label{eq:complement}
\prod_{l=1}^{n-m}|\bth_l|\ll |\Ga|\prod_{j=1}^{m}|\bnu_j|\le |\Ga|R^m
\ene
\enl
\bep
Applying lemma \ref{complement0} $n-m$ times, we see that the set of vectors
from $\Ga$ which are
orthogonal to $\bnu_j$ form a lattice $\Ga_{n-m}$ of dimension $n-m$.
Let $\bth_j$ ($j=1,\dots,n-m$) be successive minimal vectors of $\Ga_{n-m}$.
That means that $\bth_1$ is the smallest nonzero vector in $\Ga_{n-m}$; $\bth_2\in\Ga_{n-m}$
is the smallest vector linearly independent of $\bth_1$; $\bth_3\in\Ga_{n-m}$ is the smallest
vector linearly independent of $\bth_1,\bth_2$, etc.

For $r>|\bth_1|$ let $\CA_r\subset R(\boldeta_1,\dots,\boldeta_n)$ be the
set
\begin{equation*}
\CA_r = \{\bxi\in R(\boldeta_1,\dots,\boldeta_n): |\lu \bxi,
\bnu_j\ru|< 1, j = 1, 2, \dots, m,\ \& \ |\bxi|< r \}.
\end{equation*}
This set is obviously convex and symmetric about the origin.
Moreover,
\begin{equation}\label{convex1}
\volume(\CA_r)\gg r^{n-m} \prod_{j=1}^{m}|\bnu_j|^{-1}.
\end{equation}
Applying again Minkowski's convex body theorem, we find that
the set $\CA_r$ contains at least
\bee\label{convex2}
N=[2^{-n}|\Ga|^{-1}\volume(\CA_r)]\gg |\Ga|^{-1} r^{n-m} \prod_{j=1}^{m}|\bnu_j|^{-1}
\ene
pairs of points $\pm\bmu_k\in\Ga$, $k=1,\dots,N$.
Obviously, each $\bmu_k$ is orthogonal to each $\bnu_j$.
Suppose, $r<|\bth_{n-m}|$. Then, obviously,
$|\bth_p|\le r<|\bth_{p+1}|$ for some $p\le n-m-1$.
The dimension of $R(\bmu_1,\dots,\bmu_N)$ is then
$\le p$, and each
$\bmu_k$ is a linear combination of $\bth_1,\dots,\bth_p$ with integer coefficients.
Denote $\Ga_p:=Z(\bth_1,\dots,\bth_p)$. Minkowski's second theorem (see, e.g.,
\cite{Cas}, \S VIII.2, Theorem I) shows that
\bee\label{convex11}
\prod_{l=1}^p |\bth_l|\ll |\Ga_p|.
\ene
A simple packing argument
shows that $N|\Ga_p|$ is smaller than the volume of the ball of radius $(p+1)r$ in
$R(\bth_1,\dots,\bth_p)$, i.e. $N|\Ga_p|\ll r^p$. Estimate \eqref{convex11}
implies
\bees
N\prod_{l=1}^p |\bth_l|\ll r^p.
\enes

Therefore, if the condition
\bee\label{convex3}
N\prod_{l=1}^p |\bth_l| >C r^{p}
\ene
is satisfied,
where $C$ is sufficiently large, this implies that $r\ge |\bth_{p+1}|$.

Estimate \eqref{convex2} shows that if
\bee\label{ger1}
r>C |\Ga| \prod_{j=1}^{m}|\bnu_j|\prod_{l=1}^{n-m-1} |\bth_l|^{-1},
\ene
then condition \eqref{convex3} with $p=n-m-1$
will be satisfied
and this would guarantee that $r>|\bth_{n-m}|$.
In other words, if $r$ is greater than the RHS of \eqref{ger1}, then $r>|\bth_{n-m}|$.
This implies
\bees
|\bth_{n-m}|\ll|\Ga| \prod_{j=1}^{m}|\bnu_j|\prod_{l=1}^{n-m-1} |\bth_l|^{-1},
\enes
which
finishes the proof.
\enp

Let $\bnu_1,\dots,\bnu_n\in\R^d$
($n\le d$). We denote by $\BM=\BM(\bnu_1,\dots,\bnu_n)$ the $d\times n$
matrix whose $j$th column equals $\bnu_j$. We also denote
\bee\label{norm1}
\|\bnu_1\wedge\dots\wedge\bnu_n\|_2:=\sqrt{\det
(\BM^*\BM)}
\ene
($\BM^*\BM$ is obviously non-negative, and so is the
determinant). The reason for the notation is that we can think of
$\|\bnu_1\wedge\dots\wedge\bnu_n\|_2$ as being the Hilbert-Schmidt
norm of the tensor $\bnu_1\wedge\dots\wedge\bnu_n$.
\bel\label{exteriorlem}
Let
$\bnu_1,\dots,\bnu_n,\bmu_1,\dots,\bmu_m\in\R^d$.
Let $V_1=R(\bnu_1,\dots,\bnu_n)$ and
$V_2=R(\bmu_1,\dots,\bmu_m)$. Let $\al$ be the angle between $V_1$
and $V_2$. Then the following inequality holds:
\bee\label{exterior1} \sin\al\ge
\frac{\|\bnu_1\wedge\dots\wedge\bnu_n\wedge\bmu_1\wedge\dots\wedge\bmu_m\|_2}
{\|\bnu_1\wedge\dots\wedge\bnu_n\|_2\|\bmu_1\wedge\dots\wedge\bmu_m\|_2}
\ene
\enl
\bep
If we multiply matrix $\BM$ from the right by a
non-singular $n\times n$ matrix $\BB$, the expression \eqref{norm1}
is multiplied by $\det \BB$. This observation shows that elementary
transformations of the set of vectors $\bnu$ (i.e. multiplying
$\bnu_j$ by a non-zero scalar, adding $\bnu_j$ to $\bnu_k$, etc)
do not change both sides of \eqref{exterior1}; the same is the case
for elementary transformations of the vectors $\bmu$. Thus, we may assume
that vectors $\bnu$ form an orthonormal basis of $V_1$, vectors $\bmu$ form an
orthonormal basis of $V_2$, and the angle between $\bnu_1$ and
$\bmu_1$ equals $\al$. Notice that now the denominator of the RHS
of \eqref{exterior1} equals $1$. Next, we notice that an
orthogonal change of coordinates results in multiplying $\BM$ from
the left by a $d\times d$ orthogonal matrix and thus doesn't change
\eqref{norm1} and the RHS of \eqref{exterior1}; the LHS of
\eqref{exterior1} is obviously invariant under an orthogonal change
of coordinates as well. Assume, without loss of generality, that
$n\ge m$. Then, applying an orthogonal change of coordinates, we
can make our vectors to have the following form: $\bnu_j=\be_j$
($j=1,\dots,n$, where $\be_j$ are standard basis vectors),
$\bmu_j=p_j\be_j+q_j\be_{n+j}$, ($p_j,q_j\ge 0$, $p_j^2+q_j^2=1$).
Elementary geometry implies $\cos\al=p_1$, and so $\sin\al=q_1$.
Computing the determinant, we obtain:
\bees
\|\bnu_1\wedge\dots\wedge\bnu_n\wedge\bmu_1\wedge\dots\wedge\bmu_m\|_2=
\prod_{j=1}^n q_j\le q_1.
\enes
The lemma is proved.
\enp
\bel\label{angle1} Let
$\bnu_1,\dots,\bnu_n,\bmu_1,\dots,\bmu_m\in\Z^d\cap B(R)$ be
linearly independent. Let $V_1=R(\bnu_1,\dots,\bnu_n)$ and
$V_2=R(\bmu_1,\dots,\bmu_m)$. Then the angle between $V_1$ and
$V_2$ is $\gg \prod_{j=1}^n|\bnu_j|^{-1}\prod_{l=1}^m|\bmu_l|^{-1}\ge R^{-n-m}$.
\enl
\ber
It is not difficult to see that
the power $-n-m$ in lemma \ref{angle1} is optimal.
\enr
\bep We
use the inequality \eqref{exterior1} and notice that the numerator
of the RHS is a square root of an integer number (since all vectors
involved are integer) and is non-zero (since the vectors are
linearly independent). Therefore, the numerator is at least $1$.
The denominator is, obviously,
$\ll \prod_{j=1}^n|\bnu_j|\prod_{l=1}^m|\bmu_l|$. This finishes the
proof.
\enp

Using the same argument we have used while proving
Corollary \ref{angle:cor}, we can prove the following
\bec\label{angle1:cor} Let
$\bnu_1,\dots,\bnu_n,\bmu_1,\dots,\bmu_m\in\Z^d\cap B(R)$ be
linearly independent. Let $V_1=R(\BF\bnu_1,\dots,\BF\bnu_n)$ and
$V_2=R(\BF\bmu_1,\dots,\BF\bmu_m)$. Then the angle between $V_1$
and $V_2$ is $\gg R^{-n-m}$
\enc

\bel\label{angle2} Let $\bnu_1,\dots,\bnu_n\in\Z^d\cap B(R)$ and
$\bmu_1,\dots,\bmu_m\in\Z^d\cap B(R)$ be two
sets. We assume that each set consists of linearly independent vectors
(but the union of two sets is not necessary linearly independent).
Let $V_1=R(\bnu_1,\dots,\bnu_n)$ and
$V_2=R(\bmu_1,\dots,\bmu_m)$. Suppose, $\dim\left(V_1\cap
V_2\right)=l$. Then there are $l$ integer linearly independent
vectors $\bth_1,\dots,\bth_l\in \left(\Z^d\cap V_1\cap V_2\right)$
such that $|\bth_j|\ll R^{m+n-l+1}$. Moreover, the angle between
orthogonal complements to $\left(V_1\cap V_2\right)$ in $V_1$ and
$V_2$ is bounded below by $CR^{-\al}$, $\al=\al(n,m,l)=n+m+2l(m+n-l+1)$.
\enl
\bep
Denote
$\BM=\BM(\bnu_1,\dots,\bnu_n,-\bmu_1\dots,-\bmu_m)$. The
rank of $\BM$ equals $k:=m+n-l$. Without loss of generality we can
assume that the top left $k\times k$ minor of this matrix is
non-zero (otherwise we just change the order of the vectors
$-\bmu_j$ or the order of the coordinates $x_j$). In order to find the
basis of the intersection $V_1\cap V_2$ we have to solve
the system of equations
\bee\label{angle2:1}
\BM\bt=0.
\ene
Indeed, if $\bt=(t_1,\dots,t_{n+m})^T$ is a solution of
\eqref{angle2:1}, then $\sum_{p=1}^n t_p\bnu_p=\sum_{q=1}^m
t_{n+q}\bmu_q\in V_1\cap V_2$. Now the simple linear algebra
tells us that the basis of solutions of \eqref{angle2:1} is formed
by the vectors of the form $(s_1,\dots,s_k,1,0,\dots,0)$,
$(t_1,\dots,t_k,0,1,\dots,0)$,...,
$(\tau_1,\dots,\tau_k,0,\dots,0,1)$. Using Cramer's rule, we find
that each of the numbers $s_j$, $t_j$, $\tau_j$, etc is a ratio of
two determinants, each of them an integer number $\ll R^k$;
moreover, the denominator is the same for all of the numbers
$s_j$, $t_j$, etc. After multiplication by the denominator, we
obtain an integer basis of solutions of \eqref{angle2:1} with
entries $\ll R^k$. For any such solution $\bt$ the following
estimate holds: $|\sum_{p=1}^n t_p\bnu_p|\ll R^{k+1}$. This proves
the first statement of lemma. To prove the second statement, we
first use Lemma \ref{complement} to construct integer bases
$\{\boldeta_1,\dots,\boldeta_{n-l}\}$ and
$\{\bxi_1,\dots,\bxi_{m-l}\}$
of the
orthogonal complements to $V_1\cap V_2$ in $V_1$ and
$V_2$ correspondingly with properties
\bees
\prod_{j=1}^{n-l}|\boldeta_j|\ll R^{n+l(m+n-l+1)}
\enes
and
\bees
\prod_{j=1}^{m-l}|\bxi_j|\ll R^{m+l(m+n-l+1)}.
\enes
Now Lemma \ref{angle1} produces the required estimate.
This finishes the proof.
\enp
Using the
same argument we have used while proving Corollary
\ref{angle:cor}, we can prove the following
\bec\label{angle2:cor}
Let $\bnu_1,\dots,\bnu_n,\bmu_1,\dots,\bmu_m\in\Z^d\cap B(R)$ be
two linearly independent families of vectors. Let
$V_1=R(\BF\bnu_1,\dots,\BF\bnu_n)$ and
$V_2=R(\BF\bmu_1,\dots,\BF\bmu_m)$. Then the angle between
orthogonal complements to $V_1\cap V_2$ in $V_1$ and
$V_2$ is $\gg R^{-\al(n,m,l)}$.
\enc

\section{Reduction to invariant subspaces}
Let $\la=\rho^2$ be a large real number.
In this section, we use lemma \ref{perturbation2} to construct the
family of operators $\tilde H(\bk)$ the spectrum of which (or at
least the part of the spectrum near $\la$) is close to the
spectrum of $H(\bk)$.
Consider the truncated potential
\bee V'(\bx)=\sum_{\bm\in
B(R)\cap\Z^d}\hat V(\bm)e_{\bm}(\bx),
\ene
where
\begin{equation*}
e_{\bm}(\bx):=\frac{1}{(2\pi)^{d/2}} e^{i\lu\bm,\bx\ru},\ \ \bm \in
\mathbb Z^d
\end{equation*}
and
\bee\label{Fourier}
\hat V(\bm)=\int_{[0,2\pi]^d}
V(\bx)e_{-\bm}(\bx)
\ene
are the Fourier coefficients of $V$. $R$
is a large parameter the precise value of which will be chosen
later; at the moment we just state that $R\sim \rho^{\gamma}$ with
$\gamma>0$ being small. Throughout the text, we will prove various
statements which will hold under conditions of the type $R<\rho^{\gamma_j}$.
After each statement of this type, we will always assume, without possibly specifically
mentioning, that these conditions are always satisfied in what follows; at the end, we will
choose $\gamma=\min\gamma_j$.

Since $V$ is smooth, for each $n$ we have
\bee
\sup_{\bx\in\R^d}|V(\bx)-V'(\bx)|< C_n R^{-n}.
\ene
This implies
that if we denote $H'(\bk):=H_0(\bk)+V'$ with the domain
$\CD(\bk)$, the following
estimate holds for all $n$:
\bee\label{truncate1}
|\mu_j(H(\bk))-\mu_j(H'(\bk))|<C_n R^{-n}.
\ene
Throughout this and the next two sections, we
will work with the truncated operators $H'(\bk)$. These sections will be
devoted to the construction of mappings $f,g$ with properties
specified in the introduction. Let $M\in\N$ be a fixed number. For each natural $j$
we denote
\bee\label{Tj}
\T_j:=\Z^d\cap B(jR),\, \T_0:=\{0\}, \, \T'_j:=\T_j\setminus \{0\};
\ene

Let $\GV\subset\R^d$ be a
linear subspace of dimension $n$ and $r>0$. We say that $\GV$ is
an integer $r$-subspace if $\GV=R(\bth_1,\dots,\bth_n)$ and each
$\bth_j$ is an integer vector with length smaller than $r$. The
set of all integer $r$-subspaces of dimension $n$ will be denoted
by $\CV(r,n)$. We mostly will be dealing with $\CV(6MR,n)$;
for brevity we will denote $\CV(n):=\CV(6MR,n)$. If
$\bxi\in\R^d$ and $\GV\in\CV(n)$, we denote $\bxi_\GV$ and
$\bxi^{\perp}_\GV$ vectors such that
\bee\label{decompositionGV}
\bxi=\bxi_\GV+\bxi^{\perp}_\GV,\ \bxi_\GV\in \GV,\
\BG\bxi^{\perp}_\GV\perp \GV.
\ene
If $\GV\in\CV(n)$, we put
$\T(\GV):=\T_{6M}\cap \GV$, $\T'(\GV):=\T(\GV)\setminus\{0\}$.
By $p,q_n$ ($n=1,\dots,d$) we denote positive constants smaller than $1/3$; the precise
value of these constants will be specified later; we also denote $K=\rho^p$ and $L_n=\rho^{q_n}$.

Let $\GV\in\CV(n)$. We denote
\bee\label{defXi0}
\Xi_0(\GV):=\{\bxi\in\CA,|\bxi_{\GV}|<L_n\},
\ene
\bee\label{defXi1}
\Xi_1(\GV):=\bigl(\Xi_0(\GV)+\GV\bigr)\cap \CA,
\ene
\bee\label{defXi2}
\Xi_2(\GV):=\Xi_1(\GV)\setminus\bigl(\cup_{m=n+1}^d
\cup_{\GW\in\CV(m):\,\GV\subset\GW}\Xi_1(\GW)\bigr),
\ene
\bee\label{defXi3}
\Xi_3(\GV):=\Xi_2(\GV)+B(\GV,K),
\ene
and finally,
\bee\label{defXi}
\Xi(\GV):=\Xi_3(\GV)+\T_{M}.
\ene

These objects (especially $\Xi_3(\GV)$ and $\Xi(\GV)$) play a crucial role in what follows;
the pictures of them are shown in Figures \ref{fig:0}-\ref{fig:3} in the case $d=2$ (here, the integer subspaces $\GV$
are $1$-dimensional, so $\GV=R(\bth)$ with $\bth\in\T'$; we have called $\Xi_j(\bth):=\Xi_j(R(\bth))$).
\begin{figure}[!hbt]
\begin{center}
\framebox[0.8\textwidth]{\includegraphics[width=0.60\textwidth]{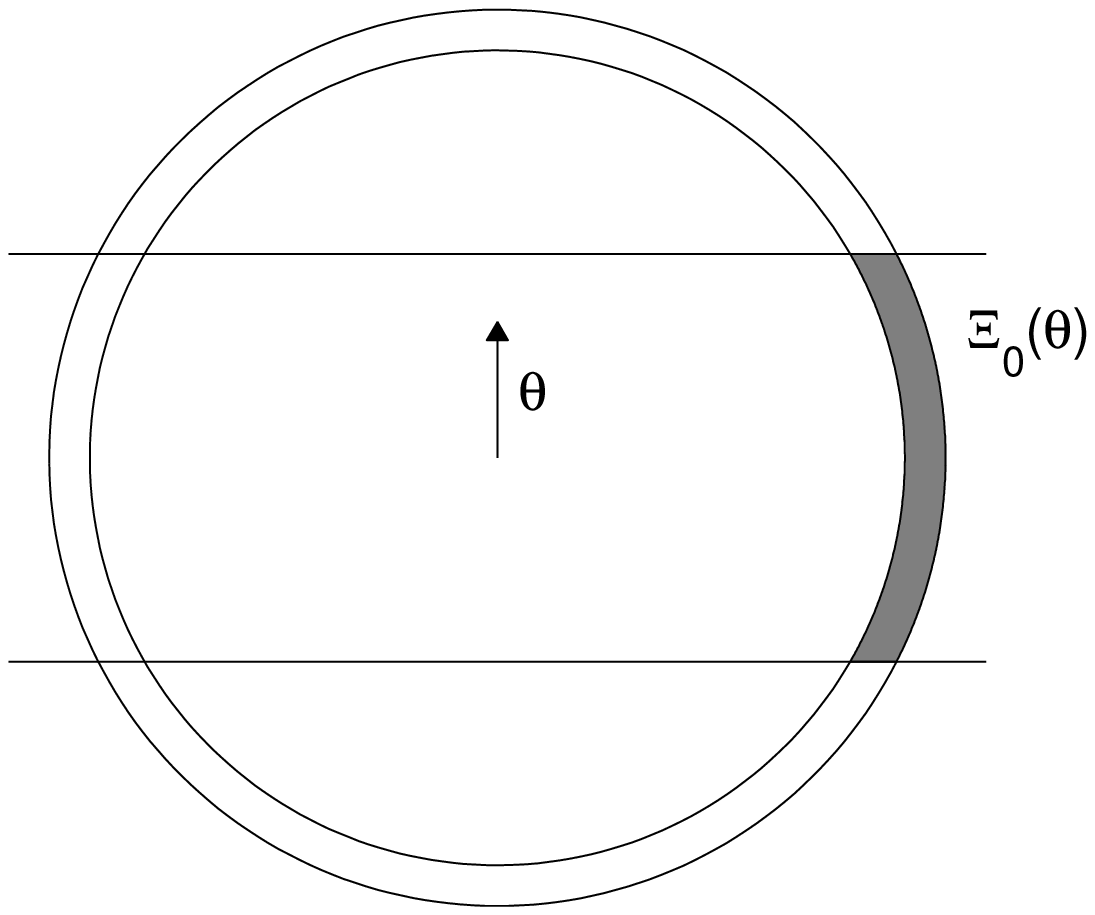}}
\caption{The set $\Xi_0(\bth)$ in the $2$-dimensional case\label{fig:0}}
\end{center}
\end{figure}
\begin{figure}[!hbt]
\begin{center}
\framebox[0.8\textwidth]{\includegraphics[width=0.60\textwidth]{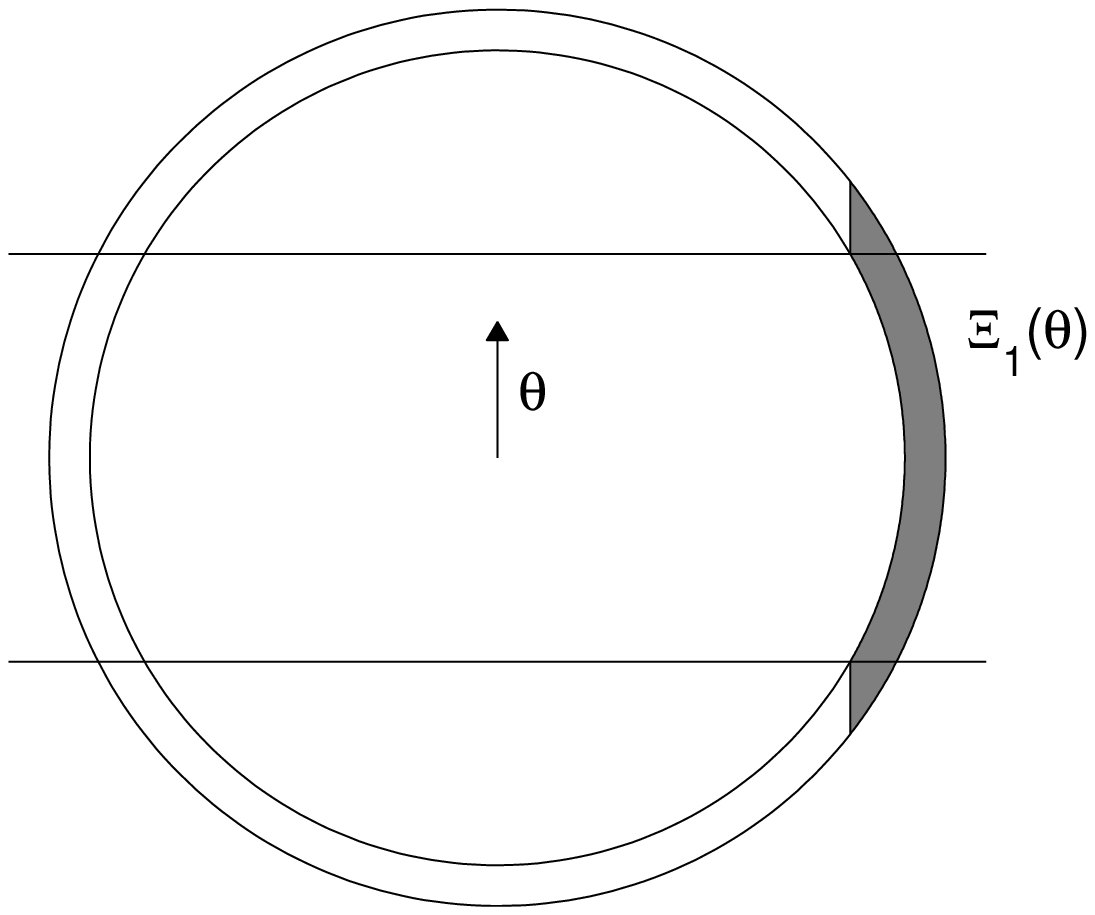}}
\caption{The set $\Xi_1(\bth)=\Xi_2(\bth)$ in the $2$-dimensional case\label{fig:1}}
\end{center}
\end{figure}
\begin{figure}[!hbt]
\begin{center}
\framebox[0.8\textwidth]{\includegraphics[width=0.60\textwidth]{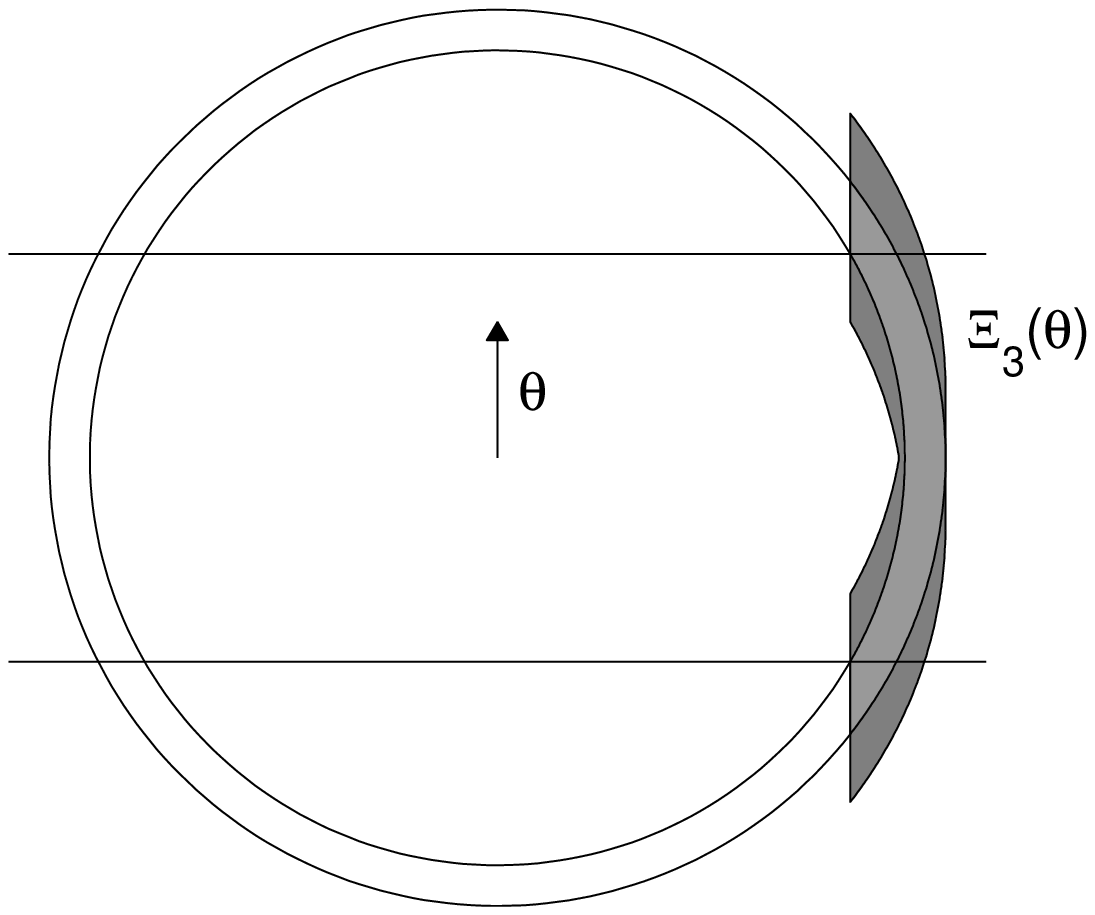}}
\caption{The set $\Xi_3(\bth)$ in the $2$-dimensional case\label{fig:2}}
\end{center}
\end{figure}
\begin{figure}[!hbt]
\begin{center}
\framebox[0.8\textwidth]{\includegraphics[width=0.60\textwidth]{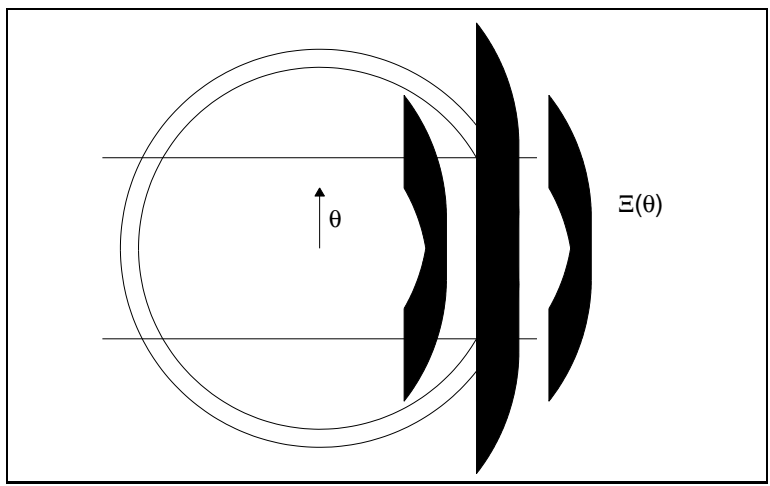}}
\caption{The set $\Xi(\bth)$ in the $2$-dimensional case; here,
$\T=\{(0,0),(\pm 1,0),(0,\pm 1)\}$ consists of five elements.\label{fig:3}}
\end{center}
\end{figure}

It may seem that the definition of these objects is
overcomplicated; for example, one may be tempted to define $\Xi_3(\GV)$ by Figure \ref{fig:4}.
This definition is indeed simpler and it would work in the $2$-dimensional case;
however, if we try to extend this definition to higher dimensions, we would find out that lemma
\ref{changedXi6} no longer holds. One more remark concerning the definitions of the sets $\Xi$ is that
it is very difficult to make a mental picture of them in high dimensions (even when $d=3$). A good approach
to working with these sets is to do it on a purely formal level, without trying to imagine how they look like.
\begin{figure}[!hbt]
\begin{center}
\framebox[0.8\textwidth]{\includegraphics[width=0.60\textwidth]{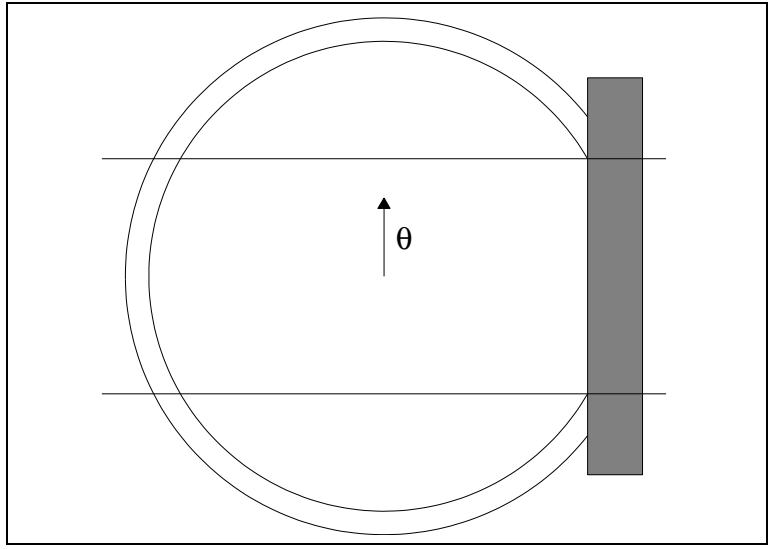}}
\caption{Bad definition of the set $\Xi_3(\bth)$ in the $2$-dimensional case\label{fig:4}}
\end{center}
\end{figure}

We also put
\bee\label{CD}
\CD:=\cup_{m=1}^d \cup_{\GW\in\CV(m)}\Xi_1(\GW)
\ene
and
\bee\label{Brho}
\CB:=\CA\setminus\CD.
\ene
We will often call the set $\CD$ {\em the resonance region} and the set $\CB$
{\em the non-resonance region}.

Note that the definitions \eqref{defXi0}--\eqref{defXi} make sense for the subspace
$\GU_0:=\{0\}\in\CV(0)$. In particular, we have $\Xi_0(\GU_0)=\Xi_1(\GU_0)=\CA$,
$\Xi_2(\GU_0)=\Xi_3(\GU_0)=\CB$, and
\bee\label{defXizero}
\Xi(\GU_0)=\CB+\T_M.
\ene

Let us now formulate several properties of the sets $\Xi_j$. In what follows, we always assume that
$\rho$ and $R$ are sufficiently large. We also assume that $L_n=\rho^{q_n}$ with
$q_{n+1}\ge q_n+3p$ for all $n$, $q_d\le 1/3$, and $K=\rho^p$ with $q_1\ge 3p>0$.
We also put $q_0=0$ so that $L_0=1$. From now on, we fix the values $p$ and $q_n$
satisfying these conditions; say, we put
\bee\label{pq}
q_n=3np,\qquad \qquad p=(9d)^{-1}.
\ene
Finally, we assume that $M>2$ and that
$\rho^p>R^{2\beta}$, where $\beta$ is the maximal possible value the exponent
$\al(n,m,l)$ from lemma \ref{angle2} can attain.
\bel\label{changedXi}
$\Xi_0(\R^d)=\emptyset$.
\enl
\bep
This statement is obvious since if $\GV=\R^d$, then for each $\bxi$ we have
$\bxi=\bxi_{\GV}$; therefore one cannot have a point $\bxi\in\CA$ with
$|\bxi_{\GV}|<L_d\le\rho^{1/3}$.
\enp
\bel\label{changedXi0}
Let $\GV\in\CV(n)$, $0\le n<d$, and $\bxi\in\Xi_1(\GV)$. Then $|\bxi_\GV|< 2L_n$.
\enl
\bep
The condition $\bxi\in\Xi_1(\GV)$ means that $\bxi\in\CA$ and there exists $\bxi'\in\CA$,
$|\bxi'_\GV|<L_n$ such that $\bxi-\bxi'\in\GV$. These conditions imply
\bees
||\BF\bxi_\GV|^2-|\BF\bxi'_\GV|^2|=||\BF\bxi|^2-|\BF\bxi'|^2|\ll 1.
\enes
Now the statement is obvious.
\enp
\bec\label{cor:changedXi0}
If $\bxi\in\Xi(\GV)$, then $|\bxi_\GV|\ll L_n$.
\enc
\bel\label{changedXi1}
Suppose, $\GV_1\in\CV(n_1)$ and $\GV_2\in\CV(n_2)$ are two subspaces such that neither of them is
contained in the other one. Let $\bxi_j\in \Xi_2(\GV_j)$. Then
$|\bxi_1-\bxi_2|> L_1$.
\enl
\bep
The conditions of lemma imply $|(\bxi_j)_{\GV_j}|\ll L_{n_j}$, $j=1,2$. Let $\GW=\GV_1+\GV_2$,
$\GU=\GV_1\cap\GV_2$, $\dim\GU=l$.
Then $\GW$ is an integer $6MR$-subspace, say $\GW\in\CV(m)$. Also, conditions of lemma
imply that $\GW\ne\GV_j$, so $m>n_j$. Suppose, the statement of lemma does not hold, i.e.
$|\bxi_1-\bxi_2|\le L_1$. Then $|(\bxi_1-\bxi_2)_{\GV_2}|\ll L_1$ and thus
$|(\bxi_1)_{\GV_2}|\ll  L_{n_2}$. By corollary \ref{angle2:cor}, the angle between
$\BF\GV_1$ and
$\BF\GV_2$ is at least $CR^{-\al(n_1,n_2,l)}$. Since the projections of $\bxi_1$ onto
$\GV_1$ and $\GV_2$ are smaller than $L_{n_1}$ and $2L_{n_2}$ respectively,
it is a simple geometry to deduce that
$|(\bxi_1)_\GW|\ll (L_{n_1}+L_{n_2})R^{\al(n_1,n_2,l)}$.
Due to the conditions stated before
lemma \ref{changedXi}, this implies $|(\bxi_1)_\GW|<L_m$. Therefore, $\bxi_1\in\Xi_1(\GW)$.
Now definition \eqref{defXi2} implies that $\bxi_1\not\in\Xi_2(\GV_1)$, which contradicts
our assumptions. Thus, $|\bxi_1-\bxi_2|> L_1$.
\enp
\bec\label{cor:changedXi1}
Suppose, $\GV_1\in\CV(n_1)$ and $\GV_2\in\CV(n_2)$ are two subspaces such that neither of them
is contained in the other one. Let $\bxi_j\in \Xi(\GV_j)$. Then
$|\bxi_1-\bxi_2|\gg L_1$.
\enc
\bel\label{changedwidth}
Let $\GV\in\CV(n)$ and $\bxi\in\Xi_3(\GV)$. Then $||\BF\bxi|^2-\rho^2|\ll KL_n$
and $||\BF\bxi^{\perp}_\GV|^2-\rho^2|\ll L_n^2$.
\enl
\bep
The assumption of lemma imply that there exists $\boldeta\in\Xi_2(\GV)$ such that
$\bxi-\boldeta\in \GV$ and $|\bxi-\boldeta|<K$. Lemma \ref{changedXi0}
implies $|\bxi_\GV|\ll L_n$, and thus
$||\BF\bxi|^2-|\BF\boldeta|^2|=||\BF\bxi_\GV|^2-|\BF\boldeta_\GV|^2|\ll KL_n$. The first
statement now follows from the fact that $\boldeta\in\CA$. Now we compute:
\bees
\rho^2-|\BF\bxi^{\perp}_\GV|^2=|\BF\bxi|^2+O(KL_n)-|\BF\bxi^{\perp}_\GV|^2=
|\BF\bxi_\GV|^2+O(KL_n)=O(L_n^2)
\enes
by corollary \ref{cor:changedXi0}.
\enp
Factorizing the LHS's of the estimates from this lemma, we immediately
obtain the following
\bec\label{cor:changedwidth}
Let $\GV\in\CV(n)$ and $\bxi\in\Xi_3(\GV)$. Then $||\BF\bxi|-\rho|\ll \rho^{p+q_n-1}$
and $||\BF\bxi^{\perp}_\GV|-\rho|\ll \rho^{2q_n-1}$.
\enc
\bel\label{changedXi4}
Let $\GV\in\CV(n)$ and $\bxi\in\Xi_3(\GV)$. Suppose, for some $\boldeta\in\CA$ we have
$\bxi-\boldeta\in\GV$. Then $\boldeta\in\Xi_2(\GV)$.
\enl
\bep
Definition \eqref{defXi1} implies that $\boldeta\in\Xi_1(\GV)$. Therefore, in order
to prove our lemma, we need to show that for any $\GW\in\CV(m)$ ($m>n$), $\GV\subset\GW$,
we have $\boldeta\not\in\Xi_1(\GW)$.
Suppose, this is not the case and  $\boldeta\in\Xi_1(\GW)$.
Then the fact that $\bxi\in\Xi_3(\GV)$ means that there exists
a vector $\tilde\bxi\in\Xi_2(\GV)$ with $\bxi-\tilde\bxi\in\GV$.
But then $\boldeta-\tilde\bxi\in\GV\subset\GW$. Therefore, $\tilde\bxi\in\Xi_1(\GW)$. This
contradicts the assumption $\tilde\bxi\in\Xi_2(\GV)$. The lemma is proved.
\enp
\bel\label{changedXi5}
Let $\GV\in\CV(n)$ and $\bxi\in\Xi_3(\GV)$. Suppose, for some $\boldeta\in\GV$ we have
$\bal:=\bxi+\boldeta\not\in\Xi_3(\GV)$. Then $||\BF\bal|^2-\rho^2|\gg K^2$.
\enl
\bep
Let $\tilde\bal$ be the point which satisfies the following conditions:
$\tilde\bal-\bal\in\GV$, $\tilde\bal\in\CA$, and the vector $\bal_\GV$
is a non-negative multiple of $\tilde\bal_\GV$ (a simple geometrical argument
shows that such a point always exists). Then lemma \ref{changedXi4} implies that
$\tilde\bal\in\Xi_2(\GV)$. Therefore, since $\bal\not\in\Xi_3(\GV)$, we have
$||\BF\tilde\bal_\GV|-|\BF\bal_\GV||=|\BF\tilde\bal_\GV-\BF\bal_\GV|\gg K$. Moreover,
\bees
\bes
&||\BF\bal|^2-|\BF\tilde\bal|^2|=||\BF\bal_\GV|^2-|\BF\tilde\bal_\GV|^2|\\&=
||\BF\bal_\GV|-|\BF\tilde\bal_\GV||\,||\BF\bal_\GV|+|\BF\tilde\bal_\GV||\ge
\bigl(|\BF\bal_\GV|-|\BF\tilde\bal_\GV|\bigr)^2\gg K^2.
\end{split}
\enes
This finishes the proof, since $||\BF\tilde\bal|^2-\rho^2|\ll 1$.
\enp
\bel\label{changedXi3}
Let $\GV\in\CV(n)$ and $\bxi\in\Xi_3(\GV)$. Suppose, $\bth\in\T_{6M}'$, $\bth\not\in\GV$.
Denote $\boldeta:=\bxi+\bth$. Then $||\BF\boldeta|^2-\rho^2|\gg K^2L_n$.
\enl
\bep
Let $\GW$ be the linear span of $\GV$ and $\bth$, and let
$\GU:=R(\bth)$ be the one-dimensional subspace.

Assume first that
$|\bxi_\GU|\le K^2L_n$. Then, since $|\bxi_\GV|\le L_n$,
the geometrical argument similar to the one used in the proof of lemma \ref{changedXi1}
implies that $|\bxi_\GW|<L_{n+1}/2$ (recall that the assumption
we have made on the exponents $p$ and $q_n$ imply that $L_{n+1}\ge K^3L_n$). Since $\bxi\in\Xi_3(\GV)$,
there exists a vector $\tilde\bxi\in\Xi_2(\GV)$, $|\bxi-\tilde\bxi|<K$. Therefore,
$|\tilde\bxi_\GW|\le |\bxi_\GW|+|(\tilde\bxi-\bxi)_\GW|<L_{n+1}$, which implies $\tilde\bxi\in\Xi_1(\GW)$.
This contradicts
the condition $\tilde\bxi\in\Xi_2(\GV)$.

Therefore, we must have $|\bxi_\GU|>K^2L_n$. This implies
\bees
||\BF\boldeta|^2-|\BF\bxi|^2|=||\BF(\bxi_\GU+\bth)|^2-|\BF\bxi_\GU|^2|\gg K^2L_n.
\enes
Now it remains to notice that lemma \ref{changedwidth} implies that
$||\BF\bxi|^2-\rho^2|\ll KL_n$. This finishes the proof.
\enp
\bec\label{cor:changedXi3}
Let $\GV\in\CV(n)$ and $\bxi\in\Xi_3(\GV)$. Suppose, $\bth\in\T_{6M}'$ and
$\boldeta=\bxi+\bth\not\in\Xi_3(\GV)$. Then $||\BF\boldeta|^2-\rho^2|\gg K^2$.
\enc
\bep
If $\bth\in\GV$, then the statement follows from lemma \ref{changedXi5}, and if
$\bth\not\in\GV$, the statement follows from lemma \ref{changedXi3}.
\enp

\bel\label{changedXi6}
For each two different integer subspaces $\GV_j\in\CV(n_j)$, $j=1,2$, $0\le n_j<d$
we have $(\Xi(\GV_1)+\T_1)\cap(\Xi(\GV_2)+\T_1)=\emptyset$.
\enl
\bep
Suppose, $\bxi\in(\Xi(\GV_1)+\T_1)\cap(\Xi(\GV_2)+\T_1)$.
Then corollary \ref{cor:changedXi1} implies that
one of the subspaces $\GV_j$ is inside the other, say $\GV_1\subset\GV_2$. Moreover,
there exist two points, $\bxi_1\in\Xi_3(\GV_1)$ and $\bxi_2\in\Xi_3(\GV_2)$ such that
$\bth_j:=\bxi_j-\bxi\in\T_{M+1}$. Then
$\bth:=\bxi_1-\bxi_2=\bth_1-\bth_2\in\T_{3M}$.

There are two possibilities: either $\bth\in\GV_2$, or $\bth\not\in\GV_2$.

Assume first that $\bth\in\GV_2$. Since
$\bxi_j\in\Xi_3(\GV_j)$, there exist points $\tilde\bxi_j\in\Xi_2(\GV_j)$ such that
$\tilde\bxi_j-\bxi_j\in\GV_j$, $|\tilde\bxi_j-\bxi_j|<K$.
But then $\tilde\bxi_1-\tilde\bxi_2\in\GV_2$. Since
$\tilde\bxi_2\in\Xi_2(\GV_2)\subset\Xi_1(\GV_2)$,
according to definition \eqref{defXi1} this means that $\tilde\bxi_1\in\Xi_1(\GV_2)$. Now
definition \eqref{defXi2} implies $\tilde\bxi_1\not\in\Xi_2(\GV_1)$ which contradicts our
assumption.

Assume now $\bth\not\in\GV_2$. Then lemma
\ref{changedXi3} implies
\bees
||\BF\bxi_1|^2-\rho^2|=||\BF(\bxi_2+\bth)|^2-\rho^2|\gg K^2 L_{n_2}.
\enes
However, this contradicts the inequality $||\BF\bxi_1|^2-\rho^2|\ll KL_{n_1}$ which
was established in lemma \ref{changedwidth}.
\enp

\bec\label{cor:changedXi6}
Each point $\bxi\in\CA$ belongs to precisely one of the sets $\Xi(\GV)$.
\enc
\bep
Indeed, definitions \eqref{defXi0}--\eqref{defXizero} imply that each point $\bxi\in\CA$
belongs to at least one of the sets $\Xi(\GV)$.
The rest follows from lemma \ref{changedXi6}.
\enp

Let us introduce more notation. Let
$\CC\subset\Rd$ be a measurable set.
We denote by $\CP^{(\bk)}(\CC)$ the orthogonal projection in $\GH =
L^2([0,2\pi]^d)$ onto the subspace spanned by the exponentials
$e_{\bxi}(\bx)$, $\bxi\in \CC$, $\{\bxi\}=\bk$.
\bel\label{potential}
For arbitrary set $\CC\subset\Rd$ and arbitrary $\bk$ we have:
\bee\label{eq:potential}
V'\CP^{(\bk)}(\CC)=\CP^{(\bk)}(\CC+\T_1)
V'\CP^{(\bk)}(\CC)
\ene
\enl
\bep This follows from the obvious
observation that if $\bxi=\bm+\bk\in\CC$ and $|\bn|\le R$, then
$\bxi+\bn\in\bigl(\CC+\T_1\bigr)$.
\enp

We are going to apply lemma \ref{perturbation2} and now we will
specify what are the projections $P^l_j$. The construction will be
the same for all values of quasi-momenta, so often we will skip
$\bk$ from the superscripts. For each $\GV\in\CV(n)$, $n=0,1,\dots,d-1$ we put
$P(\GV):=\CP^{(\bk)}\bigl(\Xi(\GV)\bigr)$. We also define
$P_j(\GV):=\CP^{(\bk)}\bigl((\Xi_3(\GV)+\T_j)\setminus
(\Xi_3(\GV)+\T_{j-1})\bigr)$, $j=1,\dots,M$, $P_0\bigl(\GV)=\CP^{(\bk)}(\Xi_3(\GV)\bigr)$.
We also denote $Q:=I-\bigl(\sum_{\GV}P(\GV)\bigr)$ (the sum is over all integer $6MR$-subspaces
of dimension $n=0,1,\dots,d-1$).
Now we apply lemma \ref{perturbation2} with the set of projections being $\{P(\GV)\}$,
$J:=[\la-20v,\la+20v]$, and $H_0=H_0(\bk)$. Let us check that all
the conditions of lemma \ref{perturbation2} are satisfied
assuming, as before, that all the conditions before lemma \ref{changedXi} are fulfilled.
Indeed, lemmas \ref{changedXi6} and \ref{potential} imply that
$P(\GV_1)P(\GV_2)=0$ and $P(\GV_1)V P(\GV_2)=0$
for different subsets $\GV_1$ and $\GV_2$ (in particular, $Q$ is also
a projection). Properties $P(\GV)=\sum_{j=0}^MP_j(\GV)$, $P_j(\GV)VP_l(\GV)=0$
for $|j-l|>1$ and $P_j(\GV)VQ=0$ for $j<M$ follow from the construction of the
projections $P_j(\GV)$ and lemma \ref{potential}.
Since $\CA\subset \cup_{\GV} \Xi_3(\GV)$, the distance between the spectrum of
$QH_0Q$ and $J$ is greater than $6v$. Corollary \ref{cor:changedXi3} implies
that the distances between the spectra of $P_j(\GV)H_0P_j(\GV)$, $j=1,\dots,M$
and $J$ are $\gg K^2$.
All these remarks imply that we can apply lemma
\ref{perturbation2} (or rather corollary \ref{perturbationnew}) and, instead of studying eigenvalues inside
$J$ of $H'(\bk)$, study eigenvalues of $\tilde
H(\bk):=\sum_{\GV}P(\GV)H'(\bk)P(\GV)$; the distance between any
eigenvalue of $H'(\bk)$ lying inside $J$ and the corresponding
eigenvalue of $\tilde H(\bk)$ is $\ll \rho^{-4Mp}$.

To be more precise, we do the following. Assume $\bxi=\bn+\bk\in\CA$. Then
$\bxi\in\Xi(\GV)$ for some uniquely defined $\GV\in\CV(n)$. In the following sections,
we will define a mapping $\tilde g:\bxi\mapsto\mu_{\tau(\bxi)}(P(\GV)H'(\bk)P(\GV))$,
where $\tau=\tau(\bxi)$ is a function with values in $\N$. The mapping $\tilde g$ will be
an injection and any eigenvalue of $P(\GV)H'(\bk)P(\GV)$ inside $J$ will have a pre-image under $\tilde g$.
Then, $\tilde g(\bxi)$ is also an eigenvalue
of $\sum_{\GV}P(\GV)H'(\bk)P(\GV)+QH'(\bk)Q$, say
\bees
\tilde g(\bxi)=\mu_{\tau_1(\bxi)}\Bigl(\sum_{\GV}P(\GV)H'(\bk)P(\GV)+QH'(\bk)Q\Bigr).
\enes
Then lemma \ref{perturbation2} implies that $|\tilde g(\bxi)-\mu_{\tau_1(\bxi)}(H'(\bk))|\ll \rho^{-4Mp}$.
We then define $f(\bxi):=\mu_{\tau_1(\bxi)}(H(\bk))$ so that $|\tilde g(\bxi)-f(\bxi)|\ll \rho^{-4Mp}$.
In order to construct the mapping $g$, we compute $\tilde g$ (or at least obtain an asymptotic formula for it)
and then, roughly speaking, throw away terms which are sufficiently small for our purposes.

In the next
two sections, we discuss how to obtain
an asymptotic formula for $\tilde g$ when $\rho$ is large. We will consider separately the
case $\bxi\in\Xi_2(\GU_0)=\CB$ (recall that $\GU_0=\{0\}\in\CV(0)$ and we have called $\CB$ the non-resonance region)
and the case of $\bxi$ lying inside the resonance region $\CD$. We start by looking at the case $\bxi\in\CB$.

\section{Computation of the eigenvalues outside resonance layers}
First of all, we notice that lemma \ref{changedXi3} implies that the
operator $P(\GU_0)H'(\bk)P(\GU_0)$ splits into the direct sum of operators.
Namely,
\bee
P(\GU_0)H'(\bk)P(\GU_0)=\bigoplus
\CP^{(\bk)}(\bxi+\T_M)H'(\bk)\CP^{(\bk)}(\bxi+\T_M),
\ene
the sum being over all
$\bxi\in\CB$ with $\{\bxi\}=\bk$. We denote by
$\tilde g(\bxi)$ the eigenvalue of
$\CP^{(\bk)}(\bxi+\T_M)H'(\bk)\CP^{(\bk)}(\bxi+\T_M)$ which lies within the distance
$v$ from $|\BF\bxi|^2$ (Lemma \ref{changedXi3} implies that this eigenvalue is unique).
Our next task is to compute
$\tilde g(\bxi)$. In this section we will prove the following lemma:
\bel\label{eigenvalues1} Let $R<\rho^{pd^{-1}/2}$.
Then the following asymptotic formula holds:
\bee\label{eq1:eigenvalues1}
\bes
&\tilde g(\bxi)\sim|\BF\bxi|^2\\
&+\sum_{r=1}^{\infty} \sum_{\boldeta_1,\dots,\boldeta_r\in \T'_M}
\sum_{n_1+\dots+n_r\ge 2}A_{n_1,\dots,n_r}
\lu\bxi,\BG\boldeta_1\ru^{-n_1}\dots\lu\bxi,\BG\boldeta_{r}\ru^{-n_r}
\end{split}
\ene
in the sense that for each $m\in\N$ we have
\bee\label{eq2:eigenvalues1}
\bes
&\tilde g(\bxi)-|\BF\bxi|^2\\
&-\sum_{r=1}^{m}\sum_{\boldeta_1,\dots,\boldeta_r\in \T'_M} \sum_{2\le
n_1+\dots+n_r\le m}A_{n_1,\dots,n_r}
\lu\bxi,\BG\boldeta_1\ru^{-n_1}\dots\lu\bxi,\BG\boldeta_{r}\ru^{-n_r}\\
&=O(\rho^{-(m+1)p}).
\end{split}
\ene
uniformly over $R<\rho^{pd^{-1}/2}$.
Here, $A_{n_1,\dots,n_p}$ is a polynomial of the Fourier
coefficients $\hat V(\boldeta_j)$ and $\hat
V(\boldeta_j-\boldeta_l)$ of the potential.
\enl
\bep
Let us
denote
\bee\label{aboldeta}
a(\boldeta)=|\BF(\bxi+\boldeta)|^2.
\ene
The matrix of $P(\bxi+\T_M)H'(\bk)P(\bxi+\T_M)$ has the following
form:

\bee\label{matrix1}
\begin{pmatrix}
a(0) &\hat V(\boldeta_1)&\hat V(\boldeta_2)& \dots&\hat
V(\boldeta_n)&\dots
\\
\overline{\hat V(\boldeta_1)}&a(\boldeta_1)&\hat
V(\boldeta_2-\boldeta_1)& \dots&\hat V(\boldeta_n-\boldeta_1)
&\dots\\
\overline{\hat V(\boldeta_2)}&\overline{\hat
V(\boldeta_2-\boldeta_1)}& a(\boldeta_2)& \dots&\hat
V(\boldeta_n-\boldeta_2) &\dots
\\
\vdots&\vdots&\vdots&\ddots&\vdots&\dots\\
\overline{\hat V(\boldeta_n)}&\overline{\hat
V(\boldeta_n-\boldeta_1)}&\overline{\hat
V(\boldeta_n-\boldeta_2)}&\dots&
a(\boldeta_n)&\dots\\
\vdots&\vdots&\vdots&\dots&\vdots&\ddots\\
\end{pmatrix}
\end{equation}

The diagonal elements of this matrix equal
$|\BF(\bxi+\boldeta)|^2$ (with $\boldeta$ running over $\T_M$) and
off-diagonal elements are Fourier coefficients of the potential
(and are thus bounded). Let $L$ be the number of columns of this matrix;
obviously, $L\asymp R^d$.

Let us compute the characteristic polynomial $p(\mu)$ of \eqref{matrix1}.
The definition of the determinant implies
\bee\label{char2n} p(\mu)=\Bigl(\prod_{\boldeta\in\T_M}
(a(\boldeta)-\mu)\Bigr)+\sum_{m=2}^L J_m(\mu),
\ene
where $J_m$ consists of products of exactly $(L-m)$ diagonal terms of
\eqref{matrix1} and $m$ off-diagonal terms. Put $J_m=J_m'+J_m''$,
where $J_m'$ (resp. $J_m''$) consists of all terms, not containing (resp. containing)
$(a(0)-\mu)$. Then we can re-write \eqref{char2n} as
\bee\label{char2} p(\mu)=\Bigl(\prod_{\boldeta\in\T'_M}
(a(\boldeta)-\mu)\Bigr)\Bigl((a(0)-\mu)+I(\mu) \Bigr),
\ene
where $I(\mu):=\sum_{m=1}^{L-1} I_m(\mu)+\sum_{m=2}^L\tilde I_m(\mu)$
with
\bees
I_m:=\frac{J_{m+1}'}{\prod_{\boldeta\in\T'_M}
(a(\boldeta)-\mu)}
\enes
and
\bees
\tilde I_m:=\frac{J_{m}''}{\prod_{\boldeta\in\T'_M}
(a(\boldeta)-\mu)}.
\enes
We can easily compute the first several terms:
\bee
I_1(\mu):=-\sum_{\boldeta\in \T'_M} \frac{|\hat
V(\boldeta)|^2}{a(\boldeta)-\mu},
\end{equation}
\bee
\bes I_2(\mu)&:=-\sum_{\boldeta,\boldeta'\in
\T'_M,\boldeta\ne\boldeta'}\frac{2\Re (\hat V(\boldeta) \hat
V(\boldeta-\boldeta')\overline{\hat
V(\boldeta')})}{(a(\boldeta)-\mu)(a(\boldeta')-\mu)},
\\
\tilde I_2(\mu)&:=-(a(0)-\mu)\sum_{\boldeta,\boldeta'\in
\T'_M,\boldeta\ne\boldeta'} \frac{|\hat
V(\boldeta-\boldeta')|^2}{(a(\boldeta)-\mu)(a(\boldeta')-\mu)}.
\end{split}
\ene
Overall, $I_m$ is the sum of $O(R^{dn})$ terms of the form
\bee
\frac{W_m(\boldeta_1,\dots, \boldeta_{n})}{(a(\boldeta_1)-\mu)\dots
(a(\boldeta_n)-\mu)},
\ene
and $\tilde I_m$ is the sum of $O(R^{dn})$ terms of the form
\bee
(a(0)-\mu)\frac{\tilde W_m(\boldeta_1,\dots, \boldeta_{n})}
{(a(\boldeta_1)-\mu)\dots (a(\boldeta_n)-\mu)}.
\ene
Here, $W_m(\boldeta_1,\dots, \boldeta_{n})$ and $\tilde
W_m(\boldeta_1,\dots, \boldeta_{n})$ are some polynomials of $\hat
V(\boldeta_j)$ and $\hat V(\boldeta_j-\boldeta_l)$.

On the interval $[a(0)-v,a(0)+v]$ the equation $p(\mu)=0$ has a
unique solution, which we have called $\tilde g(\bxi)$; this is
the solution of the equation $a(0)-\mu+I(\mu)=0$. After denoting
$F(\mu):=a(0)+I(\mu)$, this equation becomes equivalent to
$\mu=F(\mu)$. Throughout the rest of the section we will assume
that $\mu\in[a(0)-v,a(0)+v]$. Then, since $\bxi\in\CB$, lemma
\ref{changedXi3} guarantees that $|a(\boldeta)-a(0)|\gg \rho^{2p}$
for $\boldeta\in\T'_M$. This implies $I_n(\mu)=O(R^{dn}\rho^{-2np})=O(\rho^{-np})$;
similarly, $\tilde I_n(\mu)=O(\rho^{-np})$. Computing the derivatives, we see that
$\frac{d}{d\mu}I_n(\mu)$ and $\frac{d}{d\mu}\tilde I_n(\mu)$ are $O(\rho^{-np})$ as well.
Slightly more careful analysis shows that in fact
$I_1=O(R^d\rho^{-4p})=O(\rho^{-2p})$ and $\frac{d}{d\mu}I_1=O(R^d\rho^{-4p})=O(\rho^{-2p})$.
Indeed, we have:
\bee\label{I1}
\bes I_1(\mu)&=-\sum_{\boldeta\in \T'_M}
\frac{|\hat V(\boldeta)|^2}{a(\boldeta)-\mu}\\
&=-\frac{1}{2}\sum_{\boldeta\in
\T'_M}|\hat V(\boldeta)|^2(\frac{1}{a(\boldeta)-\mu}+
\frac{1}{a(-\boldeta)-\mu})\\
&=-\frac{1}{2}\sum_{\boldeta\in \T'_M}|\hat V(\boldeta)|^2
(\frac{a(\boldeta)+a(-\boldeta)-2\mu}{(a(\boldeta)-\mu)(a(-\boldeta)-\mu)})\\
&=-\sum_{\boldeta\in \T'_M}|\hat V(\boldeta)|^2
(\frac{a(0)+|\BF\boldeta|^2-\mu}{(a(\boldeta)-\mu)(a(-\boldeta)-\mu)}),
\end{split}
\end{equation}
and it remains to notice that $\sum_{\boldeta\in \T'_M}|\hat V(\boldeta)|^2$
is bounded by the square of the
$L_2$-norm of $V$.
These estimates show that when $R<\rho^{pd^{-1}/2}$, we have
$I(\mu)=O(\rho^{-2p})$ and
$\frac{d}{d\mu}F(\mu)=\frac{d}{d\mu}I(\mu)=O(\rho^{-2p})$. We will
find $\tilde\la(\bxi)$ using a sequence of approximations. We
define a sequence $\mu_k$ in the following way: $\mu_0=a(0)$,
$\mu_{k+1}=F(\mu_k)=a(0)+I(\mu_k)$. Since
$|\mu_{k+1}-\tilde g(\bxi)|= |F(\mu_{k})-F(\tilde g(\bxi))|=
|\mu_k-\tilde g(\bxi)|O(\rho^{-2p})$ and
$|\mu_0-\tilde g(\bxi)|=O(1)$, we have:
\bee\label{approximation1}
|\mu_k-\tilde g(\bxi)|=O(\rho^{-2kp}).
\ene
Therefore, we will
prove the lemma if we show that for all $k\ge 1$ the approximation
$\mu_k$ enjoys the same asymptotic behaviour
\eqref{eq1:eigenvalues1}, at least up to an error
$O(\rho^{-kp})$. This computation is straightforward. For
example, we have
\bees
\mu_1=|\BF\bxi|^2+I_1(|\BF\bxi|^2)+I_2(|\BF\bxi|^2)+\tilde
I_2(|\BF\bxi|^2)+O(\rho^{-3p}),
\enes
and, using \eqref{I1}, we
obtain:
\bee\label{new11}
\bes
I_1(a(0))&=-\sum_{\boldeta\in
\T'_M}|V(\boldeta)|^2
(\frac{|\BF\boldeta|^2}{(a(\boldeta)-a(0))(a(-\boldeta)-a(0))})\\
&=\sum_{\boldeta\in \T'_M}|V(\boldeta)|^2
(\frac{|\BF\boldeta|^2}{(2\lu\BF\bxi,\BF\boldeta\ru+|\BF\boldeta|^2)
(2\lu\BF\bxi,\BF\boldeta\ru-|\BF\boldeta|^2)})\\
&=\sum_{\boldeta\in \T'_M}|V(\boldeta)|^2
(\frac{|\BF\boldeta|^2}{4\lu\BF\bxi,\BF\boldeta\ru^2-|\BF\boldeta|^4})\\
&=\sum_{\boldeta\in \T'_M}|V(\boldeta)|^2
(\frac{4^{-1}|\BF\boldeta|^2\lu\BF\bxi,\BF\boldeta\ru^{-2}}
{1-4^{-1}|\BF\boldeta|^4\lu\BF\bxi,\BF\boldeta\ru^{-2}})\\
&=\sum_{\boldeta\in \T'_M}|V(\boldeta)|^2
\sum_{n=1}^{\infty}4^{-n}|\BF\boldeta|^{4n-2} \lu\BF\bxi,\BF\boldeta\ru^{-2n}\\
&=\sum_{\boldeta\in \T'_M}|V(\boldeta)|^2
\sum_{n=1}^{\infty}4^{-n}|\BF\boldeta|^{4n-2}
\lu\bxi,\BG\boldeta\ru^{-2n}.
\end{split}
\ene
Computations of $I_2(a(0))$ are similar (and, obviously, $\tilde
I_2(a(0))=0$), only now the result will have terms which involve
inner products of $\xi$ with two different $\boldeta$'s. Thus,
\bee\label{mu1}
\mu_1=|\BF\bxi|^2+\sum_{\boldeta_1,\boldeta_2\in
\T'_M} \sum_{n_1,n_2}A_{n_1,n_2}
\lu\bxi,\BG\boldeta_1\ru^{-n_1}\lu\bxi,\BG\boldeta_2\ru^{-n_2}+O(\rho^{-3p}),
\ene
the sum being over all $n_1$, $n_2$ with $n_1+n_2\ge 2$ (in
fact, we can take the sum over $n_1+n_2=2$, since other terms will
be $O(\rho^{-3p})$). Using induction, it is easy to prove now
that
\bee\label{muk}
\bes
\mu_k&=|\BF\bxi|^2\\
&+\sum_{r=1}^{k+1}\sum_{\boldeta_1,\dots,\boldeta_{r}\in \T'_M}
\sum_{n_1,\dots,n_{r}}A_{n_1,\dots,n_{r}}
\lu\bxi,\BG\boldeta_1\ru^{-n_1}\dots\lu\bxi,\BG\boldeta_{r}\ru^{-n_{r}}\\
&+O(\rho^{-(k+2)p}),
\end{split}
\ene
the sum being over $2\le \sum_{j=1}^{r} n_j\le k+1$;
$A_{n_1,\dots,n_p}$ is a polynomial of $\{\hat V(\boldeta_j)\}$
and $\{\hat V(\boldeta_j-\boldeta_l)\}$. Indeed, if $\mu_k$ satisfies \eqref{muk},
then a calculation similar to \eqref{new11} shows that for each $\boldeta\in\T'_M$
the fraction $\frac{1}{a(\boldeta)-\mu_k}$ can be decomposed as a sum of products of
negative powers of $\lu\bxi,\BG\boldeta_j\ru$. Therefore, all functions $I_n(\mu_k)$
(and, thus, $I(\mu_k)$) admit similar decomposition. This implies that
the next approximation $\mu_{k+1}=|\BF\bxi|^2+I(\mu_k)$ also satisfies \eqref{muk}.

Estimate \eqref{approximation1} now shows that the asymptotic
formula \eqref{eq2:eigenvalues1} holds.
\enp
We now define $g(\bxi)$ as the finite part of the RHS of the expansion \eqref{eq1:eigenvalues1}, namely
\bee\label{tildegnew}
\bes
&g(\bxi)=|\BF\bxi|^2\\
&+\sum_{r=1}^{4M} \sum_{\boldeta_1,\dots,\boldeta_r\in \T'_M}
\sum_{n_1+\dots+n_r\ge 2}A_{n_1,\dots,n_r}
\lu\bxi,\BG\boldeta_1\ru^{-n_1}\dots\lu\bxi,\BG\boldeta_{r}\ru^{-n_r}.
\end{split}
\ene
\bel
We have:
\bee\label{gtildeg}
|g(\bxi)-\tilde g(\bxi)|\ll\rho^{-4Mp}.
\ene
\enl
\bep
This follows from lemma \ref{eigenvalues1}.
\enp

\section{Computation of the eigenvalues inside resonance layers}


Now let us fix $\GV\in\CV(n)$, $1\le n\le d-1$, and try to study the eigenvalues of
$P(\GV)H'(\bk)P(\GV)$.
Let $\bxi=\bn+\bk\in\Xi_3(\GV)$. We denote
\bee\label{newBUpsj}
\BUps_j=\BUps_j(\bxi):=\bigl((\bxi+(\GV\cap\Z^d))
\cap\Xi_j(\GV)\bigr),\qquad j=0,1,2,3,
\ene
\bee\label{newBUps}
\BUps=\BUps(\bxi):=
\BUps_3(\bxi)+\T_M,
\ene
\bee\label{Pbxi}
P(\bxi):=\CP^{(\bk)}\bigl(\BUps(\bxi)\bigr),
\ene
\bee\label{Hbxi}
H'(\bxi):= P(\bxi)H'(\bk)P(\bxi),
\ene
\bee\label{H0bxi}
H_0(\bxi):=P(\bxi)H_0(\bk)P(\bxi),
\ene
and
\bee\label{Vbxi}
V'_{\bxi}:=P(\bxi)V'P(\bxi).
\ene
\begin{figure}[!hbt]
\begin{center}
\framebox[0.8\textwidth]{\includegraphics[width=0.60\textwidth]{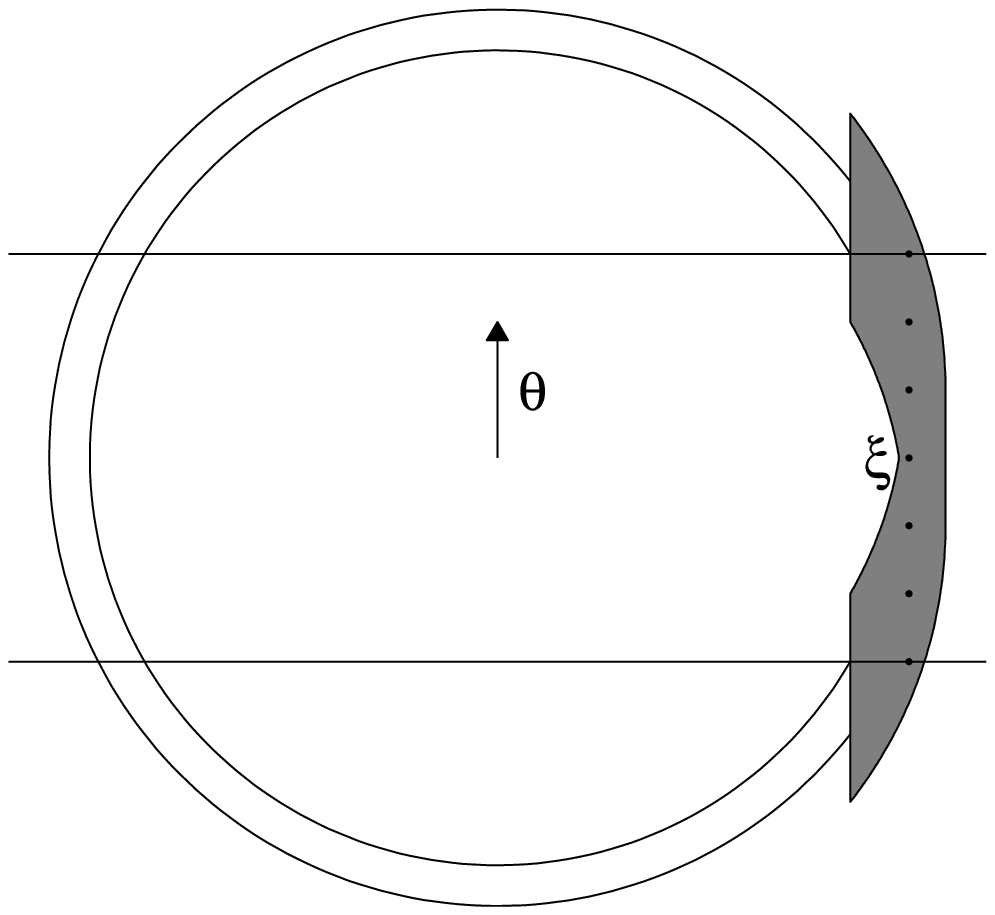}}
\caption{The sets $\Xi_3(\bth)$ and $\BUps_3(\bxi)$ \label{fig:5}}
\end{center}
\end{figure}
\begin{figure}[!hbt]
\begin{center}
\framebox[0.8\textwidth]{\includegraphics[width=0.60\textwidth]{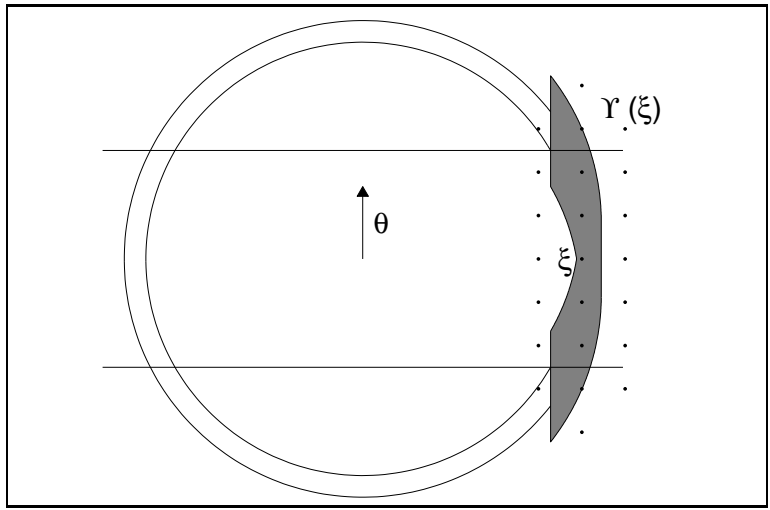}}
\caption{The sets $\Xi_3(\bth)$ and $\BUps(\bxi)$; here,
$\T=\{(0,0),(\pm 1,0),(0,\pm 1)\}$ consists of five elements. \label{fig:6}}
\end{center}
\end{figure}
Out of all sets denoted by the letter $\BUps$, we will mostly use $\BUps_3(\bxi)$ and $\BUps(\bxi)$; see
Figures \ref{fig:5}-\ref{fig:6}  for an illustration of these sets when $d=2$.
Let us establish some simple properties of these sets.
\bel\label{properties:BUps1}
Suppose, $\boldeta\in\BUps(\bxi)\setminus\BUps_3(\bxi)$. Then $||\BF\boldeta|^2-\la|\gg K^2$ (in particular,
$\boldeta\not\in\CA$).
\enl
\bep
The assumptions of the lemma imply that $\boldeta=\tilde\bxi+\bth$ with $\tilde\bxi\in\BUps_3
(\bxi)$ and $\bth\in\T_M$.
If $\bth\not\in\GV$, the statement follows from lemma \ref{changedXi3}. Assume $\bth\in\GV$.
Then $\boldeta\not\in\Xi_3(\GV)$
(otherwise we had $\boldeta\in\BUps_3(\bxi)$). Now the statement follows from lemma \ref{changedXi5}.
\enp
\bel\label{properties:BUps2}
We have $\BUps_3(\bxi)\subset\Xi_3(\GV)$ and $\BUps(\bxi)\subset\Xi(\GV)$. If $\boldeta\in\BUps(\bxi)$,
then $\boldeta-\bxi\in\Z^d$.
If for some $\bxi_1,\bxi_2\in\Xi_3(\GV)$ we have $\BUps(\bxi_1)\cap\BUps(\bxi_2)\ne\emptyset$,
then $\BUps(\bxi_1)=\BUps(\bxi_2)$.
\enl
\bep
The first three statements follow immediately from the definitions.
Assume $\BUps(\bxi_1)\cap\BUps(\bxi_2)\ne\emptyset$, say
$\boldeta\in\BUps(\bxi_1)\cap\BUps(\bxi_2)$. 
Then $\boldeta=\tilde\bxi_j+\bth_j$, $j=1,2$, with $\tilde\bxi_j\in\Xi_3(\GV)$ and $\bth_j\in\T_M$.
Then $\tilde\bxi_1=\tilde\bxi_2+(\bth_2-\bth_1)$. Since $\bth_2-\bth_1\in\T_{2M}$,
lemmas \ref{changedXi3} and \ref{changedwidth}
imply that $\bth_2-\bth_1\in\GV$. Therefore, $\bxi_2-\bxi_1\in(\GV\cap\Z^d)$, so $\BUps(\bxi_1)=\BUps(\bxi_2)$.
\enp
Lemma \ref{properties:BUps2} implies that the operator $P(\GV)
H'(\bk)P(\GV)$ splits into the
direct sum:
\bee\label{directsum61}
P(\GV) H'(\bk)P(\GV)=\bigoplus
H'(\bxi),
\ene
the sum being over all classes of equivalence of
$\bxi\in\Xi_3(\GV)$ with $\{\bxi\}=\bk$. Two vectors $\bxi_1$ and $\bxi_2$
are equivalent if $\BUps(\bxi_1)=\BUps(\bxi_2)$.

\ber\label{newremark}
The programme formulated at the end of section 5 requires to put into correspondence to each point
$\bxi\in\Xi_3(\GV)$ a number $\tilde g(\bxi)$ which is an eigenvalue of $P(\GV)H'(\bk)P(\GV)$. It is natural to choose
$\tilde g(\bxi)$ to be an eigenvalue of $H'(\bxi)$, say $\tilde g(\bxi)=\mu_j(H'(\bxi))$, where $j=j(\bxi)$ is some
natural number, and the mapping $j:\BUps(\bxi)\to\N$ is (at least) an injection. There are certain technical problems
with defining the function $j$. The first problem is that the sets $\BUps(\bxi_1)$ and $\BUps(\bxi_2)$ can have different
number of elements for different $\bxi_1,\bxi_2\in\Xi_3(\GV)$ (as Figure \ref{fig:7} illustrates), and the mapping $j$
obviously has to
take care of this fact. The second problem is that the mapping $j$ cannot possibly be continuous (otherwise,
since it takes only
natural values, it would be constant
and therefore not an injection), so $\tilde g$ as well cannot be continuous. Finally, we want $\tilde g(\bxi)$
not to change too much when
we change $\bxi$ a little. We cannot exactly achieve this (since, as we mentioned above, $\tilde g$ must be
discontinuous),
but we can achieve some weaker version of this (see lemma \ref{newpartialofs} for the precise statement).
\enr
\begin{figure}[!hbt]
\begin{center}
\framebox[0.8\textwidth]{\includegraphics[width=0.60\textwidth]{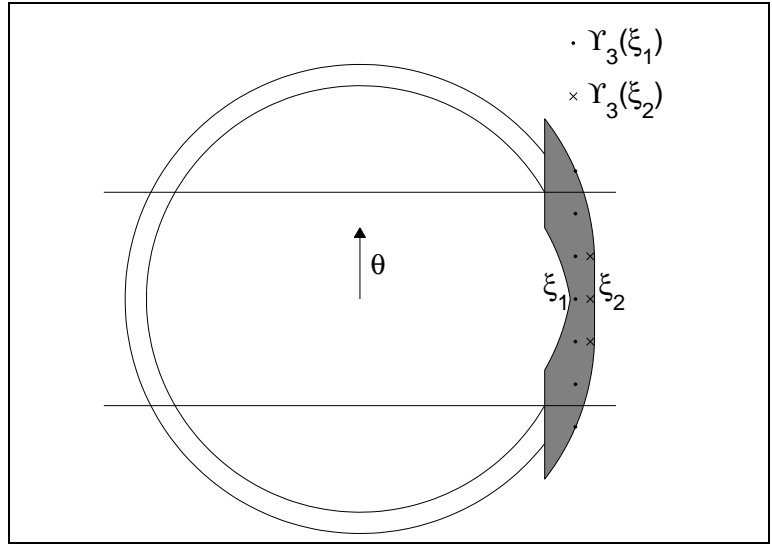}}
\caption{The sets $\BUps_3(\bxi_1)$ (dots) and $\BUps_3(\bxi_2)$ (crosses)
have different number of elements\label{fig:7}}
\end{center}
\end{figure}

Hence, we will study
operators $H'(\bk)$ for each
$\bxi=\bm+\bk\in\Xi(\GV)$ with $\{\bxi\}=\bk$. Recall that we have
denoted by $\bxi_\GV$ and $\bxi_\GV^{\perp}$ vectors such that
$\bxi=\bxi_\GV+\bxi_\GV^{\perp}$, $\bxi_\GV\in \GV$,
$\BG\bxi_\GV^{\perp}\perp \GV$. Let us also define
\bee\label{rbxi}
r=r(\bxi):=|\BF\bxi_\GV^{\perp}|,\,\bxi'_\GV:=\bxi_\GV^{\perp}/r,
\ene
so that $|\BF\bxi'_{\GV}|=1$. 
We can think of the triple $(r,\bxi'_\GV,\bxi_\GV)$ as the
cylindrical coordinates on $\Xi(\GV)$.
Corollary \ref{cor:changedXi0} implies that
$|\bxi_\GV|\ll\rho^{q_n}$; corollary \ref{cor:changedwidth} implies
\bee\label{rrho}
|r-\rho|=O(\rho^{2q_n-1})=O(\rho^{-1/3}),
\ene
since $q_n\le 1/3$; in particular, we have $r>0$.
The current objective is to
express the asymptotic behaviour of eigenvalues of $H'(\bxi)$ inside $J$ in terms
of $r$. In order to do this, we want to compare the eigenvalues of $H'(\bxi_1)$ and
$H'(\bxi_2)$ when $\bxi_1,\bxi_2\in\Xi(\GV)$ are two points which are 
close to each other.
Since the operators $H'(\bxi_1)$ and $H'(\bxi_2)$ act in different Hilbert
spaces $P(\bxi_j)\GH$, we first need to map these Hilbert spaces onto each other.
A natural idea is to employ the
mapping $F_{\bxi_1,\bxi_2}: P(\bxi_1)\GH\to P(\bxi_2)\GH$
defined in the following way:
\bee\label{mappingF}
F_{\bxi_1,\bxi_2}(e_{\boldeta})=e_{\boldeta+\bxi_2-\bxi_1}.
\ene
This mapping is `almost'
an isometry, except for the fact that it is not well-defined, i.e. it
could happen for example that $\boldeta\in \BUps(\bxi_1)$,
but $(\boldeta+\bxi_2-\bxi_1)\not\in \BUps(\bxi_2)$ (Figure \ref{fig:7} illustrates how this can happen).
In order to avoid this, we will
extend the sets $\BUps(\bxi)$. We do this in the following way. First, for $\bxi_1,\bxi_2\in\Xi_2(\GV)$ we define
\bee\label{bupsxi12}
\BUps(\bxi_1;\bxi_2):=\BUps(\bxi_1)\cup\bigl(\BUps(\bxi_2)-\bxi_2+\bxi_1\bigr)
\ene
and, similarly,
\bees
\BUps_3(\bxi_1;\bxi_2):=\BUps_3(\bxi_1)\cup\bigl(\BUps_3(\bxi_2)-\bxi_2+\bxi_1\bigr)
\enes
(the set $\BUps_3(\bxi_2;\bxi_1)$ is shown on Figure \ref{fig:8}).
We also define
\bee\label{Pbxi12}
P(\bxi_1;\bxi_2):=\CP^{(\bk_1)}\bigl(\BUps(\bxi_1;\bxi_2)\bigr),
\ene
where $\bk_1:=\{\bxi_1\}$;
\bee\label{Hbxi12}
H'(\bxi_1;\bxi_2):= P(\bxi_1;\bxi_2)H'(\bxi_1)P(\bxi_1;\bxi_2),
\ene
and
\bee\label{H0bxi12}
H_0(\bxi_1;\bxi_2):= P(\bxi_1;\bxi_2)H_0(\bxi_1)P(\bxi_1;\bxi_2).
\ene
\begin{figure}[!hbt]
\begin{center}
\framebox[0.8\textwidth]{\includegraphics[width=0.60\textwidth]{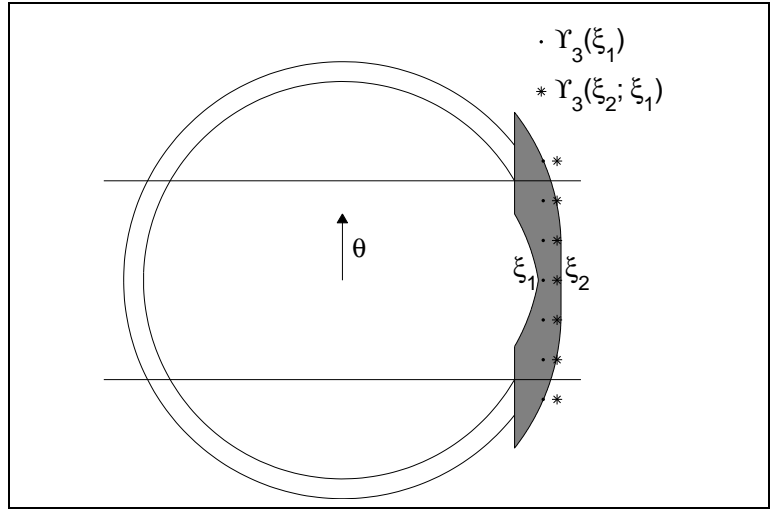}}
\caption{Now the sets $\BUps_3(\bxi_1;\bxi_2)=\BUps_3(\bxi_1)$ (dots) and $\BUps_3(\bxi_2;\bxi_1)$ (stars)
have the same number of elements\label{fig:8}}
\end{center}
\end{figure}

Suppose also that $\bxi\in\Xi_2(\GV)$ and let $U\subset\Xi_2(\GV)$ be a set containing $\bxi$ of
diameter $\ll \rho^{-1}$. Denote
\bee\label{bupsxiu}
\BUps(\bxi;U):=\cup_{\boldeta\in U}\BUps(\bxi;\boldeta),
\ene
\bees
\BUps_3(\bxi;U):=\cup_{\boldeta\in U}\BUps_3(\bxi;\boldeta),
\enes
\bees
P(\bxi;U)=\CP^{(\bk)}\bigl(\BUps(\bxi;U)\bigr),
\enes
\bee\label{Hbxi12U}
H'(\bxi;U):= P(\bxi;U)H'(\bxi)P(\bxi;U),
\ene
and
\bee\label{H0bxi12U}
H_0(\bxi;U):= P(\bxi;U)H_0(\bxi)P(\bxi;U).
\ene
Notice that $\BUps(\bxi_1;\bxi_2)=\BUps(\bxi_1;\{\bxi_1,\bxi_2\})$.

Now if we define the mapping
$F_{\bxi_1,\bxi_2}:\, P(\bxi_1;\bxi_2)\GH\to P(\bxi_2;\bxi_1)\GH$ by formula
\eqref{mappingF}, this mapping will be a bijection and an isometry, since obviously
\bees
\BUps(\bxi_2;\bxi_1)=\BUps(\bxi_1;\bxi_2)+\bxi_2-\bxi_1.
\enes
Similarly, if $U$ is any set containing $\bxi_1$ and $\bxi_2$, then the mapping
$F_{\bxi_1;U}:\, P(\bxi_1;U)\GH\to P(\bxi_2;U)\GH$ defined by
\eqref{mappingF} will be a bijection and an isometry, since
\bees
\BUps(\bxi_2;U)=\BUps(\bxi_1;U)+\bxi_2-\bxi_1.
\enes
Note also that if $\boldeta\in\BUps(\bxi;U)$, then $\boldeta-\bxi\in\Z^d$.

The problem, of course, is that in general the spectra of $H'(\bxi_1)$ and
$H'(\bxi_1;\bxi_2)$ (or $H'(\bxi_1;U)$) can be quite far from each other. However, we can
give sufficient
conditions which guarantee that the spectra of $H'(\bxi_1)$ and
$H'(\bxi_1;\bxi_2)$ (or rather the parts of the spectra lying inside $J$) are
within a small distance (of order $O(\rho^{-4Mp})$) from each other.
The following statement is a straightforward corollary of lemma \ref{perturbation2}.

\bel\label{newisometry}
a) Let $\bxi_1,\bxi_2\in \Xi_2(\GV)\subset\CA$ satisfy
%
%
%
$|\bxi_1-\bxi_2|<\rho^{-1}$.
Then there exists a bijection $G=G_{\bxi_1,\bxi_2}$
defined on a subset of the set of all eigenvalues of
$H'(\bxi_1)$ and mapping them to a subset of the set of all eigenvalues of
$H'(\bxi_1;\bxi_2)$ (eigenvalues in both sets are counted including multiplicities) satisfying the following properties:

(i) all eigenvalues of $H'(\bxi_1)$ (resp. $H'(\bxi_1;\bxi_2)$)
inside $J$ are in the domain (resp. range) of $G_{\bxi_1,\bxi_2}$;

(ii) for any eigenvalue $\mu_j(H'(\bxi_1))\in J$ (and thus in the domain of
$G_{\bxi_1,\bxi_2}$) we have:
\bee\label{bijectionofla}
|\mu_j(H'(\bxi_1))-G(\mu_j(H'(\bxi_1)))|\ll \rho^{-4Mp}.
\ene

b) Suppose $\bxi\in U\subset \Xi_2(\GV)$ and the diameter of $U$ is $\ll\rho^{-1}$.
Then there exists a bijection $G=G_{\bxi,U}$
defined on a subset of the set of all eigenvalues of
$H'(\bxi)$ and mapping them to a subset of the set of all eigenvalues of
$H'(\bxi;U)$ (eigenvalues in both sets are counted including multiplicities) satisfying the following properties:

(i) all eigenvalues of $H'(\bxi)$ (resp. $H'(\bxi;U)$)
inside $J$ are in the domain (resp. range) of $G_{\bxi,U}$;

(ii) for any eigenvalue $\mu_j(H'(\bxi))\in J$ (and thus in the domain of
$G_{\bxi,U}$) we have:
\bee\label{bijectionoflaU}
|\mu_j(H'(\bxi))-G(\mu_j(H'(\bxi)))|\ll \rho^{-4Mp}.
\ene

\enl
\bep

Let us prove part a) of this lemma; part b) is proved analogously.
Suppose, $\bxi\in(\BUps(\bxi_1;\bxi_2)\setminus\BUps(\bxi_1))$.
Let us prove that then
\bee\label{newBUps1}
||\BF\bxi|^2-\rho^2|\gg K^2.
\ene
Indeed, we obviously have
$\tilde\bxi:=\bxi+\bxi_2-\bxi_1\in\BUps(\bxi_2)\subset\Xi(\GV)$. Then definitions
\eqref{defXi3}, \eqref{defXi}, \eqref{newBUpsj} and \eqref{newBUps} 
imply that $\tilde\bxi=\tilde\boldeta+\ba+\bth$, $\boldeta\in\bigl(\Xi_2(\GV)\cap(\bxi_2+\GV)\bigr)$,
$\ba\in B(\GV,K)$, $\bth\in\T_M$.
If $\bth\not\in\GV$,
\eqref{newBUps1} follows from lemma \ref{changedXi3} and the inequality
\bees
||\BF\bxi|^2-|\BF\tilde\bxi|^2|\ll 1,
\enes
which in turn follows from the conditions of lemma.
Suppose $\bth\in\GV$. Then $\bxi-\bxi_1=\ba+\bth+(\boldeta-\bxi_2)\in\GV$ and, since
$\bxi\not\in\BUps(\bxi_1)$, we have $\bxi\not\in\BUps_3(\bxi_1)$, which in turn implies $\bxi\not\in\Xi_3(\GV)$.
Now \eqref{newBUps1} follows
from lemma \ref{changedXi5}.

Inequality \eqref{newBUps1} shows that
as before we can apply lemma \ref{perturbation2}, or rather its corollary \ref{perturbationnew}.
This time, we apply this lemma
with $H=H'(\bxi_1;\bxi_2)$,  $H_0=H_0(\bxi_1;\bxi_2)$, $n=0$,
$P^0=P(\bxi_1)$, $Q=P(\bxi_1;\bxi_2)-P(\bxi_1)$, $P_0^0=\CP^{(\bk_1)}(\BUps_3(\bxi_1))$,
$P_j^0(\GV):=P(\bxi_1;\bxi_2)\CP^{(\bk_1)}\bigl((\BUps_3(\bxi_1)+\T_j)\setminus
(\BUps_3(\bxi_1)+\T_{j-1})\bigr)P(\bxi_1;\bxi_2)$, $j=1,\dots,M$. The fulfillment of all conditions of
lemma \ref{perturbation2} follows from \eqref{newBUps1} and lemma \ref{potential}. Now the
statement of lemma immediately follows from corollary \ref{perturbationnew}.
\enp

\ber\label{rem:newisometry0}
Part (iii) of corollary \ref{perturbationnew} shows that the
bijection $G_{\bxi_1,\bxi_2}$ is given by the following formula.
Let $l=l(\bxi_1,\bxi_2)$ be the number of points
\bee\label{newsetminus}
\boldeta\in \BUps(\bxi_1;\bxi_2)\setminus\BUps(\bxi_1)
\ene
with $|\BF\boldeta|^2<\la$ (notice that if $\boldeta$ satisfies \eqref{newsetminus},
then $|\BF\boldeta|\not\in J$).
Then
$G_{\bxi_1,\bxi_2}(\mu_j(H'(\bxi_1)))=\mu_{j+l}(H'(\bxi_1;\bxi_2)))$.
Similarly, if $l=l(\bxi,U)$ is the number of points
\bee\label{newsetminusU}
\boldeta\in \BUps(\bxi;U)\setminus\BUps(\bxi)
\ene
with $|\BF\boldeta|^2<\la$, then
$G_{\bxi,U}(\mu_j(H'(\bxi)))=\mu_{j+l}(H'(\bxi;U)))$.
\enr

The next lemma shows that the eigenvalues of $H'(\bxi;U)$ do not change much if we increase $U$;
this lemma is an immediate corollary of lemma \ref{newisometry}.

\bel\label{newU1}
Let $\bxi\in U_1\subset U_2\subset\Xi_2(\GV)$ and let the diameter of $U_2$ be $\ll\rho^{-1}$.
Denote by $l=l(\bxi;U_1,U_2)$ the number of points
\bee\label{newsetminusU12}
\boldeta\in \BUps(\bxi;U_2)\setminus\BUps(\bxi;U_1)
\ene
with $|\BF\boldeta|^2<\la$. Then:

a) for any eigenvalue $\mu_j(H'(\bxi;U_1))\in J$ we have:
\bee\label{bijectionoflaU1}
|\mu_j(H'(\bxi;U_1))-\mu_{j+l}(H'(\bxi;U_2))|\ll \rho^{-4Mp};
\ene
b) the number $l(\bxi;U_1,U_2)$ does not depend on $\bxi$, i.e. if $\bxi_1,\bxi_2\in U_1$, then
$l(\bxi_1;U_1,U_2)=l(\bxi_2;U_1,U_2)$.
\enl
\bep
Part a) of lemma follows from lemma \ref{newisometry} and remark \ref{rem:newisometry0}, since
$l(\bxi;U_1,U_2)=l_2-l_1$, $l_j:=l(\bxi;U_j)$, and we have
\bee
|\mu_j(H'(\bxi))-\mu_{j+l_j}(H'(\bxi;U_j))|\ll \rho^{-4Mp},\qquad j=1,2.
\ene
Let us prove part b). Suppose, $\boldeta_1\in\BUps(\bxi_1;U_2)\setminus\BUps(\bxi_1;U_1)$.
Then, in the same way as we have proved \eqref{newBUps1}, we can show that $||\BF\boldeta_1|^2-\la|\gg K^2$.
Denote $\boldeta_2:=\boldeta_1+(\bxi_2-\bxi_1)$.
The definitions of the sets $\BUps$ imply that
$\boldeta_2\in\BUps(\bxi_2;U_2)\setminus\BUps(\bxi_2;U_1)$.
Since $|\bxi_2-\bxi_1|\ll\rho^{-1}$, we have
$||\BF\boldeta_2|^2-|\BF\boldeta_1|^2|\ll 1$.
Therefore, the inequality $|\BF\boldeta_1|^2<\la$ is satisfied if and only if the inequality
$|\BF\boldeta_2|^2<\la$ is satisfied. This proves that $l(\bxi_1;U_1,U_2)=l(\bxi_2;U_1,U_2)$.
\enp


As we have already mentioned, if $\bxi,\boldeta\in U$, we have $\BUps(\boldeta;U)=\BUps(\bxi;U)+(\boldeta-\bxi)$,
which implies that the
mapping $F_{\bxi,\boldeta}:P(\bxi;U)\GH\to P(\boldeta;U)\GH$ defined by \eqref{mappingF} is an isometry.
Thus, by considering
the sets $\BUps(\bxi;U)$ instead of $\BUps(\bxi)$ we have overcome the first difficulty mentioned in
remark \ref{newremark}.
Now we will try to face the other problems mentioned there.

Let
$\boldeta_0=0,\boldeta_1,\dots,\boldeta_p$
be the complete system
of representatives of $\T_M$ modulo $\GV$ (we assume of course that $\boldeta_j\in\T_M$). That means that
each vector $\bth\in\T_M$ has
a unique representation $\bth=\boldeta_j+\ba$, $\ba\in\GV$.
Denote 
$\BPsi_j=\BPsi_j(\bxi):=\bigl(\bxi+\boldeta_j+(\GV\cap\Z^d)\bigr)\cap\BUps(\bxi)$.
Then
\bee
\BUps(\bxi)=\bigcup_j \BPsi_j,
\ene
and this is a disjoint union (on Figure \ref{fig:6}, the set $\BPsi_0$ is the middle column of dots, and
$\BPsi_1$ and $\BPsi_2$ are the left and right columns).

Let us compute diagonal elements of $H'(\bxi)$. Let $\boldeta\in\BUps(\bxi)$.
Then $\boldeta$ can be uniquely decomposed as
\bee\label{etadecomposition}
\boldeta=\bxi+\bmu+\boldeta_j
\ene
with $\bmu\in\GV\cap\Z^d$. 
Recall that $H'(\bxi)=H_0(\bxi)+V'_{\bxi}$ and
$H_0(\bxi)e_{\boldeta}=|\BF\boldeta|^2e_{\boldeta}$ whenever $\boldeta\in\BUps(\bxi)$. 
Since $\BF\bxi_{\GV}'\perp\BF\bmu$ and $\BF\bxi_{\GV}'\perp\BF\bxi_{\GV}$,
we have:
\bee\label{AB}
\bes
|\BF\boldeta|^2&=
|\BF(\bxi+\boldeta_j+\bmu)|^2=
|\BF(r\bxi'_\GV+\bxi_\GV+\boldeta_j+\bmu)|^2\\
&=r^2+
2\lu\BF\bxi'_\GV,\BF\boldeta_j\ru r+
|\BF((\bxi+\boldeta_j)_\GV+\bmu)|^2.
\end{split}
\ene
This simple computation
implies that
\bee\label{newpencil}
H'(\bxi)=r^2I+rA+B.
\ene
Here, $A=A(\bxi)=A(\bxi_\GV,\bxi'_\GV,r)$ and
$B=B(\bxi)=B(\bxi_\GV,\bxi'_\GV,r)$ are self-adjoint operators acting in $P(\bxi)\GH$
 in the following way:
\bees
A=2\sum_{j=0}^p\lu\BF\bxi'_\GV,\BF\boldeta_j\ru\CP^{(\bk)}(\BPsi_j);
\enes
in other words,
\bee\label{operatorA}
Ae_{\boldeta}=2\lu\BF\bxi'_\GV,\BF\boldeta_j\ru e_{\boldeta}
=2\lu\BF\bxi'_\GV,\BF(\boldeta-\bxi)\ru e_{\boldeta},
\ene
and
\bee\label{operatorB}
Be_{\boldeta}=|\BF((\bxi+\boldeta_j)_\GV+\bmu)|^2e_{\boldeta}+V'_{\bxi}e_{\boldeta}
=(|\BF\boldeta_{\GV}|^2+V'_{\bxi})e_{\boldeta}
\ene
for all $\boldeta\in \BPsi_j(\bxi)$ with $\boldeta_j$ and $\bmu$ being defined by
\eqref{etadecomposition}.
These definitions imply that $\ker
A=\CP^{(\bk)}(\BPsi_0)\GH$.
Notice that
\bee\label{newboundA}
\|A(\bxi)\|\ll R<\rho^{1/3}
\ene
and
\bee\label{newboundB}
\|B(\bxi)\|\ll L_n^2<\rho^{2/3}
\ene
due to our assumptions made before lemma \ref{changedXi};
see also corollary \ref{cor:changedXi0}.

The dependence of the operator pencil $H'=r^2I+rA+B$ on $r$ is two-fold:
together with the obvious quadratic dependence, the coefficients $A$ and $B$
depend on $r$ as well. However, as we will show in lemma \ref{newdependenceonr},
the second type of dependence is rather weak.
Put
\bees
D(\bxi):=r(\bxi) A(\bxi)+B(\bxi).
\enes
By $\{\nu_j(\bxi)\}$ we denote the eigenvalues
of $D(\bxi)$. Then according to \eqref{newpencil} the eigenvalues of $H'(\bxi)$ are equal to
\bee\label{eigenvalueslamu}
\la_j(\bxi)=r^2(\bxi)+\nu_j(\bxi).
\ene

If $\bxi_1,\bxi_2\in\Xi_3(\GV)$, then we can define the operator $A(\bxi_1;\bxi_2)$
as the operator defined by \eqref{operatorA} with the domain $P(\bxi_1;\bxi_2)\GH$.
Similarly, if $U$ is a set containing $\bxi$ of diameter $\ll\rho^{-1}$, then we
define the operator $A(\bxi;U)$
as the operator defined by \eqref{operatorA} with the domain $P(\bxi;U)\GH$.
In the same way, we can define $B(\bxi_1;\bxi_2)$, $B(\bxi;U)$ (they are defined
by means of \eqref{operatorB}),
$D(\bxi_1;\bxi_2)=r(\bxi_1)A(\bxi_1;\bxi_2)+ B(\bxi_1;\bxi_2)$, and $D(\bxi;U)$. We also denote by
$\nu_j(\bxi_1;\bxi_2)$ the eigenvalues of $D(\bxi_1;\bxi_2)$ and by
\bees
\la_j(\bxi_1;\bxi_2)=r^2(\bxi_1)+\nu_j(\bxi_1;\bxi_2)
\enes
the eigenvalues of $H'(\bxi_1;\bxi_2)$; $\nu_j(\bxi;U)$ and $\la_j(\bxi;U)$ are defined
analogously.

Let us now study how the eigenvalues
change under the change of $r$.

\bel\label{newdependenceonr}
Let $\bxi\in\Xi_3(\GV)$, $r=r(\bxi)$. Let $U$ be a set of diameter $\ll\rho^{-1}$
containing $\bxi$.
Let $t$ be a real number with $|t-r|\ll\rho^{-1}$
and $\ba=\ba(t)\in\Xi_2(\GV)$ be a unique point satisfying
$(\ba)_\GV=(\bxi)_\GV$, $(\ba)'_\GV=(\bxi)'_\GV$, and $r(\ba)=t$
(thus, when we vary $t$, the path $\ba(t)$ is a straight interval which goes along the $\BF$-perpendicular
dropped from the point $\bxi$ onto $\GV$).
Suppose, $\ba\in U$.
Let $\nu_j(t)$ (resp.
$\la_j(t)$) denote the eigenvalues of $D(\ba(t);U)$ (resp. $H'(\ba(t);U)$).
Then
\bee\label{new17}
\frac{d \nu_j(t)}{d t}=O(\rho^{1/3})
\ene
and
\bee\label{new18}
\frac{d \la_j(t)}{d t}=2t+O(\rho^{1/3})
\ene
\enl
\bep
Let $t_1$, $t_2$ be real numbers satisfying $|t_j-r|\ll\rho^{-1}$ and
$\ba_1=\ba(t_1)$, $\ba_2=\ba(t_2)$ be the corresponding points inside $\Xi_3(\GV)\cap U$.
First of all, we notice that the mapping $F_{\ba_1,\ba_2}$ defined by \eqref{mappingF}
is an isometry which maps $P(\ba_1;U)\GH$ onto $P(\ba_2;U)\GH$. Moreover, the definitions of
the operators $A$ and $B$ imply that
\bees
A(\ba_1;U)=F_{\ba_2,\ba_1} A(\ba_2;U) F_{\ba_1,\ba_2};
\enes
similarly,
\bees
B(\ba_1;U)=F_{\ba_2,\ba_1} B(\ba_2;U) F_{\ba_1,\ba_2}.
\enes
These unitary equivalencies show that the eigenvalues $\nu_j(t)$ are in fact the eigenvalues
of the linear operator pencil $tA+B$, with $A$ and $B$ being any of the
operators $A(\ba;U)$ and $B(\ba;U)$ with $\ba$ satisfying
$(\ba)_\GV=(\bxi)_\GV$ and $(\ba)'_\GV=(\bxi)'_\GV$; it
does not matter which particular point $\ba$ we have chosen, since all corresponding
operators are unitarily equivalent. For example, we can choose $A=A(\bxi;U)$ and
$B=B(\bxi;U)$. Now an elementary perturbation theory shows that
\bee\label{newpert}
\frac{d\nu_j}{dt}=\lu A u_j,u_j\ru,
\ene
where $u_j$ is the eigenvector of $D$ corresponding to the eigenvalue $\nu_j$.
The estimate \eqref{newboundA} shows that $d\nu_j/dt=O(\rho^{1/3})$. This proves
\eqref{new17}. The estimate \eqref{new18} follows from this and the identity
$\la_j(t)=t^2+\nu_j(t)$.
\enp

Using similar perturbative argument, we can study how the eigenvalues change when
we change the other variables, namely, $\bxi_{\GV}$ and $\bxi'_{\GV}$.

\bel\label{newdependenceonrest}
Let $\bxi\in\Xi_3(\GV)$, and let $\ba\in\Xi_3(\GV)$ be the point satisfying
$r(\ba)=r(\bxi)$, $|\ba-\bxi|\ll\rho^{-1}$. 
Suppose, $\bxi,\ba\in U$. Then
\bee\label{new19}
|\la_j(\ba;U)-\la_j(\bxi;U)|=|\nu_j(\ba;U)-\nu_j(\bxi;U)|\ll |\ba-\bxi|\rho^{1/3}.
\ene
\enl
\bep
Formula \eqref{eigenvalueslamu} and the condition $r(\ba)=r(\bxi)$ imply that
$\la_j(\ba;U)-\la_j(\bxi;U)=\nu_j(\ba;U)-\nu_j(\bxi;U)$. Moreover, definitions
\eqref{operatorA} and \eqref{operatorB} imply that
\bee\label{newA}
||A(\bxi;U)-F_{\ba,\bxi} A(\ba;U) F_{\bxi,\ba}||\ll |\bxi'_\GV-\ba'_\GV|R\ll
|\ba-\bxi|\rho^{-2/3}
\ene
and
\bee\label{newB}
||B(\bxi;U)-F_{\bxi,\ba} B(\ba;U) F_{\ba,\bxi}||\ll |\bxi_\GV-\ba_\GV|(|\bxi_\GV|+L_n)
\ll |\ba-\bxi|\rho^{1/3}.
\ene
Indeed, let us check for example \eqref{newA}. Suppose, $\boldeta\in\BUps$ (say, $\boldeta\in\BPsi_j$).
Then we have:
\bees
A(\bxi;U)\e_{\boldeta}=2\lu\BF\bxi_{\GV}',\BF\boldeta_j\ru e_{\boldeta}
\enes
and
\bees
F_{\ba,\bxi}AF_{\bxi,\ba}\e_{\boldeta}=2\lu\BF\ba_{\GV}',\BF\boldeta_j\ru e_{\boldeta}.
\enes
Since
\bees
|\ba_{\GV}'-\bxi_{\GV}'|\ll |\ba-\bxi|\rho^{-1}
\enes
and
\bees
|\boldeta_j|\ll R,
\enes
we have \eqref{newA}. The estimate \eqref{newB} can be proved analogously.

Therefore, since $r(\bxi)\sim\rho$, we have
\bees
||D(\bxi;U)-F_{\bxi,\ba} D(\ba;U) F_{\ba,\bxi}||\ll |\ba-\bxi|\rho^{1/3}.
\enes
Since the spectra of $D(\ba;U)$ and $F_{\bxi,\ba} D(\ba;U) F_{\ba,\bxi}$ coinside,
this implies
\bees
|\nu_j(\bxi;U)-\nu_j(\ba;U)|\ll |\ba-\bxi|\rho^{1/3},
\enes
which finishes the proof.
\enp

Let us summarize the information about the spectra of $H'(\bxi)$ we have obtained so far.
Recall that $\CA_1$ is a slightly `slimmed down' version of $\CA$; it consists of all points
$\bxi$ with $|\BF\bxi|^2\in J$.

\bel\label{newsummarize}
Let $\bxi_1,\bxi_2\in U\subset\Xi_2(\GV)\cap\CA_1$
with the diameter of $U$ being $\rho^{-1}$.
Assume that $\mu_j(H'(\bxi_1;U))\in J$. Then
\bees
|\mu_j(H'(\bxi_1;U))-\mu_{j}(H'(\bxi_2;U))|\ll \rho|\bxi_1-\bxi_2|+\rho^{-4Mp}.
\enes
If we assume, moreover, that $(\bxi_1)_\GV=(\bxi_2)_\GV$ and $(\bxi_1)'_\GV=(\bxi_2)'_\GV$,
then
\bees
\mu_j(H'(\bxi_1;U))-\mu_{j}(H'(\bxi_2;U))=(2\rho+O(\rho^{1/3}))(r(\bxi_1)-r(\bxi_2))+O(\rho^{-4Mp}).
\enes
Finally, if $(\bxi_1)_\GV=(\bxi_2)_\GV$, $(\bxi_1)'_\GV=(\bxi_2)'_\GV$, and $U$ contains the interval
$I$ joining $\bxi_1$ and $\bxi_2$, then
\bees
\mu_j(H'(\bxi_1;U))-\mu_{j}(H'(\bxi_2;U))=(2\rho+O(\rho^{1/3}))(r(\bxi_1)-r(\bxi_2)).
\enes
\enl
\bep
The last statement follows directly from lemma \ref{newdependenceonr}.
Assume now that $(\bxi_1)_\GV=(\bxi_2)_\GV$ and $(\bxi_1)'_\GV=(\bxi_2)'_\GV$. Denote $U_1:=U$,
$U_2:=U\cup I$, where $I$ is the interval joining $\bxi_1$ and $\bxi_2$, and
$l=l(\bxi_1;U_1,U_2)=l(\bxi_2;U_1,U_2)$ (the last equality follows from lemma \ref{newU1}). Then
lemma \ref{newU1} implies that
\bees
|\mu_j(H'(\bxi_m;U_1))-\mu_{j+l}(H'(\bxi_m;U_2))|\ll\rho^{-4Mp},\qquad m=1,2.
\enes
Now the statement follows from
lemma \ref{newdependenceonr}. If $r(\bxi_1)=r(\bxi_2)$,
the statement follows in a similar way from lemmas \ref{newU1} and
\ref{newdependenceonrest}. In the general case, we join $\bxi_1$ and $\bxi_2$ by a path consisting of intervals
falling into either of the two cases above.
\enp

Now we will `globalize' the local mappings constructed so far, in other words,
we will define the function $j:\BUps(\bxi)\to\N$ mentioned in the remark \ref{newremark}.
Let $\bxi\in\Xi(\GV)$ and $\{\bxi\}=\bk$.
Then the set of eigenvalues $\{\mu_j(H'_0(\bxi))\}$ of the unperturbed operator $H'_0(\bxi)$
coincides with the set $\{|\BF\boldeta|^2,\, \boldeta\in\BUps(\bxi)\}$. Let us label all
numbers $\{|\BF\boldeta|^2,\, \boldeta\in\BUps(\bxi)\}$ in the increasing order; if there are two different vectors
$\boldeta, \tilde\boldeta\in\BUps(\bxi)$ with $|\BF\boldeta|^2=|\BF\tilde\boldeta|^2$, we
label them in the lexicographic order of their coordinates (i.e. we put $\boldeta$ before $\tilde\boldeta$
if either $\eta_1<\tilde\eta_1$, or $\eta_1=\tilde\eta_1$ and $\eta_2<\tilde\eta_2$, etc.)
Then to each point $\boldeta\in\BUps(\bxi)$
we have put into correspondence a natural number $j=j(\boldeta)$
such that
\bee\label{newindices}
|\BF\boldeta|^2=\mu_j(H'_0(\bxi)).
\ene
Next we define
\bees
\tilde g(\boldeta):=\mu_{j(\boldeta)}(H'(\bxi)).
\enes
This mapping is well-defined and satisfies the following obvious property:
$|\tilde g(\boldeta)-|\BF\boldeta|^2|\le v$ (recall that $v=||V||_{\infty}$).

The problem with the mapping $\tilde g$ defined in this way is that we cannot apply lemma \ref{newdependenceonr} to it, since
lemma  \ref{newdependenceonr} treats not the eigenvalues of $H'(\bxi)$, but the eigenvalues of $H'(\bxi;U)$ with the set $U$ containing
certain intervals perpendicular to $\GV$. Thus, we need to introduce a different definition which takes care of lemma
\ref{newdependenceonr} and at the same time is reasonably canonical.

Let $\bxi\in\Xi_2(\GV)$. Denote
\bees
X=X(\bxi):=\{\boldeta\in\Xi_2(\GV):
\boldeta_{\GV}=\bxi_{\GV},\,\boldeta'_{\GV}=\bxi'_{\GV}\}.
\enes
Simple geometry implies that $X(\bxi)$ is an interval of length $\ll \rho^{-1}$.
Similarly to our actions when we were defining $\tilde g$, we notice
that the set of eigenvalues $\{\mu_j(H'_0(\bxi;X(\bxi)))\}$
coincides with the set $\{|\BF\boldeta|^2,\, \boldeta\in\BUps(\bxi;X)\}$. Let us label all
numbers $\{|\BF\boldeta|^2,\, \boldeta\in\BUps(\bxi;X)\}$ in the increasing order; if there are two different vectors
$\boldeta_1, \boldeta_2\in\BUps(\bxi;X)$ with $|\BF\boldeta_1|^2=|\BF\boldeta_2|^2$, we
label them in the lexicographic order of their coordinates. Then to the point $\bxi$ we have put into correspondence a natural number $i=i(\bxi)$
such that
\bee\label{newindicesi}
|\BF\bxi|^2=\mu_i(H'_0(\bxi;X)).
\ene
Next we define $g(\bxi):=\mu_{i(\bxi)}(H'(\bxi;X))$.
This mapping is well-defined and satisfies the property
$|g(\bxi)-|\BF\bxi|^2|\le v$.

\bel\label{tildeg}
Let $\bxi\in\Xi_2(\GV)\cap\CA_1$. Then
the following properties are satisfied:

(i) $|g(\bxi)-\tilde g(\bxi)|\ll \rho^{-4Mp}$;

(ii) $g(\bxi)=r^2+s$, where $r=r(\bxi)$ and
$s=s(\bxi)=s(\bxi_{\GV},\bxi'_{\GV},r)$ is a function which
smoothly depends on $r$ with $\frac{\partial s}{\partial r}=O(\rho^{1/3})$.
\enl
\bep
Let us prove the first statement. First, we notice that the difference $i(\bxi)-j(\bxi)$
is equal to the number of points $\boldeta\in(\BUps(\bxi;X(\bxi))\setminus\BUps(\bxi)$
satisfying $|\BF\boldeta|^2<\lambda$. Now the statement follows from
lemma \ref{newisometry} and remark \ref{rem:newisometry0}.

Let us now prove the second statement. Suppose, $\bxi_1\in X(\bxi)$ and $\boldeta\in\BUps(\bxi;X)$.
Then
\bees
\boldeta_1:=\boldeta+(\bxi_1-\bxi)\in\BUps(\bxi_1;X).
\enes
Note that $(\bxi)_{\GV}=(\bxi_1)_{\GV}$ and therefore $(\boldeta)_{\GV}=(\boldeta_1)_{\GV}$.
Let us assume that
\bee\label{newxieta}
|\BF\bxi|\ge|\BF\boldeta|
\ene
and prove that this implies
$|\BF\bxi_1|\ge|\BF\boldeta_1|$. Indeed, there are two possible cases:

(i) $\bxi-\boldeta\in\GV$. Then, since $\bxi_1-\boldeta_1=\bxi-\boldeta\in\GV$, we have:
\bees
\bes
&|\BF\bxi_1|^2-|\BF\boldeta_1|^2=|\BF(\bxi_1)_{\GV}|^2-|\BF(\boldeta_1)_{\GV}|^2\\
&=|\BF(\bxi)_{\GV}|^2-|\BF(\boldeta)_{\GV}|^2=|\BF\bxi|^2-|\BF\boldeta|^2\ge 0
\end{split}
\enes

(ii) $\bxi-\boldeta\not\in\GV$. Then, in the same way we have proved estimate \eqref{newBUps1}, using
lemma \ref{changedXi3} we can show that $||\BF\boldeta|^2-\la|\gg K^2$. But since $\bxi\in\Xi_2(\GV)\subset\CA$,
we have $||\BF\bxi|^2-\la|\ll 1$. Thus, \eqref{newxieta} implies $\la-|\BF\boldeta|^2\gg K^2$. Since
$|\boldeta_1-\boldeta|=|\bxi_1-\bxi|\ll\rho^{-1}$, we have $\la-|\BF\boldeta_1|^2\gg K^2$. Since $\bxi_1\in\CA$,
this implies $|\BF\bxi_1|\ge|\BF\boldeta_1|$.

Thus, we have proved that the inequality $|\BF\bxi_1|\ge|\BF\boldeta_1|$ is equivalent to
$|\BF\bxi|\ge|\BF\boldeta|$. This implies that $i(\bxi_1)=i(\bxi)$, where $i$ is the function defined by
\eqref{newindicesi}. Now the second statement of lemma follows from lemma \ref{newdependenceonr}.
\enp
This lemma shows that the mapping $g$ behaves in a nice way as a function of $r$.
Unfortunately, the dependence on other variables is not quite so nice.
In fact, this mapping is not continuous, even modulo $O(\rho^{-4Mp})$, because
the functions $i(\bxi)$ are not continuous; moreover, a little
thought shows that we cannot, in general, define the mapping $g$ to have all properties
formulated in the introduction and be continuous at the same time. Indeed, if the function
$i=i(\bxi)$ were continuous, it would necessary have been a constant. Thus, the function $i$
has discontinuities, and the function $g$ may have discontinuities at the same points as $i$.
However,
lemmas \ref{newisometry}, \ref{newdependenceonr}, and \ref{newdependenceonrest}
show that for each small neighbourhood $U$ in the space
of quasi-momenta we can find a family of representatives of the functions $g$ which is
`almost' smooth. Namely, the following statement holds:
\bel\label{newpartialofs}
Let $I=[\ba,\bb]\subset\Xi_2(\GV)\cap\CA_1$ be a straight interval
of length $L:=|\bb-\ba|\ll\rho^{-1}$. 
Then there exists an integer vector $\bn$ such that
$|g(\bb+\bn)-g(\ba)|\ll L\rho+\rho^{-4Mp}$. Moreover, suppose in addition that there exists
an integer vector $\bm\ne 0$ such that the interval $I+\bm$ is entirely inside $\Xi_2(\GV)\cap\CA_1$.
Then there exist two different integer vectors $\bn_1$ and $\bn_2$ such that
$|g(\bb+\bn_1)-g(\ba)|\ll L\rho+\rho^{-4Mp}$ and
$|g(\bb+\bn_2)-g(\ba+\bm)|\ll L\rho+\rho^{-4Mp}$.
\enl
\bep
Lemmas \ref{newisometry} and \ref{tildeg}
show that $g(\ba)=\la_l(\ba;I)+O(\rho^{-4Mp})$
for some integer $l$. Lemma \ref{newsummarize}
now implies that
\bee\label{newLrho}
|\la_l(\ba;I)-\la_l(\bb;I)|\ll L\rho+\rho^{-4Mp}.
\ene
Once again using
lemma \ref{newisometry}, we deduce that $\la_l(\bb;I)=g(\boldeta)+O(\rho^{-4Mp})$
for some $\boldeta\in\BUps(\bb;I)$; in particular, we have $\boldeta=\bb+\bn$ for some integer vector $\bn$.
This proves the first statement.

Let us prove the second statement. Conditions of lemma imply that $g(\ba)=
\la_j(\ba;I)+O(\rho^{-4Mp})$ and $g(\ba+\bm)=\la_l(\ba+\bm;I+\bm)+O(\rho^{-4Mp})$
for some integers $j,l$. Moreover, if $\ba+\bm\in\BUps(\ba;I)$ (so that
$\BUps(\ba;I)=\BUps(\ba+\bm;I+\bm)$), then, since $\bm\ne 0$ we have
\bee\label{newjl}
j\ne l.
\ene
Lemma \ref{newsummarize}
now implies that together with \eqref{newLrho} we have
\bee\label{newLrho1}
|\la_l(\ba+\bm;I+\bm)-\la_l(\bb+\bm;I+\bm)|\ll L\rho+\rho^{-4Mp}.
\ene
Once again using
lemma \ref{newisometry}, we deduce that $\la_j(\bb;I)=g(\boldeta_1)+O(\rho^{-4Mp})$
and $\la_l(\bb+\bm;I+\bm)=g(\boldeta_2)+O(\rho^{-4Mp})$
for different points $\boldeta_1\in\BUps(\bb;I)$ and $\boldeta_2\in\BUps(\bb+\bm;I+\bm)$
(the points $\boldeta_1$ and $\boldeta_2$ are different because of \eqref{newjl}).
In particular, these inclusions imply $\boldeta_1-\bb\in\Z^d$ and $\boldeta_2-\bb\in \bm+\Z^d=\Z^d$.
This proves the second statement.
\enp

Thus, we have proved the following lemma, which
is the main result of this section:

\bel\label{mainlemmasection6}
Let $\GV\in\CV(n)$. Then there are two mappings
$\tilde g,g:\Xi_2(\GV)\to\R$ which satisfy the following properties:

(i) $\tilde g(\bxi)$ is an eigenvalue of $ P(\GV) H'(\bk)P(\GV)$
with $\{\bxi\}=\bk$. All
eigenvalues of $ P(\GV) H'(\bk) P(\GV)$ inside $J$ are in
the image of $\tilde g$.

(ii) If $\bxi\in\CA_1$, then
$|\tilde g(\bxi)-g(\bxi)|\le C\rho^{-4Mp}$ and $|g(\bxi)-|\BF\bxi|^2|\le 2v$.

(iii) $g(\bxi)=r^2+s(\bxi)$
with $r:=|\BF\bxi_\GV^{\perp}|$ and $\frac{\partial s}{\partial r}=O(\rho^{1/3})$.
\enl
\bep
The only statement which has not been checked so far is that
$|g(\bxi)-|\BF\bxi|^2|\le 2v$.
This follows immediately from the second statement of this lemma together with the
inequality $|\tilde g(\bxi)-|\BF\bxi|^2|\le v$ and
$|\tilde g(\bxi)-g(\bxi)|\le C\rho^{-4Mp}$.
\enp

Now, we can put together the results of the previous sections

\bet\label{maintheorem1}
Suppose, 
$R$ is sufficiently large, all conditions before lemma \ref{changedXi} are satisfied, and $R<\rho^{pd^{-1}/2}$.
Then there are two mappings $ f,g:\CA\to\R$
which satisfy the following properties:

(i) $f(\bxi)$ is an eigenvalue of $H'(\bk)$ with
$\{\bxi\}=\bk$; $| f(\bxi)-|\BF\bxi|^2|\le 2v$. $f$ is an injection (if we count all eigenvalues
with multiplicities) and all
eigenvalues of $H'(\bk)$ inside $J$ are in the image of $f$.

(ii) If $\bxi\in\CA_1$, then $| f(\bxi)-g(\bxi)|\le C\rho^{-4Mp}$.

(iii) We can decompose the domain of $g$ into the disjoint union:
$\CA=\CB\cup\bigcup_{n=1}^{d-1}\bigcup_{\GV\in\CV(6MR,n)}\Xi_2(\GV)$.
For any $\bxi\in\CB$
\bee\label{eq1:maintheorem1}
\bes
&g(\bxi)=|\BF\bxi|^2\\
&+\sum_{j=1}^{2M} \sum_{\boldeta_1,\dots,\boldeta_j\in \T'_M}
\sum_{2\le n_1+\dots+n_j\le 2M}C_{n_1,\dots,n_j}
\lu\bxi,\BG\boldeta_1\ru^{-n_1}\dots\lu\bxi,\BG\boldeta_{j}\ru^{-n_j}.
\end{split}
\ene
For any $\bxi\in\Xi_2(\GV)$
\bee\label{eq2:maintheorem1}
g(\bxi)=r^2(\bxi)+s(\bxi),
\ene
with
$r(\bxi)=|\BF\bxi_\GV^{\perp}|$, $s(\bxi)=s(r,\bxi_\GV,\bxi'_\GV)$,
$\frac{\partial s}{\partial r}=O(\rho^{1/3})$.
\ent
\bep
We have described the construction of the mapping $f$ at the end of section 5.
Mapping $g$ is constructed in sections 6 and 7.
\enp
Let us formulate an important property of the mapping $g$, which is a global
version of lemma \ref{newpartialofs}.
\bel\label{partialofs2}
Let $I=[\ba,\bb]\subset\CA_1$ be a straight interval
of length $L:=|\bb-\ba|\ll\rho^{-1}$.
Then there exists an integer vector $\bn$ such that
$|g(\bb+\bn)-g(\ba)|\ll L\rho+\rho^{-4Mp+d}$.
Moreover, suppose in addition that there exists
an integer vector $\bm\ne 0$ such that the interval $I+\bm$ is entirely inside $\CA_1$.
Then there exist two different integer vectors $\bn_1$ and $\bn_2$ such that
$|g(\bb+\bn_1)-g(\ba)|\ll L\rho+\rho^{-4Mp+d}$ and
$|g(\bb+\bn_2)-g(\ba+\bm)|\ll L\rho+\rho^{-4Mp+d}$.
\enl
\bep
Let us parametrise the interval $I$ so that
\bees
I=\{\bxi(t),t\in [t_{min},t_{max}]\},
\enes
$\ba=\bxi(t_{min})$, $\bb=\bxi(t_{max})$.
Let us prove the first statement.
If the interval $\bxi(t)$ lies entirely inside $\CB$, then the statement is obvious since
the length of the gradient of $g$ inside $\CB$ is $\ll\rho$, so we can take $\bn=0$.
If the interval $\bxi(t)$ lies entirely inside $\Xi_2(\GV)$ for $\GV\in\CV(n)$, the
statement has been proved in lemma \ref{newpartialofs}.
Consider the general case. Denote by $y_j(\bk):=\mu_j(H'(\bk))$ the $j$-th eigenvalue
of $H'(\bk)$.
Then the definition of the mapping $f$ implies that if $y_j(\bk)\in J$, then
\bees
y_j(\bk)=f(\bn+\bk)
\enes
for some integer vector $\bn$; the opposite is also true, namely
if $f(\bn+\bk)\in J$, then $f(\bn+\bk)=y_j(\bk)$ for some $j$.
Notice also that for each $j$ the function $y_j$ is continuous.

Now let us return to the study of the behaviour of the function $g(\bxi(t))$. Suppose for
definiteness that $\bxi(t_{min})\in\CB$. Then, as we mentioned in the beginning of proof,
since the gradient of $g$ has length $\ll \rho$, we have
$|g(\bxi(t))-g(\bxi(t_{min}))|\ll |\bxi(t)-\bxi(t_{min})|\rho$
as soon as $\bxi(t)$ stays inside $\CB$.
Suppose that $t_1$ is the point at which $\bxi(t)$ crosses the boundary of $\CB$. Then
\bee\label{eq0:partialofs2}
|g(\bxi(t_1-0))-g(\bxi(t_{min}))|\ll |\bxi(t)-\bxi(t_{min})|\rho
\ene
According to the relationship between the mapping $f$ and functions $y_j$ stated above,
there exists an index $j$ such that
$f(\bxi(t_{1}-0))=y_j(\{\bxi(t_{1}-0)\})$ (recall that if $\bxi=\bn+\bk$, then
we call $\bk=\{\bxi\}$ the fractional part of $\bxi$). Since $y_j$ is continuous function,
$y_j(\{\bxi(t_{1}-0)\})=y_j(\{\bxi(t_{1}+0)\})$. Using the relationship between
the mapping $f$ and functions $y_j$ again, we deduce that
there exists an integer vector $\bn_1$ such that
$y_j(\{\bxi(t_{1}+0))\})=f(\bxi(t_{1}+0)+\bn_1)$. Property (ii) of theorem
\ref{maintheorem1} implies that
$f(\bxi(t_{1}-0))=g(\bxi(t_{1}-0))+O(\rho^{-4Mp})$ and, similarly,
$f(\bxi(t_{1}+0)+\bn_1)=g(\bxi(t_{1}+0)+\bn_1)+O(\rho^{-4Mp})$. All these estimates imply
\bee\label{eq1:partialofs2}
g(\bxi(t_{1}+0)+\bn_1)=g(\bxi(t_{1}-0))+O(\rho^{-4Mp}).
\ene
Since
$\bxi(t_1+0)+\bn_1\in\CA_1$, we have either $\bxi(t_1+0)+\bn_1\in\Xi_2(\GV)$ or
$\bxi(t_1+0)+\bn_1\in\CB$. Assume the former. Let $t_2>t_1$ be the smallest value of $t$
at which $\bxi(t)+\bn_1$ 
crosses the boundary of
$\Xi_2(\GV)$. Then lemma \ref{newpartialofs} implies that thee exists an integer vector
$\bn_2$ such that
\bee\label{eq2:partialofs2}
|g(\bxi(t_2-0)+\bn_2)-g(\bxi(t_1+0)+\bn_1)|\ll |\bxi(t_2)-\bxi(t_1)|\rho+O(\rho^{-4Mp}).
\ene
Now repeating the argument we have already used at the moment $t_1$, we deduce that there
exists an integer $\bn_3$ such that
\bee\label{eq3:partialofs2}
g(\bxi(t_{2}+0)+\bn_3)=g(\bxi(t_{2}-0)+\bn_2)+O(\rho^{-4Mp}).
\ene
Now we repeat the process
and increase $t$ beginning from $t_2$ until we hit another piece of boundary of some
$\Xi_2(\GW)$ at $t=t_3$, etc. The shift of the function $g$ at each of the points $t_j$ of
hitting the boundary is $O(\rho^{-4Mp})$. The number of such points is $\ll \rho^{d}$, since
for each fixed integer vector $\bm$ the number of intersections of the interval
$(\bxi(t)+\bm)$ ($t\in[t_{min},t_{max}]$) with the boundaries of all sets $\Xi_2(\GV)$
is finite, and the number of possible integer vectors $\bm$ allowed here is $\ll \rho^{d}$
(obviously, the length of each of these integer vectors is $\ll\rho$). Now formulas
\eqref{eq0:partialofs2}--\eqref{eq3:partialofs2} lead to the desired result.

The proof of the second statement is similar and can be derived from the proof of the first
statement in the same way as the proof of the second part of lemma \ref{newpartialofs}
follows from the proof of the first part of that lemma.
\enp

Now it remains to extend the above results to the `full' operator $H(\bk)$.

\bec\label{maincorollary1}
For
each natural $N$ there exist mappings $f,g:\CA\to\R$ which
satisfy the following properties:

(i) $f(\bxi)$ is an eigenvalue of $H(\bk)$ with $\{\bxi\}=\bk$;
$|f(\bxi)-|\BF\bxi|^2|\le 2v$. $f$ is an injection (if we count all eigenvalues with multiplicities)
and all eigenvalues of $H(\bk)$ inside
$J$ are in the image of $f$.

(ii) If $\bxi\in\CA_1$, then $|f(\bxi)-g(\bxi)|\le \rho^{-N}$.

(iii) We can decompose the domain of $g$ into the disjoint union:
$\CA=\CB\cup\bigcup_{n=1}^{d-1}\bigcup_{\GV\in\CV(n)}\Xi_2(\GV)$.
For any $\bxi\in\CB_{\rho}$
\bee\label{eq1:maincorollary1}
\bes
&g(\bxi)=|\BF\bxi|^2\\
&+\sum_{j=1}^{2M} \sum_{\boldeta_1,\dots,\boldeta_j\in \T'_M}
\sum_{2\le n_1+\dots+n_j\le 2M}C_{n_1,\dots,n_j}
\lu\bxi,\BG\boldeta_1\ru^{-n_1}\dots\lu\bxi,\BG\boldeta_{j}\ru^{-n_j}
\end{split}
\ene
with $M=[(N+d)(4p)^{-1}]+1$.
For any $\bxi\in(\Xi_2(\GV)\cap\CA_1)$
\bee\label{eq2:maincorollary1}
g(\bxi)=r^2+s(\bxi),
\ene
with
$r:=|\BF\bxi_\GV^{\perp}|$, $s(\bxi)=s(r,\bxi_\GV,\bxi'_\GV)$ and
$\frac{\partial s}{\partial r}=O(\rho^{1/3})$.

(iv) Let $I=[\ba,\bb]\subset\CA_1$ be a straight interval
of length $L:=|\bb-\ba|\ll\rho^{-1}$.
Then there exists an integer vector $\bn$ such that
$|g(\bb+\bn)-g(\ba)|\ll L\rho+\rho^{-N}$. Moreover, suppose $\bm\ne 0$
is a given integer vector such that the interval $I+\bm$ is entirely inside $\CA_1$.
Then there exist two different integer vectors $\bn_1$ and $\bn_2$ such that
$|g(\bb+\bn_1)-g(\ba)|\ll L\rho+\rho^{-N}$ and
$|g(\bb+\bn_2)-g(\ba+\bm)|\ll L\rho+\rho^{-N}$.
\enc
\bep
We use theorem
\ref{maintheorem1} for the operator $H'(\bk)$ with $M=[(N+d)(4p)^{-1}]+1$.
Estimate \eqref{truncate1} implies that
$|\mu_j(H(\bk))-\mu_j(H'(\bk))|<\rho^{-N-1}$, so that all the
required properties are fulfilled.
\enp
\ber The function $f$ is
not necessarily continuous.
\enr

Before we continue with the proof of the Bethe-Sommerfeld
conjecture, let us formulate a theorem which immediately follows
from our results, just to illustrate their usefulness.
Recall that by $N(\la)$ we have denoted
the integrated density of states of the operator \eqref{Schroedinger1} defined in
\eqref{densityofstates}.
\bet\label{clusters}
For each natural $n$ we have the following
estimate: $N(\la+\la^{-n})-N(\la-\la^{-n})=O(\la^{d/2-n-1})$.
\ent
\bep We use corollary \ref{maincorollary1} with $N=2n+1$. Then
\bee\label{eq:density}
\bes
&N(\la+\la^{-n})-N(\la-\la^{-n})=\volume(f^{-1}([\la-\la^{-n},\la+\la^{-n}]))\\
&\le \volume(g^{-1}([\la-2\la^{-n},\la+2\la^{-n}]))=O(\la^{d/2-n-1}),
\end{split}
\ene
the last equality being an easy geometric exercise (which will anyway be established
in the next section).
\enp
\ber
As it was pointed out to the author by Yu.Karpeshina, it seems possible that using the results of
this paper (including the results from the next section) one can prove the following lower bound:
\bees
N(\la+\ep)-N(\la)\gg\ep\la^{(d-2)/2},
\enes
uniformly over $\ep<1$ as $\la\to\infty$ (in particular, $\ep$ does not have to be a negative power
of $\lambda$). We will not prove this estimate in our paper though.
\enr

\section{Proof of the Bethe-Sommerfeld conjecture}
Throughout this section we keep the notation from the previous section.
Without specific mentioning, we always assume
that $\rho$ is sufficiently large; the precise value of the power $N$ will
be chosen later. In what follows, it will be convenient to consider a slightly
slimmed down resonance set. Namely, we introduce the set
\bees
\tilde\CB:=\{\bxi\in\CA_1: |\bxi_{\GU}|>\rho^{1/2},\forall \GU\in\CV(1)\}.
\enes
In other words, $\tilde\CB$ consists of all points $\bxi\in\CA_1$ the $\BF$-projections of
which to all vectors $\boldeta\in\T'_{6M}$ has $\BF$-length larger than $\rho^{1/2}$.
Obviously, $\tilde\CB\subset \CB$. We also denote $\tilde\CD:=\CA_1\setminus\tilde\CB$.

Now we will study various properties of mappings $f$ and $g$.
We begin with the function $g$.

For each positive $\de\le v$
denote $\CA(\de)$, $\CB(\de)$, and $\CD(\de)$ to be intersections
of $g^{-1}([\rho^2-\de,\rho^2+\de])$ with
$\CA_1$, $\tilde\CB$, and $\tilde\CD$ correspondingly.
The following is a simple geometry:
\bel\label{volumeCA5}
The following estimates hold:
\bee\label{volumeCA}
\volume(\CA(\de))\asymp \rho^{d-2}\de,
\ene
\bee\label{volumeCB}
\volume(\CB(\de))\asymp \rho^{d-2}\de,
\ene
and
\bee\label{volumeCD}
\volume(\CD(\de))\ll \rho^{(3d-7)/3}\de.
\ene
\enl
\bep
Let $\bxi=r\bxi'\in\CB$, $|\BF\bxi'|=1$ Then the definition of $g$ implies that
\bee\label{partialg}
\frac{\partial g}{\partial r}\asymp \rho
\ene
uniformly over $\bxi'$. Therefore, for each fixed $\bxi'$ the intersection of
$g^{-1}([\rho^2-\de,\rho^2+\de])$ with the set $\{r\bxi',\, r>0\}$ is an interval
of length $\asymp \de\rho^{-1}$. Integrating over $\bxi'$, we obtain \eqref{volumeCB}.
Estimate \eqref{volumeCD} is obtained in a similar way, only for $\bxi\in\Xi(\GV)$ we put
$r:=|\BF\bxi_\GV^{\perp}|$. Then the estimate \eqref{partialg} is
still valid.
Let $\boldeta\in\T'_M$. Then \eqref{partialg} implies that the set of all points $\bxi\in\CA(\de)$
such that the $\BF$-projection of $\bxi$ onto $\boldeta$ has $\BF$-length
smaller than $\rho^{1/2}$ has volume $O(\rho^{(2d-5)/2}\de)$. Since the number of elements
in $\T'_M$ is $O(R^d)=O(\rho^{p/2})$, we have
\bees
\volume(\CD(\de))\ll \rho^{(2d-5+p)/2}\de\ll\rho^{(3d-7)/3}\de,
\enes
since $p<1/3$.
Finally, \eqref{volumeCA} is the sum of \eqref{volumeCB} and \eqref{volumeCD}.
\enp
\ber
Putting $\de=2\la^{-n}$ in \eqref{volumeCA}, we establish the last equality in \eqref{eq:density}.
\enr

The next estimate is more subtle.

\bel\label{volumeCB1}
Let $d\ge 3$. Then for large enough $\rho$ and $\de<\rho^{-1}$ the following estimate holds uniformly
over $\ba\in\R^d$ with $|\ba|>1$:
\bee\label{eq:volumeCB1}
\volume\bigl(\CB(\de)\cap (\CB(\de)+\ba)\bigr)\ll (\de^2\rho^{d-3}+\de\rho^{-d}).
\ene
If $d=2$, similar estimate holds with $\de^{3/2}+\de\rho^{-2}$ in the RHS.
\enl
\bep
After making the substitution $\bnu=\BF\bxi$, 
the function $g$
in new coordinates will have the form $h(\bnu)=|\bnu|^2+G(\bnu)$, with
\bee\label{G}
G(\bnu)=O(|\bnu|^{-1/2})
\ene
and
\bee\label{partialG}
\frac{\partial G}{\partial\nu_j}\le C_2|\bnu|^{-1}
\ene
for all $j=1,\dots,d$,
provided $\bnu\in\BF\tilde\CB$; these estimates follow from lemma \ref{eigenvalues1}.
We need to estimate the volume of the set
\bee
\bes
\CX=\{\bnu\in\bigl(\BF(\CB)\cap(\BF(\CB)+\BF\ba)\bigr),\,& h(\bnu)\in [\rho^2-\de,\rho^2+\de],\\ &h(\bnu-\BF\ba)\in [\rho^2-\de,\rho^2+\de]\}.
\end{split}
\ene
Indeed, we have $\CX=\BF\Bigl(\CB(\de)\cap \bigl(\CB(\de)+\ba\bigr)\Bigr)$, so
the volume of $\CX$ equals $\det\BF$ times the volume of the set
$\CB(\de)\cap \bigl(\CB(\de)+\ba\bigr)$.
Denote $\bb:=\BF\ba$.
First, we will estimate the $2$-dimensional area of the intersection of $\CX$ with arbitrary
$2$-dimensional plane containing the origin and vector $\bb$; the volume of $\CX$ then can
be obtained using the integration in cylindrical coordinates. So, let $\GV$ be any
$2$-dimensional plane containing the origin and $\bb$, and let us estimate
the area of $\CX_\GV:=\GV\cap\CX$.
Let us introduce cartesian coordinates in $\GV$ so that
$\bnu\in\GV$ has coordinates $(\nu_1,\nu_2)$
with $\nu_1$ going along $\bb$, and $\nu_2$ being orthogonal to $\bb$.
For any $\bnu\in\CX_\GV$ estimate \eqref{G} implies
\bees
h(\bnu)=\nu_1^2+\nu_2^2+O(\rho^{-1/2}),
\enes
and so
\bees
2\de\ge |h(\bnu)-h(\bnu-\bb)|=|\nu_1^2-(|\bb|-\nu_1)^2|+O(\rho^{-1/2}).
\enes
This implies that
\bee\label{wherenu1is}
\frac{|\bb|}{3}<\nu_1<\frac{2|\bb|}{3}
\ene
when $\rho$ is sufficiently
large, and therefore
\bee\label{partialg1}
\frac{\partial h(\bnu)}{\partial\nu_1}\gg |\bb|
\ene
whenever $\bnu\in\CX_\GV$.
Thus, for any fixed $t\in\R$,
the intersection of the line $\nu_2=t$ with $\CX_\GV$ is an interval
of length $\ll |\bb|^{-1}\de$.

Let us cut $\CX_\GV$ into two parts: $\CX_\GV=\CX_\GV^1\cup\CX_\GV^2$ with
$\CX_\GV^1:=\{\bnu\in\CX_\GV,\,|\nu_2|\le 2C_2\rho^{-1}\}$,
$\CX_\GV^2=\CX_\GV\setminus\CX_\GV^1$, and estimate the volumes of these sets ($C_2$ is the
constant from \eqref{partialG}). A simple geometrical argument shows that if $\CX_\GV^1$
is nonempty, then $|\bb|\gg\rho$. This, together with the remark after \eqref{partialg1},
implies that the area of $\CX_\GV^1$ is $\ll \rho^{-2}\de$.
Now we define the `rotated' set $\CX^1$ which consists of the points from $\CX$
which belong to $\CX_\GV^1$ for some $\GV$. Computing the volume of this set using
integration in the cylindrical coordinates, we obtain
\bee\label{CX1}
\volume(\CX^1)\ll \rho^{-d}\de.
\ene

Now consider $\CX_\GV^2$. Let us decompose $\CX_\GV^2=\overline{\CX_\GV^2}\cup\underline{\CX_\GV^2}$, where
\bees
\overline{\CX_\GV^2}=\{\bnu\in\CX_\GV^2:\nu_2>0\}
\enes
and
\bees
\underline{\CX_\GV^2}=\{\bnu\in\CX_\GV^2:\nu_2<0\}.
\enes
Notice that for any $\bnu\in\overline\CX_\GV^2$, formula \eqref{partialG} implies
\bee\label{partialg2}
\frac{\partial h(\bnu)}{\partial\nu_2}\gg \nu_2.
\ene
Let $\bnu^l=(\nu^l_1,\nu^l_2)$ be the point in
the closure of $\overline{\CX_\GV^2}$
with the smallest value of the first coordinate:
$\nu^l_1\le\nu_1$ for any $\bnu=(\nu_1,\nu_2)\in\overline{\CX_\GV^2}$.
Analogously, we define $\bnu^r$
to be the point in the closure of $\overline{\CX_\GV^2}$
with the largest first coordinate,
$\bnu^t$ the point with the largest second coordinate, and $\bnu^b$ the point
with the smallest second coordinate (see Figure \ref{fig:9} for an illustration). Note that $\nu^t\ll\rho$.
\begin{figure}[!hbt]
\begin{center}
\framebox[0.8\textwidth]{\includegraphics[width=0.60\textwidth]{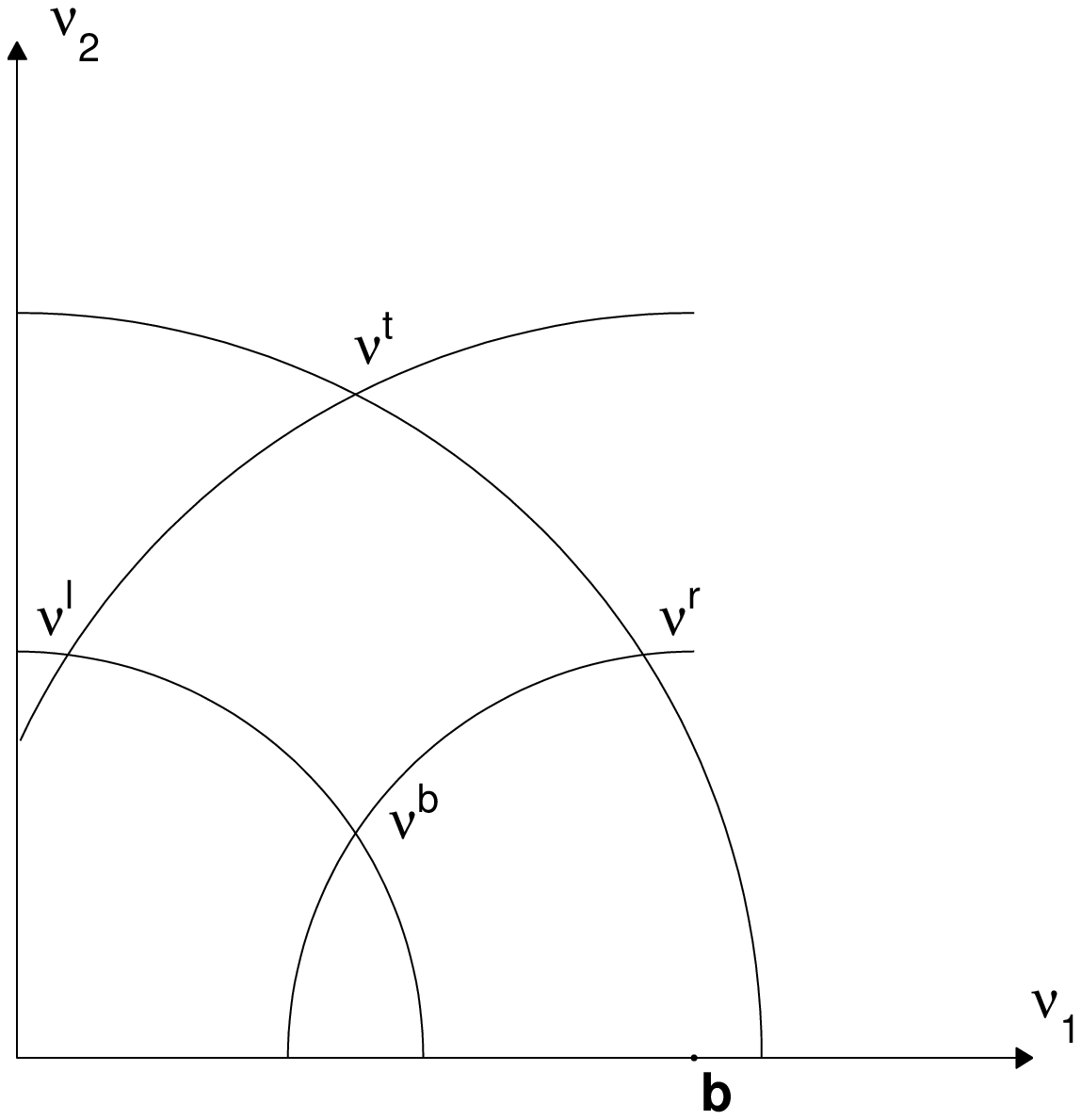}}
\caption{The set $\overline{\CX_\GV^2}$ (the area bounded by four arcs)\label{fig:9}}
\end{center}
\end{figure}

Let us prove that
\bee\label{width0}
\nu_1^r-\nu^l_1\ll \de.
\ene

Indeed, suppose first that $\nu^r_2\ge \nu^l_2$. Let $\bnu^{rl}:=(\nu^r_1,\nu^l_2)$. Then,
since $h$ is an increasing function of $\nu_2$ when $\nu_2>2C_2\rho^{-1}$, we have
$h(\bnu^{rl})\le h(\bnu^r)\le \rho^2+\de$. Therefore,
$h(\bnu^{rl})-h(\bnu^l)\le 2\de$. Estimate \eqref{partialg1} then
implies \eqref{width0}.

Suppose now that $\nu^r_2\le \nu^l_2$. Let $\bnu^{lr}:=(\nu^l_1,\nu^r_2)$. Then
$h(\bnu^{lr}-\bb)\le h(\bnu^l-\bb)\le \rho^2+\de$.
Therefore,
$h(\bnu^{lr}-\bb)-h(\bnu^r-\bb)\le 2\de$. Now, \eqref{wherenu1is} and
\eqref{partialg1} imply \eqref{width0}.

Thus, we have estimated the width of $\CX_\GV^2$. Let us estimate its hight
(i.e. $\nu_2^t-\nu^b_2$). Let us assume that $\nu^t_1\ge\nu^b_1$; otherwise, we use the same
trick as in the previous paragraph and consider $h(\cdot-\bb)$ instead of $h$.
Let $\bnu^{bt}:=(\nu^b_1,\nu^t_2)$. Then
$h(\bnu^{bt})\le h(\bnu^t)\le \rho^2+\de$. Therefore,
$h(\bnu^{bt})-h(\bnu^b)\le 2\de$. Now, \eqref{partialg2} implies
\bee\label{hightd2}
(\nu_2^t)^2-(\nu_2^b)^2=2\int_{\nu_2^b}^{\nu_2^t}\nu_2d\nu_2
\ll \int_{\nu_2^b}^{\nu_2^t}\frac{\partial h}{\partial\nu_2}(\nu_1^b,\nu_2)d\nu_2
\le 2\de.
\ene
Therefore, we have the following estimate for the hight of $\CX_\GV^2$:
\bee\label{hight}
\nu_2^t-\nu_2^b\ll \frac{\de}{\nu_2^t+\nu_2^b}.
\ene
Now, we can estimate the volume of $\CX^2:=\CX\setminus\CX^1$ using estimates \eqref{width0}
and \eqref{hight}. The cylindrical integration produces the following:
\bee\label{CX2}
\volume(\CX^2)\ll \frac{\de^2}{\nu_2^t+\nu_2^b}(\nu_2^t)^{d-2}\le
\de^2 (\nu_2^t)^{d-3}\le \de^2\rho^{d-3}.
\ene
Equations \eqref{CX1} and \eqref{CX2} imply \eqref{eq:volumeCB1}. If $d=2$, we have
to notice that \eqref{hightd2} implies $\nu_2^t-\nu_2^b\ll \de^{1/2}$ and then use \eqref{CX1}
and \eqref{width0}.
\enp

As was mentioned already, the function $f$
is not necessarily continuous. We now give a sufficient condition for its continuity.
Recall that $v$ is the $L_{\infty}$-norm of the potential $V$.

\bel\label{continuousf}
Let $\bxi\in\CB(v)$ be a point of discontinuity of $f$. Then there is a non-zero
vector $\bn\in\Z^d$ such that
\bee\label{eq:continuousf}
|g(\bxi+\bn)-g(\bxi)|\le 2\rho^{-N}.
\ene
\enl
\bep
If $\bxi=\bm+\bk\in\CB(v)$ is a point of discontinuity of a bounded function $f$,
there exist two sequences $\{\bxi_j\}$
and $\{\tilde \bxi_j\}$ which both converge to $\bxi$, such that the limits
$\la(\bxi):=\lim f(\bxi_j)$
and $\tilde \la(\bxi):=\lim f(\tilde \bxi_j)$ exist and are different. Since the points
$f(\bxi_j)$ are eigenvalues of $H(\{\bxi_j\})$, the limit $\la$ is an eigenvalue of
$H(\bk)$ (it is well-known that the spectrum of $H(\bk)$ is continuously dependent on $\bk$).
The same argument implies that $\tilde\la$ is also an eigenvalue of $H(\bk)$. Since
$\la\ne\tilde \la$, at most one of these points can be equal to $f(\bxi)$.
Say, $\tilde\la\ne f(\bxi)$. But since $\tilde \la$ is inside $J$, it must belong to the
image of $f$, say $\tilde\la=f(\tilde\bxi)$, $\{\tilde\bxi\}=\{\bxi\}$. Thus,
$\tilde\bxi=\bxi+\bn$ with $\bn\in\Z^d$. Since the function $g$ is continuous in $\tilde\CB$,
$\lim g(\tilde\bxi_j)=g(\bxi)$, and so 
\bees
|g(\bxi)-\tilde\la|=\lim |g(\tilde\bxi_j)-f(\tilde \bxi_j)|\le \rho^{-N}.
\enes
But we also have
$|g(\tilde\bxi)-\tilde\la|=|g(\tilde\bxi)-f(\tilde\bxi)|\le \rho^{-N}$. The last two
inequalities imply \eqref{eq:continuousf}.
\enp
\bec\label{cor:continuousf} There is a constant $C_3$ with the following properties.
Let
\bees
I:=\{\bxi(t):\, t\in[t_{min},t_{max}]\}\subset \CB(v).
\enes
be a straight interval of length $L<\rho^{-1}\de$.
Suppose that there is a
point $t_0\in[t_{min},t_{max}]$ with the property that for each non-zero
$\bn\in\Z^d$ $g(\bxi(t_0)+\bn)$ is either outside the interval
\bees
[g(\bxi(t_0))-C_3\rho^{-N}-C_3\rho L,g(\bxi(t_0))+C_3\rho^{-N}+C_3\rho L]
\enes
or not defined. Then $f(\bxi(t))$ is a continuous function of $t$.
\enc
\bep
Suppose not. Then previous lemma implies that there is a point
$t_1\in[t_{min},t_{max}]$ and a non-zero vector $\bn\in\Z^d$
such that $|g(\bxi(t_1)+\bn)-g(\bxi(t_1))|\le 2\rho^{-N}$.
Since $|\bxi(t_1)-\bxi(t_0)|\le|\bxi(t_{max})-\bxi(t_{min})|\le L$, it follows that $(I+\bn)\subset \CA_1$, and now lemma \ref{partialofs2} implies that for two
different integer vectors $\bm_1$ and $\bm_2$ we have
$|g(\bxi(t_0)+\bm_1)-g(\bxi(t_1)+\bn)|\ll \rho L+\rho^{-N}$ and
$|g(\bxi(t_0)+\bm_2)-g(\bxi(t_1))|\ll \rho L+\rho^{-N}$.
Since $\bxi(t)\in\CB$ for all $t$ and the length of the gradient of $g$ is $\ll\rho$ in $\CB$,
we also have $|g(\bxi(t_1))-g(\bxi(t_0))|\ll \rho L$. Thus, we have
$|g(\bxi(t_0)+\bm_j)-g(\bxi(t_0))|\le C\rho^{-N}+C \rho L$ ($j=1,2$). Since at least one of
vectors $\bm_j$ is non-zero, this
contradicts the assumption of the corollary.
\enp
Now we are ready to prove the Bethe-Sommerfeld conjecture. Since in the two-dimensional
case it has been proved, we will assume that $d\ge 3$.
\bet\label{maintheorem3} Let $d\ge 3$.
Then all sufficiently large points $\la=\rho^2$ are
inside the spectrum of $H$. Moreover, there exists a positive constant
$c_4$ such that
for large enough $\rho$ the whole interval $[\rho^2-c_4\rho^{1-d},\rho^2+c_4\rho^{1-d}]$ lies
inside some spectral band.
\ent
\bep
Put $N=d$ in the corollary \ref{maincorollary1}.
Also put $\de=c_4\rho^{1-d}$
(the precise value of $c_4$ will be chosen later). For each unit vector $\boldeta\in\R^d$
we denote $I_{\boldeta}$ to be the intersection of $\{r\boldeta,\,r>0\}$ with $\CA(\de)$.
We will consider only vectors $\boldeta$ for which $I_{\boldeta}\subset\tilde\CB$.
As was mentioned in the proof of lemma \ref{volumeCA5},
the length $L$ of any interval $I_{\boldeta}$ satisfies
$L\asymp \de\rho^{-1}$. 
Let us prove
that $f$ is continuous on at least one of the intervals $I_{\boldeta}\subset\tilde\CB$.
Suppose this is not the case. Then corollary \ref{cor:continuousf} tells us
that for each point
$\bxi\in\CB(\de)$ there is a non-zero integer vector $\bn$ such that
\bee
|g(\bxi+\bn)-g(\bxi)|\le C_3(\rho^{-d}+\rho L)\ll(\rho^{-d}+\de).
\ene
Since $|g(\bxi)-\rho^2|\le\de$, this implies $|g(\bxi+\bn)-\rho^2|\le C_5(\rho^{-d}+\de)=:\de_1$,
and thus $\bxi+\bn\in\CA(\de_1)$; notice that $C_5>1$ and so $\de_1>\de$.
Therefore, each point $\bxi\in\CB(\de)$
also belongs to the set $\bigl(\CA(\de_1)-\bn\bigr)$
for a non-zero integer $\bn$; obviously,
$|\bn|\ll\rho$. In other words,
\bee\label{cover1}
\bes
\CB(\de) &\subset \bigcup_{\bn\in\Z^d\cap B(C\rho),\bn\ne 0}
\bigl(\CA(\de_1)-\bn\bigr)\\
&=\bigcup_{\bn\ne 0}\bigl(\CB(\de_1)-\bn\bigr)\cup
\bigcup_{\bn\ne 0}\bigl(\CD(\de_1)-\bn\bigr)
\end{split}
\ene
To proceed further, we need more notation. Denote $\CD_0(\de_1)$ to be the set
of all points $\bnu$ from $\CD(\de_1)$ for which there is no non-zero $\bn\in\Z^d$
satisfying $\bnu-\bn\in\CB(\de)$;  $\CD_1(\de_1)$ to be the set
of all points $\bnu$ from $\CD(\de_1)$ for which there is a unique non-zero $\bn\in\Z^d$
satisfying $\bnu-\bn\in\CB(\de)$; and $\CD_2(\de_1)$ to be the rest of the points from
$\CD(\de_1)$ (i.e. $\CD_2(\de_1)$ consists
of all points $\bnu$ from $\CD(\de_1)$ for which there exist at least two different   non-zero vectors $\bn_1,\bn_2\in\Z^d$ satisfying $\bnu-\bn_j\in\CB(\de)$).
Then a little thought shows that we can replace
$\CD(\de_1)$ by $\CD_1(\de_1)$ in the RHS of \eqref{cover1}.
Indeed, this is shown in the following lemma.
\bel\label{CD02}
The following formulae hold:
\bee
\CB(\de)\bigcap\Bigl(\bigcup_{\bn\ne 0}
\bigl(\CD_0(\de_1)-\bn\bigr)\Bigr)=\emptyset,
\ene
and
\bee
\Bigl(\bigcup_{\bn\ne 0}
\bigl(\CD_2(\de_1)-\bn\bigr)\Bigr)\subset \bigcup_{\bn\ne 0}\bigl(\CB(\de_1)-\bn\bigr)
\ene
\enl
\bep
The first formula is an immediate corollary of the definition of $\CD_0(\de_1)$.
Let us prove the second formula.
Suppose, $\bnu\in\CD_2(\de_1)$. Then there exist two
integer vectors, $\bn_1$ and $\bn_2$ such that $\bnu-\bn_j\in\CB(\de)$.
Let $\bm$ be an integer vector. Then $\bm$ is different
from either $\bn_1$ or $\bn_2$, say $\bm\ne\bn_1$.
Since $\de_1\ge \de$, this implies:
\bees
\bnu-\bm=\bnu-\bn_1-(\bm-\bn_1)\in\bigl(\CB(\de)-(\bm-\bn_1)\bigr)\subset
\bigcup_{\bn\ne 0}\bigl(\CB(\de_1)-\bn\bigr).
\enes
This finishes the proof of the lemma.
\enp
This lemma shows that we can re-write \eqref{cover1} as
\bee\label{cover2}
\CB(\de)
\subset\bigcup_{\bn\ne 0}\bigl(\CB(\de_1)-\bn\bigr)\bigcup
\bigcup_{\bn\ne 0}\bigl(\CD_1(\de_1)-\bn\bigr).
\ene
This, obviously, implies
\bee\label{cover3}
\bes
\CB(\de)
=&\bigcup_{\bn\ne 0}\Bigl(\bigl(\CB(\de_1)-\bn\bigr)\cap\CB(\de)\Bigr)\\
\bigcup&
\bigcup_{\bn\ne 0}\Bigl(\bigl(\CD_1(\de_1)-\bn\bigr)\cap\CB(\de)\Bigr).
\end{split}
\ene
Now let us compare volumes of the sets in both sides of \eqref{cover3}. The volume of the
LHS we already know from \eqref{volumeCB}: it is $\gg\rho^{d-2}\de$.
The definition of the set $\CD_1$ implies that
\bee\label{volumecover1}
\bes
&\volume\Bigl(\bigcup_{\bn}\Bigl(\bigl(\CD_1(\de_1)-\bn\bigr)
\cap\CB(\de)\Bigr)\Bigr)\\&\le \volume(\CD_1(\de_1))
\ll  \rho^{(3d-7)/3}\de_1\ll\rho^{(3d-7)/3}(\rho^{-d}+\de).
\end{split}
\ene
Finally, lemma \ref{volumeCB1}, inequality $\de<\de_1$ and the fact that the union in \eqref{cover3} consists of no more than $C\rho^d$ terms imply
\bee\label{volumecover2}
\bes
&\volume\Bigl(\bigcup_{\bn}\Bigl(\bigl(\CB(\de_1)-\bn\bigr)
\cap\CB(\de)\Bigr)\Bigr)\ll\rho^d(\de_1^2\rho^{d-3}+\de_1\rho^{-d})\\
&\ll \rho^d\bigl((\rho^{-d}+\de)^2\rho^{d-3}+
(\rho^{-d}+\de)\rho^{-d}\bigr)\ll \de^2\rho^{2d-3}+\de\rho^{d-3}+\rho^{-3}.
\end{split}
\ene
Putting all these inequalities together, we get
\bee\label{volumecover3}
\rho^{d-2}\de<C_6(\de^2\rho^{2d-3}+\de\rho^{(3d-7)/3}+\rho^{-7/3}).
\ene
It is time to recall that $\de=c_4\rho^{1-d}$. Plugging this into
\eqref{volumecover3}, we obtain
\bee\label{volumecover4}
c_4\rho^{-1}<C_6(c_4^2\rho^{-1}+c_4\rho^{-4/3}+\rho^{-7/3}).
\ene
Now, if we choose $c_4$ to be small enough (i.e. $c_4<C_6^{-1}$), the inequality
\eqref{volumecover4} will not be satisfied for sufficiently large $\rho$. Thus, our assumption that
function $f$ is discontinuous on every interval $I_{\boldeta}\subset\CB(\de)$
leads to a contradiction
(provided we have chosen small enough $c_4$). Therefore, there is an interval
$I_{\boldeta}\subset\CB(\de)$
on which $f$ is continuous. Since the value of $f$ on one end of this interval is
$\le \rho^2-c_4\rho^{1-d}$, and the value on the other end is $\ge \rho^2+c_4\rho^{1-d}$,
the point $\rho^2$ must be in the range of $f$. The first part of the theorem is proved.
In order to prove the second part of the theorem, we notice that the interval $I_{\boldeta}$
which we found satisfies the following condition: for each point $\bxi\in I_{\boldeta}$ and
each non-zero integer vector $\bn$ such that $\bxi+\bn\in\CA_1$ we have
$|g(\bxi+\bn)-g(\bxi)|>2\rho^{-N}$. This implies $f(\bxi+\bn)-f(\bxi)\ne 0$. Therefore,
$f(\bxi)$ is a simple eigenvalue of $H(\{\bxi\})$ for each $\bxi\in I_{\boldeta}$.
This implies that the interval
$[\rho^2-c_4\rho^{1-d},\rho^2+c_4\rho^{1-d}]$ is inside the spectral band. The theorem
is proved.
\enp

\bibliographystyle{amsplain}

\providecommand{\bysame} {\leavevmode\hbox
to3em{\hrulefill}\thinspace}

\end{document}